\documentclass[a4paper,12pt, twoside, onecolumn]{report}
\usepackage[osf,sc]{mathpazo}
\usepackage{amsmath,amsthm,amsfonts,latexsym,amsopn,verbatim,
amscd,amssymb}
\usepackage{hyperref}

\usepackage[osf,sc]{mathpazo}
\usepackage{amsfonts}
\usepackage{amsmath, amsthm, amssymb, enumerate, tikz, hyperref}

\usepackage{verbatim}
\usepackage{graphicx}
\usepackage{color}
\usepackage{epsfig,bm,epsf,float}\usepackage{fancyhdr}\usepackage{setspace}
\usepackage[top=30mm, bottom = 22mm, left = 30mm, right = 20mm]{geometry}

\newtheorem{proposition}{Proposition}[section]

\newtheorem{definition}{Definition}[section]
\newtheorem{example}{Example}[section]
\newtheorem{theorem}{Theorem}[section]
\newtheorem{lemma}{Lemma}[section]
\newtheorem{corollary}{Corollary}[section]
\newtheorem{remark}{Remark}[section]

\numberwithin{equation}{section} 
\numberwithin{table}{section}
\usepackage{makeidx}

\begin{document}
	\thispagestyle{empty}
	\begin{center}
		\begin{Large}
			\textsc
			{\bf{ALGORITHMIC CONSTRUCTION OF REPRESENTATIONS OF FINITE SOLVABLE GROUPS}}\\ 
		\end{Large}
		\begin{center}
			\large{Thesis}\\
			\large{{Submitted in Partial Fulfillment of the Requirements}}\\
			\large{{ for the Degree of }}\\
			\large{\bf \large{DOCTOR OF PHILOSOPHY}}\\

		\end{center}
	\end{center}
	\begin{center}
		{\bf \large{by}}\\
		\vspace{.5cm}
		\begin{Large}
			\textsc{\bf SOHAM SWADHIN PRADHAN}\\
			\textsc{(07409305)}
		\end{Large}
	\end{center}
	\begin{center}
		\begin{Large}
			{\bf Supervisor: Prof. Murali K. Srinivasan}
		\end{Large}
		\\
		\vspace{.2cm}
		\begin{Large}
			{\bf Co-supervisor: Prof. Ravi S. Kulkarni}
		\end{Large}
	\end{center}
	\vspace {.2cm}
	\begin{center}
		\begin{Large}
			\textsc{\bf{DEPARTMENT OF MATHEMATICS\\
					\bf{INDIAN INSTITUTE OF TECHNOLOGY BOMBAY}\\
					2018\\[1ex]}}
		\end{Large}
	\end{center}	
	\pagenumbering{roman}
 	
\cleardoublepage      	 
	\thispagestyle{empty}
	\vspace*{8cm}
	\begin{center}
		{\Large \it 
			Dedicated To \\
			\vspace*{0.5cm}
			\it My Teachers }
	\end{center}
\cleardoublepage
\newpage
\thispagestyle{empty}
\cleardoublepage
\thispagestyle{plain}
	
	\newpage
	\thispagestyle{empty}
	\cleardoublepage
	
	\thispagestyle{plain} 
	\begin{center}
		{\bf \large Declaration}
	\end{center}
	I declare that this written submission represents my ideas in my own words and where others' ideas or words have
	been included, I have adequately cited and referenced the original sources.  I also declare that I have adhered to
	all principles of academic honesty and integrity, and have not misrepresented or fabricated or falsified any
	idea/data/fact/source in my submission. I understand that any violation of the above will be cause for
	disciplinary action by the Institute and can also evoke penal action from the sources which have thus not been
	properly cited or from whom proper permission has not been taken when needed.
	
	\vspace*{3cm}
	Place: IIT Bombay\hfill \hspace{.7cm} Soham Swadhin Pradhan \hspace{1.5cm} \par 
	\hfill Roll No. 07409305\hspace*{2.3cm} \par
	\vspace*{1cm}
	\noindent
	\newpage
	\thispagestyle{empty}
	\cleardoublepage
\chapter*{Abstract}
	\addcontentsline{toc}{chapter}{Abstract} 
	\markboth{Abstract}{Abstract}
	The dominant theme of this thesis is the construction of matrix representations of finite solvable groups using a suitable system of generators. For a finite group $G$, generated by $n$ elements $x_{1}, \dots , x_{n}$, say, to know an $F$-representation $(\rho, V)$, we need to know only $\rho(x_{i})$'s. If $V$ is finite dimensional, for the matrix representation of $\rho$, we choose a basis of $V$, and w.r.t. that basis we need to find matrices for $\rho(x_{i})$'s. If $F$ is a field of characteristic $0$ or prime to $|G|$ and $(\rho, V)$ is an irreducible $F$-representation of $G$, then $(\rho, V)$ lies as a minimal left-sided ideal in the semisimple group ring $F[G]$. So for the matrix representation of $\rho$, we find a basis of $V$ as elements of $F[G]$, and w.r.t. that basis we find the matrices for $\rho(x_{i})$'s. We follow this procedure for constructing irreducible matrix representations of finite solvable groups over $\mathbb{C}$, the field of complex numbers. 
	
	For a finite solvable group $G$ of order $N = p_{1}p_{2}\dots p_{n}$, where $p_{i}$'s are primes, there always exists a subnormal series: $\langle {e} \rangle = G_{o} < G_{1} < \dots < G_{n} = G$ such that $G_{i}/G_{i-1}$ is isomorphic to a cyclic group of order $p_{i}$, $i = 1,2,\dots,n$. Associated with this series, there exists a system of generators consisting $n$ elements $x_{1}, x_{2}, \dots , x_{n}$ (say), such that $G_{i} = \langle x_{1}, x_{2}, \dots , x_{i} \rangle$, $i = 1,2,\dots,n$, which is called a ``long system of generators". In terms of this system of generators and conjugacy class sum of $x_{i}$ in $G_{i}$, $i = 1,2, \dots, n$, we present an algorithm for constructing the irreducible matrix representations of $G$ over $\mathbb{C}$ within the group algebra $\mathbb{C}[G]$. This algorithmic construction needs the knowledge of primitive central idempotents, a well defined set of primitive (not necessarily central) idempotents and the ``diagonal subalgebra" of $\mathbb{C}[G]$. In terms of this system of generators, we give simple expressions for the primitive central idempotents, a well defined system of primitive (not necessarily central) idempotents  and a convenient set of generators of the ``diagonal subalgebra" of $\mathbb{C}[G]$.
	
	For a finite abelian group, we present an algorithm for constructing the inequivalent irreducible matrix representations over a field of characteristic $0$ or prime to the order of the group and a systematic way of computing the primitive central idempotents of the group algebra. Besides that, we give simple expressions of the primitive central idempotents of the rational group algebra of a finite abelian group using a ``long presentation" and its Wedderburn decomposition. 
\newpage
\thispagestyle{empty}
\cleardoublepage
\tableofcontents
\cleardoublepage
\newpage
\setcounter{page}{1}
\pagenumbering{arabic}
\chapter{Introduction}
\thispagestyle{empty}
The thesis mainly consists of the algorithmic construction of matrix representations of finite solvable groups and the computation of primitive central idempotents of semisimple group algebras. 
Two basic problems in ordinary representation theory of finite groups are 
\begin{itemize}
\item [1.] construction of the inequivalent irreducible matrix representations of a finite group over a field of characteristic $0$ or prime to order of the group; 
\item[2.] computation of the primitive central idempotents of a semisimple group algebra.
\end{itemize}
The theme of this thesis work is to solve these two problems for certain classes of groups over certain fields. This dissertation consists of the following works:
\begin{itemize}
\item[1.] Let $G$ be a finite group and $F$ be a field of characteristic $0$ or prime to $|G|$. We show that $F$-conjugacy of an element of order $n$ in $G$ is the union of certain conjugacy classes, and which can be determined by the decomposition of $n$-th cyclotomic polynomial $\Phi_{n}(X)$ into irreducible polynomials over $F$. We give a formula for computing the primitive central idempotents of $F[G]$, in terms of "$F$-conjugacy classes", which can be obtained from the "$F$-character table" of $G$. 
\item[2.] The classical approach to complex representations of finite solvable groups is based on Clifford theory (see \cite{alg111}, \cite{cl50}, \cite{combi91}). We study this approach in details with various examples and note some important observations.  
\item[3.] We present an algorithm for constructing the inequivalent irreducible matrix representations of a finite solvable group over the field of complex numbers using a "long presentation". In terms of "system of long generators", we obtain simple expressions for the primitive central idempotents, a well-defined system of primitive (not necessarily central) idempotents and a convenient set of generators of the "diagonal subalgebra" of the complex group algebra of a finite solvable group. 
\item[4.] For a finite abelian group $G$ and a field $F$ of characteristic $0$ or prime to $|G|$, we present an algorithm for constructing the irreducible matrix representations of $G$ over ${F}$ and present a systematic way of computing the primitive central idempotents of $F[G]$. Besides that, we give simple expressions of the primitive central idempotents in $\mathbb{Q}[G]$ using a "long presentation" of $G$ and find its Wedderburn decomposition.
\end{itemize}
\section{Definitions}
We begin with some basic definitions, which can be found in \cite{cur93}, \cite{Weintraub}.
\subsection{Representations}
\begin{definition}\rm
Let $G$ be a finite group and $F$ be a field. Let $V$ be a finite dimensional vector space over $F$. An {\bf{$F$-representation}} $(\rho, V)$ of $G$ is a homomorphism $\rho: G \rightarrow \mathrm{GL}(V)$. The dimension of $V$ is called the {\bf{degree}} of the representation $(\rho, V)$, and is denoted by deg$\rho$.
\end{definition}
\begin{remark}
If $\rho: G \rightarrow \mathrm{GL}(V)$ is an $F$-representation of $G$, then to understand $\rho$ we do not need to know $\rho(x)$ for every $x$ in $G$. If $G$ has a presentation
\begin{equation*}
G = \langle \,\, x_1, x_2, ..., x_n \, | \, \cdots \,\, \rangle,
\end{equation*}
where $x_i'$s are certain generators, satisfying certain relations, then every element of $G$ can be expressed as a word in $x_1, x_2, \dots , x_{n}$, not necessarily in a unique way. Then to know $\rho$ we need to know only $\rho(x_i)'s$.
\end{remark}
\subsection{Long presentations}

\begin{definition}\rm
A {\bf subnormal series} of a group $G$ is a chain of subgroups 
\begin{center}
			$\langle{e}\rangle = G_0 < G_1 < \dots < G_n = G$
\end{center}
such that, for all $i$, $ G_{i+1}$ is normal in  $G_i$. 
\end{definition}
\begin{definition}\rm
A finite group $G$ is {\bf{solvable}} if it has a subnormal series
\begin{center}
$\langle{e}\rangle = G_0 < G_1 < \dots < G_n = G$
\end{center}
such that each factor group $G_{i+1}/G_i$ is a cyclic group of prime order.
\end{definition}

\begin{definition}\rm$\label{long1}$
		Let $G$ be a finite solvable group. 
		Then there is a subnormal series $\langle{e}\rangle = G_0 < G_1 < \dots < G_n = G$ such that for $i = 1, 2, \dots , n$, we have $G_i/G_{i-1}$ is isomorphic to 
		a cyclic group of prime order $p_i$ (say). 
		Suppose that $x_i$ is an element in $G_i$ such that $x_iG_{i-1}$ is a generator of $G_i/G_{i-1}.$ So $G$ has a presentation
		\begin{align*}
		\langle \,x_1, x_2, \dots , x_n \, |\, & x_i^{p_i} = w_i(x_1, x_2, ..., x_{i-1}), \,x_i^{-1}x_jx_i = w_{ij}(x_1, x_2,..., x_{i-1}) \mbox{ for } j<i \,\rangle,
		\end{align*}
		where $w_i$ and $w_{ij}$ are certain words in $x_1, x_2, ..., x_{i-1}$. 
		We call such a presentation, a 
		{\bf{long presentation}} of $G$ and $\{x_1, x_2, \dots , x_n\}$ is called {{\bf long system of generators}} or simply {\bf{long generators}}. 

\end{definition}
	\begin{remark}
		Each element of $G$ can be expressed uniquely as $x_1^{a_1}x_2^{a_2}...x_n^{a_n}$, where  
		$0 \le a_i < p_i$.
	\end{remark}
	\subsection{Primitive central idempotents}
	
	Let $G$ be a finite group and $F$ be a field. We consider all formal sums $\sum_{g \in G}{\alpha_{g}g}, \,\,\,\, \alpha_{g} \in F.$
	We define operations on the formal sums by the rules	
	$$\sum_{g \in G}{\alpha_{g}g} + \sum_{g \in G}{\beta_{g}g} = \sum_{g \in G}{(\alpha_{g} + \beta_{g})g}$$
	and 
	$$(\sum_{g \in G}{\alpha_{g}g})(\sum_{h \in G}{\alpha_{h}h}) = \sum_{g, h \in G}{\alpha_{g}\beta_{h}{gh}} = \sum_{t \in G}{\gamma_{t}t},\,
	\mathrm{where}\, \gamma_{t} = \sum_{g \in G}{\alpha_{g}{\beta_{g^{-1}t}}}.$$
	We define 
	$\alpha{\sum_{g \in G}{\alpha_{g}g}} = {\sum_{g \in G}{\alpha\alpha_{g}g}}, \, \alpha \in F$.
	With these definitions, it can be shown that the set of all formal sums forms an $F$-algebra, which is denoted by $F[G]$ and is called the {\bf{group algebra}} of $G$ over $F$. 
	
	Artin-Wedderburn structure theorem on semisimple algebras (see \cite{Weintraub}, Theorem $2.16$) states that a semisimple group algebra $F[G]$ is the direct sum of its minimal two sided ideals $A_{1}, A_{2}, \dots ,$ $ A_{k}$ (say), where each $A_{i}$ is isomorphic to a matrix ring over finite dimensional division algebra over $F$. Moreover, each $A_{i}$ is equal to $F[G]e_{i}$, where $e_{i}$, $1 \leq i \leq k$ satisfy the following properties:
	\begin{itemize}
		\item[(i)] $e_{i}$ is an { \bf idempotent}, i.e., $e^{2}_{i} = e_{i}$;
		\item[(ii)] $e_{i}$ is a {\bf central} element of $F[G]$, i.e., $e_{i}$ belongs to the center of $F[G]$;
		\item[(iii)] $e_{i}$ and $e_{j}$ are { \bf orthogonal}, i.e., $e_{i}e_{j} = 0,\, \mathrm{for}\, i \neq j$;
		\item[(iv)] $1 = e_{1} + e_{2} + \cdots + e_{k}$;
		\item[(v)] $e_{i}$ can not be written as $e_{i} = e^{'}_{i} + e^{''}_{i}$, where $e^{'}_{i}$ and $e^{''}_{i}$ are orthogonal central idempotents in $F[G]$. 
	\end{itemize}
	The set of elements of $F[G]$ satisfying $(i)-(v)$ is called {\bf{the complete set of primitive central idempotents}} of $F[G]$.
	\subsection{PCI-diagrams}
	Let $G$ be a finite group and $H$ be a subgroup of $G$. Let $F$ be a field with characteristic $0$ or prime to $|G|$. Let $\Omega_{G, F}$ and $\Omega_{H, F}$ denote the set of inequivalent irreducible $F$-representations of $G$ and $H$ respectively. 
	If $(\rho, V)$ is a representation of $G$, then its restriction to $H$, denoted by
	$\rho\downarrow^G_H$ or Res$(V)\downarrow^{G}_{H}$, is a representation of $H$. There is an obvious correspondence $\phi: \Omega_{G,F}\rightharpoonup \Omega_{H, F},$ which is associated with an irreducible
	representation $\rho$ of $G$, the irreducible components of $\rho\downarrow^G_H$. 
	
	Let
	\begin{equation*}
	\rho\downarrow^G_H = \sum_i m_i\eta_i,
	\end{equation*}
	where $m_i$ is a natural number and $\eta_i \in \Omega_{H, F}$ is the decomposition of $\rho\downarrow^G_H$ into irreducible representations of $H$ over $F$. Then draw $m_i$ edges from $\rho$ to $\eta_i$. We call this bipartite graph the {\bf{restriction graph}} of the pair $(G, H)$ and 
	denote it by {\bf $\Gamma(G \downarrow H)$}. 
	Pictorially, the directed restriction graph is:
	
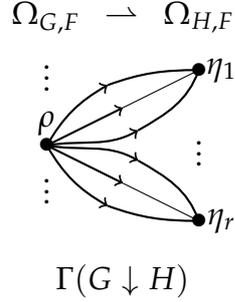
\begin{center}
		\begin{tikzpicture}
		\node at (0,0) {$ $};
		\node at (-2,0) {$\bullet$};
		\node at (-2,2) {$\bullet$};
		\node at (-4,1) {$\bullet$};
	
		\node at (-2,1) {$\vdots$};
	
		\node at (-4,2) {$\vdots$};
		\node at (-4,0.5) {$\vdots$};
	
		\node at (-4,1.3) {$\rho$};
	
		\node at (-1.7,0) {$\eta_r$};
		\node at (-1.7,2) {$\eta_1$};
	
		\draw [thick, ->] (-4,1) -- (-3,1.5); \draw (-3,1.5)--(-2,2);
		\draw [thick, ->] (-4,1) -- (-3,0.5); \draw (-3,0.5)--(-2,0);
	
		\draw [thick, ->] (-4,1) to [out=60, in=210] (-3.2,1.8); 
		\draw [thick] (-3.2,1.8) to [out=20,in=180] (-2,2);
		
		\draw [thick, ->] (-4,1) to [out=10, in=210] (-2.8,1.2); 
		\draw [thick] (-2.8,1.2) to [out=30,in=250] (-2,2);
	
		\draw [thick, ->] (-4,1) to [out=-60, in=-210] (-3.2,0.2); 
		\draw [thick] (-3.2,0.2) to [out=-20,in=-180] (-2,0);
	
		\draw [thick, ->] (-4,1) to [out=-10, in=-210] (-2.8,0.8); 
		\draw [thick] (-2.8,0.8) to [out=-30,in=-250] (-2,0);
	
		\node at (-3,2.7) {$\Omega_{G, F}\,\,\,\rightharpoonup \,\,\,\Omega_{H, F} $};
	
		\node at (-3,-.8) {$\Gamma(G \downarrow H)$};
	
		\end{tikzpicture}
	\end{center}

	There is an adjoint correspondence $\psi: \Omega_{H, F} \rightharpoonup \Omega_{G, F}$. 
	 
	Let $(\eta, W)$  be an irreducible representation of $H$.
	Then $V := F[G]\otimes_{F[H]} W$ is an $F$-vector space on which $G$ starts 
	acting from the left. Let $\{g_1 = 1, g_{2}, \dots ,g_n\}$ be a complete set of left-coset representatives of $H$ in $G$. Then 
	$$V  = (g_1\otimes W) \oplus (g_2\otimes W) \oplus \cdots \oplus (g_n\otimes W) = W_1 \oplus \cdots \oplus W_n.$$
	To define action of $G$ on $V$, it suffices to define action on elements of each component $g_i\otimes W$. 
	Consider $g\in G$ and $g_i\otimes w\in g_i\otimes W$. 
	Now $gg_i = g_jh$ for uniquely determined $j$ ($1\leq j\leq n$) and $h\in H$.
	Then we define $g.(g_i\otimes w)=gg_i\otimes w$. Notice that 
	$$gg_i\otimes w=g_jh\otimes w=g_j\otimes hw\in g_j\otimes W\subseteq V,$$
	and this action is independent of a choice of 
	left-coset representatives $g_i$'s. 
	The corresponding representation is called the   
	{\bf induced representation} 
	induced from $\eta$, and is denoted by  $\eta\uparrow^G_H$ or 
	$\mathrm{Ind}(W)\uparrow^G_H$. Let
	\begin{equation*}
	\eta\uparrow^G_H = \sum_j  \,\,n_j \rho_j,
	\end{equation*}
	where $n_j$ is a natural number and 
	$\rho_j \in \Omega_{G, F}$,
	is the decomposition into irreducible representations of $G$. Then we construct another directed bipartite graph based again on the two sets $\Omega_{G, F}$  and $\Omega_{H,F}$ with $n_j$ 
	edges joining   $\eta$ to $\rho_j$. We call this directed bipartite graph, the {\bf{induction graph}} of the pair $(G, H)$ and denote it by $\Gamma(H\uparrow G)$. Pictorially, the induction graph is:
	\begin{center}
		\begin{tikzpicture}
	
		\node at (1,1) {$\bullet$};
		\node at (3,0) {$\bullet$};
		\node at (3,2) {$\bullet$};
	
		\node at (3,1) {$\vdots$};

		\node at (1,2) {$\vdots$};
		\node at (1,0.5) {$\vdots$};

		\node at (1,1.3) {$\eta$};

		\node at (3.3,0) {$\rho_s$};
		\node at (3.3,2) {$\rho_1$};

		\draw [thick, ->] (1,1) -- (2,1.5); \draw (2,1.5)--(3,2);
		\draw [thick, ->] (1,1) -- (2,0.5); \draw (2,0.5)--(3,0);

		\draw [thick, ->](1,1) to [out=50, in= 210] (1.8, 1.8);
		\draw [thick](1.8, 1.8) to [out = 20, in= 180] (3,2);

		\draw [thick, ->](1,1) to [out=5, in=210] (2.2,1.2);
		\draw [thick, ->](2.2, 1.2) to [out=30, in=250] (3,2);

		\draw [thick, ->](1,1) to [out=120, in= 175] (1.9, .2);
		\draw[thick] (1.9,.2) to [out = 175, in = 180] (3,0);
		\draw [thick, ->](1,1) to [out=0, in=140] (2.2, .8);
		\draw [thick] (2.2, .8) to [out = 140, in = 100] (3,0);

		\node at (2,2.7) {$\Omega_{H, F}\,\,\,\rightharpoonup\,\,\, \Omega_{G, F} $};
	
		\node at (2,-.6) {$\Gamma(H \uparrow G)$};
		\end{tikzpicture}
	\end{center}

	\begin{theorem}$(Frobenius\, Reciprocity)$ $(\cite{cur93}, \cite{ful108}, \cite{Weintraub}, \cite{ser118}, \cite{james123})$\label{Fro1}
		Let $G$ be a finite group and $H$ be a subgroup of $G$. Let $F$ be a field of characteristic $0$ or prime to the order of $G$. 
		Let $(\rho, V)$ be a representation of $G$ and $(\eta, W)$ that of $H$. 
		Then there is an isomorphism of $F$-vector spaces
		$$\mathrm{Hom}_{F[G]}(V, \mathrm{Ind}(W)\uparrow^{G}_{H}) 
		\simeq \mathrm{Hom}_{F[H]}(\mathrm{Res}(V)\downarrow^{G}_{H}, W).$$
	\end{theorem}

	\begin{corollary}\label{Fro} 
		Let $G$ be a finite group and $H$ be a subgroup of $G$. Let $F$ be a field of characteristic $0$ or prime to $|G|$. Then the following statements are true.
		\begin{itemize}	
			\item [(1)] If the multiple edges are ignored, the directed bipartite graph $\Gamma(G \downarrow H)$ is same as the directed bipartite graph $\Gamma(H \uparrow G)$ with arrows reversed.
			\item[(2)] If $F$ is an algebraically closed field, then the directed bipartite graph $\Gamma(G \downarrow H)$ is same as the directed bipartite graph $\Gamma(H \uparrow G)$ with arrows reversed.
		\end{itemize}
	\end{corollary}

	\begin{definition}\rm
		Let $G$ be a finite group. Let $F$ be a field with characteristic $0$ or prime to $|G|$. Let
		\begin{equation}\tag{*}  
		\langle{e}\rangle = G_0 < G_1 < G_2 < \dots < G_n = G.
		\end{equation}
		be a multiplicity free chain of subgroups.
		By Corollary \ref{Fro}, for each $i = 1,2,\dots, n$, the restriction graph 
		$\Gamma( G_{i} \downarrow G_{i-1})$ is same as induction graph $\Gamma( G_{i-1} \uparrow G_{i})$ with arrows reversed, and ignoring arrows we abbreviate it simply by $\Gamma(G_{i} | G_{i-1})$. There is $1-1$ correspondence between the set of inequivalent $F$-irreducible representations of $G_{i}$ and the  
		set of primitive central idempotents in the semisimple group algebra $F[G_{i}]$. We attach to each vertex in the undirected bipartite graph $\Gamma(G_{i} | G_{i-1})$, the corresponding primitive central idempotent, and we call the resulting undirected  bipartite graph a {\bf{PCI-diagram}} of the pair $(G_{i}, G_{i-1})$. We call the union of PCI-diagrams of all  pairs $(G_{i}, G_{i-1}), i = 1, 2, \dots, n,$ a {\bf{ PCI-diagram}} of $G$ over $F$ associated with the series $(*)$. 
	\end{definition}
	\section{Motivation}
	In this section, we give the motivation of four works which are mentioned at the beginning of this chapter.
	
	$(1)$ The Witt-Berman theorem (see \cite{berman121}, \cite{cur93}, Theorem $42.8$) asserts that, for a finite group $G$ and a field $F$ of characteristic $0$, the number of irreducible $F$-representations of $G$ is equal to the number of "$F$-conjugacy classes" of $G$, which allows us to consider "$F$-character table" like usual character table in Frobenius theory on group representations. If $F$ is an algebraically closed field, there is a formula (see \cite{Yam80}, \cite{combi91}, Corollary $2.1.7$) for computing primitive central idempotents of $F[G]$ in terms of usual characters and usual conjugacy classes. If $F$ is non-algebraically closed field, there is also a formula for computing primitive central idempotents of $F[G]$ (see \cite{Yam80}, \cite{combi91}, Theorem $2.1.6$), but we express it in terms of irreducible $F$-characters and "$F$-conjugacy classes". Most importantly, the complete set of primitive central idempotents of $F[G]$ can be obtained from the "$F$-character table". 
	When one asks for representations over a non-algebraically closed field $F$, the {\it arithmetic} of $F$ significantly comes into play, in particular, the knowledge of which roots of unity lie in the field $F$. We show that "$F$-conjugacy class" of an element of order $n$ is the union of certain conjugacy classes of $G$, and which can be determined by the decomposition of $n$-th cyclotomic polynomial $\Phi_{n}(X)$ into irreducible polynomials over $F$.
	
	$(2)$ The classical approach to representations of finite solvable groups over algebraically closed fields is based on Clifford theory on group representations (see \cite{alg111}, \cite{cl50}, \cite{combi91}, Section $3.6$). But there is no such approach to representations of finite solvable groups over non-algebraically closed fields. To my knowledge, there is no literature on it. In general, the rationality question on group representations was answered by Schur (see \cite{re98}, \cite{sc70}). After reading Schur's work (see \cite{re98}, \cite{sc70}), we realized that to study representations of solvable groups over non-algebraically closed fields it is important to understand the algebraically closed situation. This motivated us to study the classical approach in details and independently explore many results with various examples.  
	
	$(3)$ Frobenius, one of the pioneers in group representations, besides contributing to the development of general theory, as an illustration he classified, up to isomorphism, the irreducible complex representations of $S_{n}$. But Frobenius actually gave only the characters of irreducible representations of $S_n$. On the other hand, Young (see \cite{young}) dealt only with the representations of $S_{n}$. He used natural inclusions $S_{1} \subset S_{2} \subset \cdots \subset S_{n}$ to construct irreducible matrix representations of $S_{n}$ over $\mathbb{C}$. The crucial property of this chain is that $S_{k}$ is generated by adding a single element to $S_{k-1}$. The idea of "long generators" for finite solvable groups came from this chain. A new approach to the finite dimensional irreducible complex representations of symmetric groups was developed by Anatoly Vershik and Andrei Okounkov (see \cite{Mu106}, \cite{ver107}). We are going to do something similar for finite solvable groups.
	Intrinsically, our philosophy of constructing the irreducible matrix representations of finite solvable groups over $\mathbb{C}$ is that of Young, but the details are much simpler.
	
$(4)$ The connection between group representations and the structure theory of group algebras was widely recognized after very influential papers (see \cite{no60}, \cite{noe117}) by Emmy Noether. 
Artin-Wedderburn theorem of semisimple algebras (see \cite{Dor94}, Theorem $3.1$, \cite{Weintraub}, Theorem $2.16$) states that simple components of a semisimple group algebra $F[G]$ are in $1-1$ correspondence with irreducible $F$-representations of $G$, and each simple component is generated by a primitive central idempotent in $F[G]$. So, the problem of constructing irreducible $F$-representations of $G$ is closely related to the problem: {\it Given a semisimple group algebra $F[G]$, compute the complete set of primitive central idempotents of $F[G]$?} The traditional approach (see \cite{combi91}, \cite{Yam80}) for computing the primitive central idempotents of $F[G]$ involves characters and computations in algebraic extensions of $F$. In view of computational
difficulties for this method, the following problem has been raised in recent years: 

{\bf{Problem:} \it Find a character-free expression of the primitive central idempotents of a semisimple group algebra $F[G]$ and its Wedderburn decomposition.}

In the last few years, this problem has been solved for certain class of groups and for some
specific fields. Jespers, Leal, Paques (see \cite{jespers125}) worked out expressions of primitive central
idempotents of nilpotent groups over $\mathbb{Q}$ . 
After we studied their papers, we realised that their works could be explained much better if a suitable system of generators of groups is used in a systematic way.
In fact, for a finite abelian group $G$, we give simple expressions of primitive central idempotents of $\mathbb{Q}[G]$ using a long presentation of $G$ and find its Wedderburn decomposition. Jespers, Leal and Paques (see \cite{jespers125} ) solved this problem.  However, we solve it by using a long presentation of $G$. To my knowledge, a systematic way of construction of  irreducible representations of finite abelian groups over arbitrary fields is not available in the literature. We give algorithm for constructing the irreducible representations of a finite abelian group over a field of characteristic $0$ or prime to the order of the group and present a systematic way for computing the primitive central idempotents of the abelian group algebra.
 \section{Main results of the thesis}
	My thesis work starts from Chapter $3$. We briefly touch upon the main results.
	
	In Chapter $3$, we show that "$F$-conjugacy class" of an element of order $n$ in a group $G$ is the union of certain conjugacy classes, and which is determined by the decomposition of $n$-th cyclotomic polynomial $\Phi_{n}(X)$ into irreducible polynomials over $F$. We give a formula for computing the primitive central idempotents of a semisimple group algebra $F[G]$ in terms of $F$-characters and "$F$-conjugacy classes" of $G$, and which can be read from the "$F$-character table" of $G$. 
 
	Let $L_1, L_2,\dots , L_r$ be the $F$-conjugacy classes of $G$. Let $\chi$ be an ${F}$-character of $G$, and $\chi(L_i)$ denote the common value of $\chi$ over $L_i$. Let $L_{i}$ be the $F$-conjugacy class of an element of $x$. We denote the $F$-conjugacy class of $x^{-1}$ by $L_{i}{^{-1}}$. For any subset $S$ of $G$, $S^*$ denotes the formal sum of elements of $S$. 

\begin{theorem}
		Let $G$ be a finite group and $F$ a field of characteristic $0$ or prime to $|G|$. Let $(\rho, V)$ be an irreducible $F$-representation of $G$ with corresponding primitive central idempotent $e$. 	Let $\chi$ be the character corresponding to $\rho$ and $\mathrm{End}_{F[G]}(V) = D$, dim$_{D}(V) = n$. Then $e = \frac{n}{|G|} \sum_{i = 1}^{r} {\chi(L_{{i}}^{-1})}L_{{i}}^*.$
\end{theorem}
In Chapter $4$, we describe inductive construction of complex representations of finite solvable groups using Clifford theory on group representations (see \cite{alg111}, \cite{cl50}, \cite{combi91}). 
Using the following theorem, one can inductively construct the irreducible representations of a finite solvable group over $\mathbb{C}$. This is a known result (see \cite{alg111}, Theorem $13.52$), but we have given an indirect proof and also  given many important consequences.   
	\begin{theorem}(Index-$p$ theorem)
		Let $G$ be a group, and $H$ a normal subgroup of index $p$, $p$ a prime.
		Let $\eta$ be an irreducible representation of $H$ over $\mathbb{C}$.
		\begin{itemize}
			\item[(1)]If the $G/H$-orbit of $\eta$ is a singleton, then $\eta$ extends
			to $p$ mutually inequivalent representations $\rho_1, \rho_2,\dots, \rho_p$  of $G$ over $\mathbb{C}$.
			\item[(2)] If the $G/H$-orbit of $\eta$ consists of $p$ points $\eta = \eta_1, \eta_2,\dots, \eta_p$ then the induced representations $\eta_1\uparrow_H^G$, $\eta_2\uparrow_H^G,\dots, \eta_p\uparrow_H^G$ are equivalent, say $\rho$, which is an irreducible representation over $\mathbb{C}$.
		\end{itemize}
	\end{theorem}
\begin{corollary}
		Let $G$ be a solvable group. Then $G$ has a 
		subnormal series such that the successive quotients are 
		isomorphic to cyclic groups of prime order. The last but one term in the subnormal series is a cyclic group of prime order. So all its irreducible 
		complex representations are of degree one. Then starting with degree one 
		representations, we can build all the complex irreducible representations of $G$ by the processes of successive extension and induction.
	\end{corollary}
\begin{definition}
Let $G$ be a finite group. Let $H$ be a subgroup of $G$ and of index $n$. 
Let $\eta$ be an $F$-representation of $H$.
An $F$-representation $\rho$ of $G$ is said to be a 
{\it generalized extension} of $\eta$ if $\rho \downarrow_{H}^{G} = 
\underbrace{ \eta \oplus \cdots \oplus \eta }_{m \, \mathrm{copies}}$, where $1 < m \leq n$.
\end{definition}
The following theorem is an analogue of the Index-$p$ theorem in general situation.
\begin{theorem}
Let $G$ be a group and $H$ a normal subgroup of index $p$, a prime. 
Let $F$ be a field of characteristic $0$ or prime to $|G|$. Let $\eta$ be an irreducible $F$-representation of $H$.
\begin{enumerate}
\item[(1)] If $G/H$-orbit of $\eta$ is singleton, then there exists an irreducible $F$-representation of $G$ which is either an extension or a generalized extension of $\eta$.

\item[(2)] If the $G/H$-orbit of $\eta$ consists of $p$ points $\eta = \eta_1, \eta_2, ..., \eta_p$ then the induced representations $\eta_1\uparrow_H^G$, $\eta_2\uparrow_H^G, ... , \eta_p\uparrow_H^G$, are equivalent, say $\rho$, which is an irreducible.
		\end{enumerate}
		\begin{corollary}
			Let $G$ be a finite solvable group. Let $F$ be a field of characteristic $0$ or prime to $|G|$. Then every irreducible representation of $G$ can be obtained 
			starting with an irreducible representation of a cyclic subgroup of prime order by
			the processes of successive induction, extension and generalized extension.
		\end{corollary}
	\end{theorem}
	In Chapter $5$, we define notion of the "diagonal subalgebra" (unique up to conjugacy) of a group algebra over an algebraically closed field of characteristic $0$. A finite solvable group always has a subnormal series whose successive quotients are cyclic groups of prime order. By Theorem $\ref{index}$, such a subnormal series is a multiplicity free chain of subgroups over $\mathbb{C}$ so that one can consider the "Gelfand-Tsetlin algebra" associated with that series. Moreover, "Gelfand-Tsetlin algebra" (see \cite{Mu106}, Section $2$) is the "diagonal subalgebra" of the complex group algebra. 
	The following theorem gives a convenient set of generators of the "Gelfand-Tsetlin algebra" in terms of a long system of generators. 
	\begin{theorem}\label{diagonal}
		Let $G$ be a finite solvable group. Let
		\begin{center}
			$ \langle{e}\rangle = G_0 < G_1 < \cdots < G_n = G$
		\end{center}
		be a subnormal series with quotient groups $G_{i}/G_{i-1}$ are isomorphic to cyclic groups
		of prime order $p_{i}$
		and associated long presentation 
		\begin{align*}
		G=\langle x_1,\ldots,x_n \, |\, & x_i^{p_i} = w_i(x_1,\ldots,x_{i-1}), 
		x_i^{-1}x_jx_i = w_{ij}(x_1,\ldots, x_{i-1}), \,\, j<i \rangle,
		\end{align*}
		where $w_i$ and $w_{ij}$ are certain words in $x_1, x_2,\ldots, x_{i-1}$. 
		For each $i = 1, 2, \ldots, n$, let $\overline{C}_{G_{i}}(x_{i})$ 
		denote the $G_{i}$-conjugacy class sum of $x_{i}$ in $\mathbb{C}[G_{i}]$. 
		Then
		$${GZ}_{n} = \langle\,\, {\overline{C}_{G_1}(x_1)}, 
		\,\,{\overline{C}_{G_2}(x_2)},\ldots, {\overline{C}_{G_n}(x_n)}\rangle.$$
	\end{theorem}
	The following theorem (see \cite{berman120}) is due to S. D. Berman.  
	We give an indirect proof of this theorem, which we do not find in any literature.
	\begin{theorem}
Let $G$ be a finite group and $H$ be a normal subgroup of index $p$,
a prime. Let $G/H = \langle {xH} \rangle$, for some $x$ in $G$. Let $(\eta, W)$ be an irreducible $\mathbb{C}$-representation of $H$ and $e_{\eta}$ be its corresponding primitive central idempotent in $\mathbb{C}[H]$. 
We distinguish two cases:
\begin{itemize}
			\item[(1)] If $e_{\eta}$ is a central idempotent in $\mathbb{C}[G]$, 
			then $\eta$ extends to $p$ distinct irreducible representations
			$\rho_{1}, \rho_{2}, \dots , \rho_{p}$ (say) of $G$. 
			It follows that $(\overline{C}_{G}(x))^{p}{e_{\eta}} = {\lambda}{e_{\eta}}$,
			where $\overline{C}_{G}(x)$ denotes the $G$-conjugacy class sum of $x$ and $\lambda \neq 0$. For each $i$, 
			$1 \leq i \leq p$, let $e_{\rho_{i}}$ be the primitive central idempotent corresponding to the representation $\rho_{i}$. Then
			$${e_{\rho_{i}} = \frac{1}{p} 
				\Big{(} 1 + \zeta^ic + \zeta^{2i} c^2 +\cdots +\zeta^{i(p-1)}c^{p-1}\Big{)} 
				e_{\eta},}$$
			where $c = \frac{\overline{C}_{G}(x){e_{\eta}} }{\sqrt[p]{\lambda}}$ and $\zeta$ 
			is a primitive $p^{th}$ root of unity in $\mathbb{C}$.  
			Moreover, 
			$$e_{\eta} = e_{\rho_1} + e_{\rho_2} + \cdots + e_{\rho_p}.$$
			
			\item[(2)] If $e_{\eta}$ is not a central idempotent in $\mathbb{C}[G]$,
			then 
			$\eta\uparrow_H^G, \eta^{x}\uparrow_H^G, \ldots ,\eta^{x^{p-1}}\uparrow_H^G$ 
			are all equivalent to an irreducible representation $\rho$ (say) of $G$ 
			over $\mathbb{C}$ and in this case,
			$$e_{\rho} = e_{\eta} + e_{\eta^{x}} + \cdots + e_{\eta^{x^{p-1}}}.$$
		\end{itemize}
	\end{theorem}
In Chapter $6$, we present an algorithm for constructing the irreducible matrix representations of a finite solvable group $G$ over $\mathbb{C}$.  We briefly describe the algorithm.
	
We first fix a subnormal series of $G$ whose successive quotients are cyclic groups of prime order. Associated with that series, we have long presentation, long system of generators and PCI-diagram.
	Let $(\rho, V_{\rho})$ be an irreducible $\mathbb{C}$-representation of $G$ and $e_{\rho}$ 
	the corresponding primitive central idempotent in $\mathbb{C}[G]$. From the PCI-diagram, we can obtain the primitive central idempotent $e_{\rho}$, a system of primitive (not necessarily central) idempotents whose sum is $e_{\rho}$ and the "diagonal subalgebra" of $\mathbb{C}[G]e_{\rho}$. Using them, we find an ordered basis for the representation space $V$ as elements of $\mathbb{C}[G]$. Moreover, all of them can be expressed in terms of the long system of generators. Finally w.r.t. that basis, we find matrices for the long generators.  
	
	In Chapter $7$, we present an algorithm for constructing the inequivalent irreducible matrix representations of a finite abelian group over a field of characteristic $0$ or prime to the order of the group. For that, we first construct the irreducible representations of a cyclic group and then  consider the general case.
	
The following theorem gives the inequivalent irreducible representations of a cyclic group $C_{n}$ over a field of characteristic $0$ or prime to $n$.
	\begin{theorem}\label{theorem}
		Let $G = C_{n} = \langle{x \,|\, x^{n} = 1}\rangle$. Let $F$ be a field of characteristic $0$ or prime to $n$. Let $X^{n} - 1 =  \prod^{k}_{i = 1}f_{i}(X)$ be the decomposition into irreducible polynomials over $F$. Then $G$ has $k$ irreducible representations $\rho_{1}, \rho_{2}, \dots , \rho_{k}$, say, which are defined by: $\rho_{i}(x) = C_{f_{i}(X)}$, where $C_{f_{i}(X)}$ denotes the companion matrix of $f_{i}(X)$.
	\end{theorem}
	The following theorem gives the faithful irreducible representations of a cyclic group $C_{n}$ over a field of characteristic $0$ or prime to $n$.
\begin{theorem}\label{th2}
Let $G = C_{n} = \langle x \,|\, x^{n} = 1 \rangle$. Let $F$ be a field of characteristic $0$ or prime to $n$. Let $\Phi_{n}(X) = \prod^{k} _{i = 1}{f_{i}(X)}$ be the factorization into irreducible polynomials over $F$. Then 
\begin{itemize}
\item [1.] There is a bijective correspondence between the set of faithful irreducible representations of $G$ over $F$ and the set of irreducible factors $\Phi_{n}(X)$ over $F$. 
\item [2.] If $\rho_{i}$ is the faithful irreducible representation corresponding to the irreducible factor $f_{i}(X)$, then $\rho_{i}$ is defined by: $x \mapsto C_{f_{i}(X)}$, where $C_{f_{i}(X)}$ denotes the companion matrix of $f_{i}(X)$.
\end{itemize}
\end{theorem}
	The following theorem (see \cite{gor109}, Chapter $3$, Theorem $2.3$) is a known result, which asserts that to construct the irreducible representations of a finite abelian group, it is sufficient to construct the faithful irreducible representations of its cyclic quotients.
	\begin{theorem} \label{th3}
		Let $G$ be a finite abelian group. Let $F$ be a field of characteristic $0$ or prime to $|G|$. Then every irreducible representation of $G$ over $F$ factors through a faithful irreducible representation of a cyclic quotient.
	\end{theorem}
	The following theorem gives a parametrization of the set of irreducible representations of a finite abelian group.
	\begin{theorem}\label{th4}
		Let $G$ be a finite abelian group. Let $F$ be a field of characteristic $0$ or prime to $|G|$. Let $\Omega_{G}$ denote the set of irreducible $F$-representations of $G$. Let $C_{n}$ denote a cyclic group of order $n$.
		For $d$ dividing $|G|$, let $\mathcal{H}_{d} = \{H \leq G \,|\, G/H \cong C_{d}\}$. Let $\Omega^{o}_{C_{d}}$ denote the set of faithful irreducible representations of $C_{d}$. Then $\Omega_{G}$ is in a bijective correspondence with the set 
		$\displaystyle {\cup_{d | |G|}(\mathcal{H}_{d} \times \Omega^{o}_{C_{d}})}$.
	\end{theorem}
	By Theorems $1.3.7 - 1.3.9$, we give an algorithm for constructing the irreducible representations of a finite abelian group.
	
	Besides that, we give expressions for the primitive central idempotents of rational group algebra of a finite abelian group using a long presentation and its Wedderburn decomposition. For that, it is sufficient to consider its $p$-primary parts. We now briefly describe how to get the primitive central idempotents of rational group algebra of a finite abelian $p$-group.

Let $G$ be an abelian $p$-group of order $p^m$, $p$ a prime and $m \geq 1$. Let us consider a subnormal series:
$$ \langle e \rangle = G_{o} \leq G_1 \leq \dots \leq G_{l}\leq G_{l+1}\dots \leq G_{m} = G$$ such that $[G_{i}\colon G_{i-1}] = p$. Associated with the above subnormal series, we have a long presentation of $G$ and let $\{x_1, x_2, \dots , x_m\}$
	be the associated long system of generators of $G$. 
	We define 
	$e_{x_i}= ({1+x_i + \dots + x_i^{p-1}})/{p}$ and  $e_{x_i}'=1-e_{x_i}$.
	We stated four rules to draw the PCI-diagram of $G$ over $\mathbb{Q}$ from which we can read primitive central idempotents of $\mathbb{Q}[G]$.\\
	\noindent
	{\it \textbf{Rule 0:}}
	{\it Each non trivial primitive central idempotent at any stage can be expressed in a product form and
		contains exactly one $e^{'}_{z}$ as a factor, where $z$ is a generator of the chosen long presentation of $G$.}

	\noindent
	{\it \textbf{Rule 1:}}
	{\it Let $e$ be the trivial primitive central idempotent in the $l$-th stage of the PCI-diagram. Let $u$ be the new generator
		introduced in the $(l + 1)$-st stage of chosen long presentation of $G$. Then $e$ is adjacent to precisely two vertices
		$ee_{u}$ and $ee^{'}_{u}$ in the $(l + 1)$-st stage of PCI-diagram.}\\ 
\noindent
{ \it \textbf{Rule 2:}}
{\it Let $e$ be a non trivial primitive central idempotent in the $l$-th stage of PCI-diagram. Then $e$ contains precisely one
$e^{'}_{z}$ as a factor, where $z$ is a generator of the chosen long presentation. Let $u$ be the long generator introduced in the
$(l + 1)$-st stage of PCI-diagram such that $z = u^{p^s}$, for some positive integer $s$. Then $e$ is adjacent to itself in the $(l + 1)$-st stage of the PCI-diagram.}\\ 

\noindent
{\it \textbf{Rule 3:}}
{\it Let $e$ be a non trivial primitive central idempotent in the $l$-th stage of PCI-diagram. Then $e$ contains precisely one
$e^{'}_{z}$ as a factor,
where $z$ is a generator of the chosen long presentation. Let $u$ be the long generator introduced in the $(l + 1)$-st stage of PCI-diagram
such that $z \neq u^{p^s}$, for any positive integer $s$. Then $e$ is adjacent to precisely $p$ vertices $ee_{u}, ee_{zu}, \dots ,
ee_{z^{p - 1}u}$ in the $(l + 1)$-st stage of the PCI-diagram.}
 \section{Organisation of the thesis}
This thesis consists of nine chapters. We briefly touch upon the contents of each chapter.
	\begin{itemize}
		\item In Chapter $2$, {'\it {Semisimple algebras}'}, we give a brief introduction to semisimple algebras.  
		\item In Chapter $3$, {'\it {$F$-conjugacy, $F$-character table, $F$-idempotents'}}, we present a brief introduction to Schur's theory (see \cite{re98}, \cite{sc70}) on group representations over a field with characteristic $0$. We give a formula for computing primitive central idempotents in a semisimple group algebra $F[G]$ in terms of $F$-characters and $F$-conjuagcy classes in $G$, which can be obtained from the "$F$-character table".
		
		\item In Chapter $4$, {'\it{Clifford theory'}}, we present the traditional approach to construction of irreducible representations of finite solvable groups over algebraically closed fields, and which is based on Clifford theory (see \cite{alg111}, \cite{cl50}, \cite{combi91}). We describe this approach in details and illustrate it with various examples. 

		\item In Chapter $5$, {'\it{Diagonal subalgebra'}}, we introduce the notion of "diagonal subalgebra" (unique up to conjugacy) of a group algebra over an algebraically closed field. We give an introduction to "Gelfand-Tsetlin alegbra" (see \cite{Mu106}, Section $2$) for an inductive chain of subgroups of a finite group with "simple branching". A finite solvable group always has a subnormal series such that each successive quotient group is a cyclic group of prime order, and which is a multiplicity free chain of subgroups. So, one can consider long generators and "Gelfand-Tsetlin algebra" associated with that. 
		We find a set of generators of "Gelfand-Tsetlin algebra" (see \cite{Mu106}, Section $2$) of a solvable group using long generators.
		
		\item In Chapter $6$, {'\it {Algorithmic construction of matrix representations of a finite solvable group over $\mathbb{C}$'}}, we give an algorithm for constructing the irreducible matrix representations of a finite solvable group $G$ over $\mathbb{C}$ using a long presentation.
		
		\item In Chapter $7$, {'\it {Representations of finite abelian groups'}}, we give an algorithm for constructing the irreducible matrix representations and the primitive central idempotents of a finite abelian group over a field of characteristic $0$ or prime to order of the group and give a systematic way of computing the primitive central idempotents of the group algebra. Besides that, using a long presentation, we give a  character-free expression of the primitive central idempotents of a rational group algebra of finite abelian group and its Wedderburn decomposition. 

		\item In Chapter $8$, {'\it{Examples'}}, we illustrate our algorithm and the approach based on Clifford theory (see \cite{alg111}, \cite{cl50}, \cite{combi91}) for constructing irreducible matrix representations of finite solvable groups over $\mathbb{C}$ with various examples.
		\item In Chapter $9$, {'\it{Summary and future work'}}, we give a summary of the thesis and list few problems for future research work.

\end{itemize}
\chapter{Semisimple algebras}
In this chapter, we present a brief introduction to theory of semisimple algebras. We know that study of ordinary group representations is equivalent to study of modules over semisimple group algebras. The main objective of this chapter is to state Artin-Wedderburn structure theorem for semisimple algebras and its applications to semisimple group algebras.

\section{Basic definitions}
Let $F$ be a field. 
An {\bf algebra } $A$ over $F$ is a ring with identity, which is also an 
$F$-vector space such that 
\begin{equation*}
\alpha(ab)=(\alpha a)b=a(\alpha b) \hskip5mm(\alpha\in F, \,\, a, b \in A).
\end{equation*}
The two basic examples of algebras which frequently occur in 
representation theory (of groups) and which 
will also occur throughout the thesis are the following. 

(1) If $D$ is a division ring with center $F$, then ${\rm M_n}(D)$, the 
set of $n\times n$ matrices over $D$, is an algebra over $F$ with usual addition, 
multiplication and scaling of matrices. 

(2) If $G$ is any group and $F$ any field, consider $F[G]$ the 
free $F$-module on the set $G$. A typical element of $F[G]$ is of the form 
$\sum_{g\in G} \alpha_g g$ with $\alpha_g\in F$ and  
$\alpha_g=0$ for all but finitely many $g\in G$. Define multiplication in $F[G]$ as 
follows:
$$\Big{(}\sum_{g\in G} \alpha_g g\Big{)}
{\Big(}\sum_{g\in G} \beta_g g\Big{)}
=\sum_{g\in G}\Big{(}\sum_{hk=g} \alpha_h\beta_k \Big{)} g.$$
Then $F[G]$ is an algebra over $F$. 

With this definition and examples of algebra, we move to recall 
some well-known notions associated to algebra. For this we fix an 
algebra $A$ over a field $F$. 
A {\bf subalgebra} of $A$ is a subspace of $A$, which is also a subring of $A$. 
The set
\begin{equation*}
Z(A):= \{z \in A \,| \,za = az \mbox{ for all } a \in A \}
\end{equation*}
is called the {\bf center} of $A$. If $Z(A) = A$ then the algebra $A$ 
is said to be {\it{commutative}}.
An element $e$ in $A$ is said to be an {\bf idempotent} if $e^2=e$. 
there exists $k \in \mathbb{N}$ such that $n^{k} = 0$.
A left ideal $I$ of $A$ is said to be {\bf nilpotent} if 
there exists $k \in \mathbb{N}$ such that 
\begin{center}
	$I^k:=\Big{\{} a \in A \,|\, a = \sum_{s=1}^m a_{1,s}a_{2,s} 
	\cdots a_{k,s},\, m \in \mathbb{N},\, a_{i,s} \in I \Big{\}} = 0.$
\end{center}
The {\bf radical} of $A$ is defined as the sum of all nilpotent left ideals of $A$.

We are now ready to define special types of algebras. 
An algebra $A$ is {\bf semisimple} if $\mathrm{rad}A = 0$. 
An algebra $A$ is {\bf simple} if the only two-sided ideals of $A$ are $A$ 
and $\{0\}$. As an example, the algebra ${\rm M}_n(F)$, $F$ any field, is a
simple algebra over $F$. If $G$ is a finite group then $\mathbb{C}[G]$ is a
semisimple algebra. 

\begin{remark}\rm
	The radical rad$\,A$ of an algebra $A$ is a two-sided ideal of $A$ and the factor $A/\mathrm{rad}{\,A}$ is a semisimple algebra.
\end{remark}

Let $A$ be an algebra. Let $I$, $J$ be two-sided ideals of $A$.
The ideals $I$ and $J$ are said to be isomorphic if 
they are isomorphic as $A$-modules. 
The ideal $I$ is said to be {\bf minimal} in $A$ if 
$I \neq \{0\}$ and $I$ does not contain any two-sided ideal, except $0$ and $I$.

We now recall an important result regarding the structure of semisimple algebras.

\begin{theorem}[Structure of semisimple algebras]
	Let $A$ be a finite-dimensional semisimple algebra over a field $F$. 
	For a minimal left ideal $I$ of $A$, let $B_{I}$ denote the sum of all minimal 
	left ideals of $A$, which are isomorphic to $I$. Then the following holds. 
	\begin{enumerate}
		\item  Every $B_{I}$ is a simple algebra over $F$, and a minimal two-sided ideal 
		of $A$. 
		\item The algebra $A$ is direct sum of all minimal two-sided ideals $B_{I}$, 
		where $I$  runs over a full set of non isomorphic minimal left ideals of $A$. 
		\item The decomposition in (2) is unique up to the order of the summands, 
		and the algebras $B_{I}$'s are called the simple components of $A$.
	\end{enumerate}
\end{theorem}
For proof of the theorem, one may refer to \cite{cur93}, Theorem $25.15$.

Let $A$ be a finite-dimensional semisimple algebra over a field $F$. 
By Theorem $1.1.1$, $A = B_1\oplus B_2\oplus\cdots\oplus B_k$, where $B_i$'s 
are simple algebras over $F$. Let $e_i$ denote the identity element of $B_i$.  
Then each $e_i$ is an idempotent in $A$ and 
$1_A = e_1+e_2+\cdots+e_k$, where $1_A$ is the identity element of $A$. 
Further $e_ie_j=0$ for $i \neq j$ and $e_i\in Z(A)$ for all $i$. 
Note that from $e_i$'s, one can get the simple algebra $B_{i}$ using the formula 
$B_{i} = Ae_{i}$.  This observation suggests that the computation of simple 
components of an algebra leads to the problem of determining a 
system of primitive central idempotents of $A$.

\begin{theorem} [Structure of simple algebras] 
	Let $A$ be a finite-dimensional simple $F$-algebra. 
	Then there exists $m \in \mathbb{N}$ such that $A$ is isomorphic to the 
	algebra of $m \times m$ matrices with entries in a finite dimensional 
	division algebra $D$ over $F$. There exists minimal left ideals 
	$I_1, I_2, \ldots, I_m$ of $A$ such that $A=I_1\oplus I_2\oplus\cdots\oplus I_m$, 
	and each of which are isomorphic to $D^{m}$. 
	Such a decomposition is unique up to isomorphism.
\end{theorem}
For proof of the theorem, see \cite{cur93}, Theorem $26.4$.

Let $A$ be a finite-dimensional simple $F$-algebra. 
Then the problem of the computation of the Wedderburn decomposition of $A$ 
can be reduced to the problem of the extraction of a system of primitive 
idempotents $\{u_{1}, u_{2}, \dots , u_{m}\}$ in $A$. 
However, such idempotents are not central and the system 
is not in general unique.

\begin{theorem}[Wedderburn-Artin]
	Let $A$ be a finite dimensional semisimple algebras
	if and only if it is a direct sum of simple algebras, i.e., a direct sum
	of matrix algebras over division algebras.
\end{theorem}
For proof of the theorem, see \cite{cur93}, Theorem $26.5$. 

\section{Modules over an algebra}

Let $A$ be an algebra over a field $F$. 
A {\bf left module} over $A$ is an $F$-vector space with an action 
$\delta: A \times M \longrightarrow M$ such that, for all $v, 
v'\in V$, $a, a'\in A, \alpha\in F$, the following conditions hold.
\begin{enumerate}
	\item $a(v + v') = av + av'$;
	\item $(a + a')v = av + av'$;
	\item $a'(av) = (a'a)v$;
	\item $1_A v = v$.
\end{enumerate}

From now onwards, we just say $A$-module instead of {\it left $A$-module}.

If $V$ and $W$ are two $A$-modules, then an $A$-{\bf homomorphism} 
is an $F$-linear map $\phi: V \rightarrow W$ 
such that $\phi(av) = a\phi(v)$ for $a \in A$ and $v \in V$. 
The set of all $A$-homomorphisms from $V$ to $W$ is denoted by Hom$_{A}(V, W)$ 
and is an $F$-vector space. The set of all the $A$-{\bf endomorphism} 
of $V$ is an $F$-algebra and is denoted with 
End$_{A}\,V =$ Hom$_{A}(V, V)$. 
An invertible $A$-homomorphism is called an $A$-{\bf isomorphism}. Two 
$A$-modules $V$ and  $W$ are said to be {\bf isomorphic} if there exists an 
$A$-isomorphism $\phi: V \longrightarrow W$.

Let $A$ be an $F$-algebra and let $V$ be an $A$-module.
\begin{enumerate}
	\item $V$ is said to be {\bf irreducible} if the only $A$-submodules of $V$ are 
	$0$ and $V$; otherwise, $V$ is said to be {\bf reducible}.
	\item $V$ is said to be {\bf decomposable} if there exists $A$-submodules 
	$U$, $W$ of $V$ such that $V = U \oplus W$;  when this happens, $W$ is called 
	a {\bf complement} of $U$ in $V$. $V$ is said to be {\bf indecomposable} if it 
	is not decomposable. A decomposition of $V$ into a direct sum of 
	indecomposable is a {\bf total decomposition} of $V$.
	\item  $V$ is said to be {\bf completely reducible} if for every $A$-submodule 
	$U$ of $V$ there exist a  complement for $U$ in $V$. 
	Equivalently, it can be shown that $V$ is completely reducible if it can be 
	decomposed as a direct sum of irreducible $A$-submodules of $V$.
\end{enumerate}

Every $A$-module admits a total decomposition, 
but only completely reducible have total decomposition made up of 
irreducible submodules. 
An important fact about total decompositions is that they are unique up 
to isomorphisms. This is due to the following classical result:

\begin{theorem}[Krull-Schmidt]
	Let $A$ be an Artinian algebra, and $V$ a finite-dimensional $A$-module. 
	Consider two decompositions of $V$ 
	into direct sums of non-zero indecomposable $A$-modules, namely:
	$V= V_1\oplus V_2\oplus \cdots \oplus V_h = W_1\oplus W_2\oplus\cdots\oplus W_k$. 
	Then $h = k$ and there exist a permutation $i\mapsto j_i$ of of 
	$\{1, 2,\ldots , h\}$ such that $V_1 \cong W_{j_1}, 
	V_{2} \cong W_{j_2}, \ldots , V_h\cong W_{j_h}$.
\end{theorem}
For proof of the theorem, see \cite{cur93}, Theorem $14.5$.

There are facts which are true for {\it irreducible modules}, 
but not for general modules. 
\begin{lemma} [Schur's lemma]
	Let $A$ be an $F$-algebra, and $V,W$ be irreducible $A$-modules.
	\begin{enumerate}
		\item Let $\phi \in $ Hom$_{A}(V, W)$. 
		Then $\phi$ is either zero or an $A$-isomorphism;
		\item If $\phi$ has an eigenvalue $\mu$ in $F$, then $\phi = \mu\mathrm{I}$. 
		In particular, if $F$ is an algebraically closed, 
		then $\mathrm{Hom}_A(V, V)$ is the set of scalar multiplications.
	\end{enumerate}
\end{lemma}
For proof of the lemma, see \cite{ser118}, Proposition $4$.

It is important to establish a good class of 
algebras such that every $A$-module is completely reducible. 
The following result tells us 
that a nice class of such algebras are group 
algebras over fields with characteristic not dividing $|G|$. 

\begin{theorem}[Maschke's theorem]
	Let $G$ be a finite group and $F$ a field whose characteristic does not divide 
	$|G|$. Then every $F[G]$-module is completely reducible.
\end{theorem}
For proof, see \cite{james123}, Theorem $8.1$.

Let $W$ be an $A$-module, and $V$ an irreducible $A$-module. 

{\bf $V$-homogeneous component} of 
$W$ is the sum of all irreducible $A$-submodules of $W$ that are isomorphic 
to $V$ and is denoted by $V(W)$. The module $W$ is said to be {\bf homogeneous} 
if $W = V(W)$ for some irreducible $A$-module $V$. 

Consider now a completely reducible $A$-module $W$. Then
$W = V_{1} \oplus V_{2} \dots \oplus V_{r},$
for some irreducible $A$-submodules $V_{i}$ of $W$. 
Consider an irreducible $A$-module $V$ and the homogeneous component $V(W)$. 
Then one can prove that $V(W)$ is the direct sum of those $V_{i}$'s 
which are isomorphic to $V$. It follows that every completely 
reducible module is direct sum of its homogeneous components. 
Every $F$-algebra $A$ has a natural structure of left $A$-module, 
called the {\bf regular module} of $A$. The submodules of 
the regular left module $A$ are exactly the left ideals of $A$ and 
it is possible to prove that the semi-simplicity of $A$ is 
equivalent to its complete reducibility as regular module. 
So, if characteristic of $F$ does not divide $|G|$, 
then the group algebra $F[G]$ is a semisimple algebra. 
A translation of the Wedderburn-Artin theorem in this context will give 
us plenty of information about the structure of structure of a group algebra.

\begin{theorem}
	Let $G$ be a finite group and let $F$ be a field such that characteristic of $F$ does not divide $|G|$. Then
	\begin{enumerate}
		\item $F[G]$ is a direct sum of a finite number of minimal two-sided ideals $\{B_{i}\}, 1 \leq i \leq r$, the simple components of $F[G]$. Each $B_{i}$ is a simple ring.
		\item Each simple component $B_{i}$ is isomorphic to a full matrix ring of the form $M_{n_{i}}(D_{i})$, where $D_{i}$ is a division ring containing an isomorphic copy of $F$ in its centre, and the isomorphism:
		$F[G] \cong \oplus^{r}_{i = 1}M_{n_{i}}(D_{i})$
		as an $F$-algebras.
	\end{enumerate}
\end{theorem}
For proof of the theorem, see \cite{milies126}, Theorem $3.4.9$.
\chapter{$F$-conjugacy, $F$-character table, $F$-idempotents}
Let $G$ be a finite group and $F$ be a field with characteristic $0$ 
or prime to $|G|$. Let $\overline{F}$ be the algebraic closure of $F$. 
The complete set of primitive central idempotents in $F[G]$ 
is useful for constructing the irreducible $F$-representations of $G$. 
The classical approach (see \cite{Dor94}, \cite{Yam80}) 
involves computations in algebraic extensions of $F$.
In this chapter, we give a formula for computing the primitive 
central idempotents of $F[G]$ in terms of irreducible $F$-characters 
and "$F$-conjugacy classes" of $G$. Moreover, 
one can read the complete set of 
primitive central idempotents in $F[G]$ from the "$F$-character table".

\section{Schur's theory of representations of finite groups}
In this section, we briefly recall 
Schur's theory (see \cite{re98}, \cite{sc70}) of the representations of finite groups over fields of characteristic $0$ (not necessarily algebraically closed), which involves arithmetic aspects . 
\begin{definition}\rm
	Let $L$ be a field extension of $F$ and $(\rho,V)$ an $L$-representation of 
	$G$. The representation $(\rho,V)$ is said to be {\it{realizable}} 
	over $F$ if 
	there exists an $F$-representation $(\eta,U)$ such that 
	$V\cong U\otimes_F L.$
\end{definition}        

\begin{definition}\rm
	Let $(\rho,V)$ be an irreducible $\overline{F}$-representation of $G$ 
	with character $\chi$, and let $F(\chi)$ denotes the algebraic extension of 
	$F$ obtained by adjoining 
	the character values $\{\chi(g): g \in G\}$ to $F$. The {\it{Schur index}} 
	of $(\rho,V)$ with respect to the field $F$ is defined as
	$$m_F(\chi):=\mbox{min} \{ [K:F(\chi)] \,|\, V
	\mbox{ is realizable over } K\}$$ 
	or equivalently it is the smallest positive integer $m$ such that 
	$m(\rho,V)$ is defined over $F(\chi)$.
\end{definition}

\begin{definition}\rm 
	Let $(\rho,V)$ be an irreducible $F$-representation of $G$. 
	We say $(\rho,V)$ is {\it absolutely irreducible} if 
	for every extension $L$ over $F$, $V\otimes_F L$ is $L$-irreducible.
\end{definition} 

\begin{definition}\rm
	A field $S$ is said to be {\it splitting field} for $G$ if every irreducible 
	representation of $G$ over $S$ is absolutely irreducible.
\end{definition}

\begin{definition}\rm
	Let $(\rho, V)$ be an $F$-representation of $G$. Then the {\it $F$-character}
	$\chi_{\rho}$ corresponding to the representation  $(\rho, V)$ is a map from $G$ 
	into $F$, and for any element $x$ in $G$, $\chi_{\rho}(x)$ is defined by 
	trace of the linear operator $\rho(x)$.
\end{definition}

\begin{definition}\rm
	Suppose $S$ is a Galois extension of $F$ and a splitting field for $G$. 
	Let $\chi$ be an $S$-character of $G$ and $\alpha\in {\rm Gal}(S/F)$. 
	The function $\chi^{\alpha}: G \rightarrow S$ 
	defined by $\chi^{\alpha}(g):= \alpha(\chi(g)),$ for all $g\in G,$ is 
	an $S$-character of $G$, called an 
	{\it algebraically conjugate} of $\chi$  with respect to $F$.
\end{definition}
\begin{theorem} [see\cite{re98}, Theorem $2$] (Schur)  $\label{sch}$
	Let $G$ be a finite group and $F$ be field of characteristic $0$ or prime to $|G|$. Let $W$ an irreducible $F[G]$-module. 
	Let $V$ be an irreducible submodule of the $\overline{F}[G]$-module 
	$W\otimes_{F}{\overline{F}}$. Then 
	$$W\otimes_{F}{\overline{F}} \cong s(V \oplus V^{\alpha_{1}} \oplus \dots 
	\oplus V^{\alpha_{t}}),$$
	where $V^{\alpha_{i}}$'s are pairwise non-isomorphic algebraically 
	conjugates modules, and $s$ is an integer. If $V$ affords the character
	$\chi$, then $s = m_{F}(\chi)$. If $W$ affords the character $\theta$, then 
	$$\theta = m_{F}(\chi)\sum_{\alpha \in H}{\chi^{\alpha}},$$
	where $H$ is the Galois group of $F(\chi)$ over $F$.
\end{theorem}

\section{{$F$-conjugacy classes}}

In this section, we define $F$-conjugacy and show $F$-conjugacy class of 
an element of $G$ depends on the arithmetic of the field $F$, 
in fact it depends on decomposition of cyclotomic polynomials into irreducible factors over $F$.
Throughout this section we assume that $F$ is a field of characteristic $0$ or prime to $|G|$.
\begin{definition}\rm
	Two elements $x$, $y$ in $G$ are said to be {\bf $F$-conjugate}: 
	$x \sim_F y$ if for all finite dimensional 
	$F$-representations $(\rho, V)$ with the associated characters $\chi_{\rho}$, 
	we have $\chi_{\rho}(x) = \chi_{\rho}(y)$.
\end{definition}

\begin{definition}\rm
	Let $x$ be an element in $G$. The {\bf $F$-conjugacy class} of $x$ 
	consists of all those elements in $G$, which are
	$F$-conjugate to $x$. We denote $F$-conjugacy class of $x$ by {\bf $C_{F}(x)$} 
	and conjugacy class of $x$ by {\bf $C(x)$}.
\end{definition}
\begin{remark}
	By the Frobenius-Schur theory, when $F$ is algebraically closed, then 
	$F$-conjugacy class is usual conjugacy class in purely group theoretic  sense. 
	When $F$ is not algebraically closed, $F$ has arithmetic and then 
	$F$-conjugacy depends on both the structure of $G$ and arithmetic of $F$. 
\end{remark}
\subsection{Decomposition of $n$-th cyclotomic polynomials over a field $F$ of characteristic not dividing $|G|$}
Let $n$ be a positive integer and $F$ a field of characteristic $0$ or prime to $n$. The $n$-th roots of $1$ in $\overline{F}$ form a cyclic 
group. Consider the $n$-th {\bf cyclotomic polynomial} $\Phi_{n}(X)\in\mathbb{Z}[X]$, 
which is the monic irreducible polynomial over $\mathbb{Z}$ satisfied by a 
primitive $n$-th root of unity in $\bar{F}$. The cyclotomic polynomial 
$\Phi_{n}(X)$ can be viewed an element 
in $F[X]$: if char$F=0$ then this is clear since $F$ contains 
$\mathbb{Q}$. If char$F=p$, a prime, then one considers the 
image of $\Phi_{n}(X)$ under the isomorphism 
$\psi: \mathbb{Z}[X] \mapsto \mathbb{F}_{p}[X]$, which maps $X$ to 
$X$ and $\mathbb{Z}$ onto $\mathbb{F}_{p}$ 
by reduction modulo $p$ of the coefficients. We denote again this 
image by $\Phi_{n}(X)$. We describe the decomposition 
of $\Phi_{n}(X)$ 
into irreducible factors over $F$.

\begin{proposition}\label{decom}
	Let $\Phi_n(X)=f_1(X)f_2(X)\cdots f_k(X)$ be the decomposition of 
	$\Phi_n(X)$ into irreducible factors over $F$. Then
	\begin{enumerate}
		\item The degrees of all $f_i(X)$'s are same.
		\item Let $\zeta$ be a root of one $f_i(X)$. Then all the 
		roots of $f_i(X)$ are $\{\zeta^{r_1}, \zeta^{r_2}, \ldots,\zeta^{r_s}\}$, 
		where all $r_i$'s are natural numbers with $r_1=1$, and 
		the sequence $\{ r_1, r_2,\ldots, r_s \}$ is independent of 
		irreducible factors of $\Phi_n(X)$ and any root of $\Phi_n(X)$.
	\end{enumerate}
\end{proposition}

\begin{proof}\rm
	The splitting field of the polynomial $\Phi_{n}(X)$ over $F$ is $F(\zeta)$.
	The Galois group Gal$(F(\zeta)/F)$, 
	permutes all the roots of $\Phi_n(X)$. Then the set all the roots of $f_i(X)$
	is $\{\sigma(\zeta)\, | \, \sigma \in {\rm Gal}(F(\zeta)/F)\}$.
	So the set of all the roots of 
	$f_i(X)$ is $\{\zeta^{r_1},\zeta^{r_2},\ldots,\zeta^{r_s}\}$, 
	where $r_i$'s are natural numbers with $r_1 = 1$. 
	Therefore the degree of  $f_i(X)$ is $s$, and So the set of all the roots
	of $f_i(X)$ is $\{\zeta^{r_1}, \zeta^{r_2},\cdots, \zeta^{r_{s}}\}$, 
	where $r_i$'s are natural numbers with $r_{1} = 1$. Therefore the degree 
	of  $f_i(X)$ is $s$, and hence the degrees of  $f_{i}(X)$'s are same. 
	The sequence $\{{r_1} = 1, r_2,\cdots,r_s\}$ is independent roots
	of  $f_i(X)$ up to the order and also independent of $f_i(X)$.
\end{proof}
\begin{theorem}
	Let $G$ be a finite group and $F$ a field such that characteristic of
	$F$ does not divide $|G|$.
	Let $x$ be an element of $G$, and the order of $x$ is $n$.
	Then $C_{F}{(x)}$ is equal to
	$C(x)\cup C(x^{r_2})\cup\cdots \cup C(x^{r_s})$, where
	$r_1 = 1, r_2,\ldots, r_s$ is the sequence associated with $\Phi_n(X)$ as in the above proposition.
\end{theorem}

\begin{proof}
	Fix $x\in G$ with $x$ of order $n$.
	We first prove that $C(x)\cup C(x^{r_{2}})\cup\cdots\cup C(x^{r_{s}})
	\subseteq C_{F}(x)$.
	Let $\rho$ be an arbitrary  $F$-representation of $G$, and $\chi$ be its
	character. Let $\zeta$ be a primitive $n$-th root of
	$1\in F$ (where $n$ is the order of $x$).
	Then $\chi(x)=\zeta^{u_1}+\zeta^{u_2}+\cdots +\zeta^{u_t}$ for some
	$u_i\in\mathbb{Z}$, and
	it belongs to $F$.
	Since $r_i$ ($2\leq i\leq s$) is relatively prime to $n$, $\sigma_i:\zeta\mapsto
	\zeta^{r_i}$ is an $F$-automorphism of $F(\zeta)$ and
	$$\chi(x^{r_i})=\zeta^{r_iu_1}+\zeta^{r_iu_2}+\cdots+\zeta^{r_iu_t}\in F.$$
	But then
	$$\chi(x^{r_i})=\zeta^{r_iu_1}+\cdots+\zeta^{r_iu_t}
	=\sigma_i(\zeta^{u_1}+\cdots+\zeta^{u_t})=\sigma_i(\chi(x))=\chi(x),$$
	for every irreducible $F$-character $\chi$, which shows that $C(x)\cup
	C(x^{r_{2}})\cup\cdots\cup C(x^{r_{s}})
	\subseteq C_{F}(x)$.
	
	Let $I_n$ denotes those distinct integers, relatively
	prime to $n$, which give all the automorphisms in ${\rm Gal}(F(\zeta)/F)$.
	Thus, $I_n$ is a subgroup of $(\mathbb{Z}/n\mathbb{Z})^*$. Define following
	relations on $G$:
	\begin{align*}
	g\sim h & \mbox{ if } g \mbox{ is conjugate to } h^a \mbox{ for some } a\in
	I_n,\\
	g\sim_F h & \mbox{ if } \chi(g)=\chi(h) \mbox{ for all irreducible
		$F$-characters of $G$}.
	\end{align*}
	The relation $\sim$ is equivalence relation on $G$ since $I_n$ is a (sub)group.
	Let $[g]_{\sim}$ and $[g]_{\sim_F}$ denotes the equivalence classes of $g$
	under $\sim$ and $\sim_F$ respectively. Clearly $[g]_{\sim}\subseteq
	[g]_{\sim_F}$ for all $g\in G$. Consider the sets
	\begin{align*}
	M &=\{ f:G\rightarrow F ~\mid~  \mbox{$f$ is constant on each $\sim$
		equivalence class of $G$}\},\\
	N &=\{ f:G\rightarrow F ~\mid~  \mbox{$f$ is constant on each $\sim_F$
		equivalence class of $G$}\}.
	\end{align*}
	Then $M$ and $N$ are $F$-vector spaces and we have the subspace relations:
	$$\mbox{span}\{ \phi_1,\ldots,\phi_r\}\subseteq N\subseteq M$$
	where $\phi_1,\ldots, \phi_r$ are all the distinct $F$-irreducible characters
	of $G$. We show that $M=\mbox{span}\{ \phi_1,\ldots,\phi_r\}$,
	which will imply
	that $[x]_{\sim}=[x]_{\sim_F}$ and the proof of the theorem will be complete.
	
	Consider $\theta\in M$. Since $\theta$ is constant on each conjugacy class of
	$G$, we can write $\theta=\sum_{i=1}^k c_i\chi_i$ where $c_i\in F(\zeta)$ and
	$\chi_i$'s are
	irreducible characters of $G$ over $F(\zeta)$.
	For $a\in I_n$, let $\sigma$ denotes the $F$-automorphism $\zeta\mapsto
	\zeta^a$ of $F(\zeta)$. Then for any $g\in G$, $\theta(g)=\theta(g^a)$, hence
	$$\sum_{i=1}^k c_i\sigma(\chi_i(g))=\sum_{i=1}^k
	c_i\chi_i(g^a)=\theta(g^a)=\theta(g)=\sum_{i=1}^k c_i\chi_i(g).$$
	Therefore $c_i=c_j$ if and only if $\sigma(\chi_i)=\chi_j$ (by independence of
	$\chi_i$'s over $F(\zeta)$). Therefore $\theta$ is an $F(\zeta)$-linear
	combination of $\phi_i$'s, say
	$$\theta=\sum_{i=1}^r d_i\phi_i.$$
	
	Claim: Each $d_i$ belongs to $F$.
	
	Since $\phi_1,\ldots,\phi_r$ are linearly independent over $F$, there exists
	$x_j\in G$ for $j=1,2,\ldots, r$ such that the $r\times r$ matrix
	$(\phi_i(x_j))$ is non-singular. Therefore, $d_i$'s are uniquely determined by
	the system of linear equations
	$$\theta(x_j)=\sum_{i=1}^r d_i\phi_i(x_j), \hskip5mm j=1,2,\ldots, r.$$
	It follows that $d_i=\sigma(d_i)$, for any $F$-automorphism $\sigma$ of
	$F(\zeta)$.
	
	Thus $\theta$ is an $F$-linear combination of $\phi_1,\ldots, \phi_r$, and now
	the proof of the theorem is complete.
\end{proof}
\begin{corollary}\rm
	The $F$-conjugacy class of $x$ is uniquely determined by the roots of just 
	one irreducible factor of $\Phi_{n}(X)$ over $F$, where $n$ is order of $x$.
\end{corollary}

\begin{remark}\rm
	In general $C_F(x)$ is not necessarily equal to the 
	{\it disjoint} union of $C(x^{r_i})$, for $i=1,2,\ldots,s$.
\end{remark}

\begin{example}\rm
	Let $F = {\mathbb{Q}}$. Then 
	$C_{\mathbb{Q}}(x)$ is equal to the union of $C(x^i)$, 
	where $i$ runs over the natural numbers relatively prime to $n$. 
	This follows immediately from the well known fact, 
	$\Phi_{n}(X)$ is irreducible over $\mathbb{Q}$.
\end{example}

\begin{example}\rm
	Let $F=\mathbb{R}$. Then  $C_{\mathbb{R}}(x)$ 
	is equal to the union of $C(x)$ and $C(x^{-1})$.
\end{example}

\begin{example}\rm
	Let $F=\mathbb{F}_q$, $q = p^r$, $p$ a prime and $r$ is a positive integer. 
	Let $d$ be the smallest positive integer such that $q^d \equiv 1 \pmod n$. 
	If $\zeta$ is a root of one irreducible factor of $\Phi_n(X)$ over 
	$\mathbb{F}_q$, then $\zeta^q, \zeta^{q^2},\ldots ,\zeta^{q^{d-1}}$ 
	are the other roots which follows immediately from the well known fact that 
	the automorphisms of $\mathbb{F}_q(\zeta)$ over $\mathbb{F}_q$ are 
	$\zeta\mapsto \zeta^{q^i}$ ($0\leq i<d$). 
	Then $C_{F}{(x)}$ is equal to the union of $C(x^{q^i})$ for 
	$i= 0,1,\ldots, d-1$.
\end{example}

\section{{$F$-character tables}}
In this section, we introduce the notion of "$F$-character table", where $F$ is a field of characteristic $0$,
which is an wonderful generalization of usual character table in the 
Frobenius theory. The following theorem is known as Witt-Berman theorem 
(see \cite{berman121}, \cite{cur93}). For the sake of completeness we now give a proof of this theorem. 
\begin{theorem} [\cite{cur93}, Theorem $42.8$](Witt-Berman)
	Let $G$ be a finite group and $F$ be a field of characteristic $0$. Then the number of $F$-conjugacy classes of $G$ 
	is equal to the number of irreducible $F$-representations of $G$.
\end{theorem}
\begin{proof}
	Let $\rho_1,\rho_2,\ldots,\rho_r$ be all the inequivalent $F$-irreducible 
	representations of $G$ and $\phi_1,\phi_2,\ldots,\phi_r$ be their 
	corresponding $F$-characters. Let $s$ denotes the number of 
	$F$-conjugacy classes of $G$. We consider the $F$-space of mappings 
	$f$ from $G$ into $F$, which are constant on $F$-conjugacy classes of $G$, 
	and is denoted by $M$. Then of course the dimension of $M$ is 
	equal to $s$. By definition of $F$-conjugacy classes it is clear that 
	$\phi_1,\phi_2,\ldots,\phi_r$ are all belong to $M$. Since 
	$\phi_1, \phi_2,\ldots, \phi_r$ are $F$-linearly independent, 
	we get $r\leq s$. To complete the proof it is enough to show 
	$M=$span$\{\phi_1, \phi_2, \ldots , \phi_r\}$. 
	
	Let $u$ be the exponent of $G$ and 
	$\zeta$ be a primitive $u$-th roots of unity in $\overline{F}$. 
	The field $L = F(\zeta )$ is a splitting field for $G$. 
	By Schur's theory (see Theorem $\ref{sch}$), $\phi_i=m_i(\chi_1^i + \chi_2^i + \cdots +\chi_{\delta_i}^i)$ 
	is the decomposition of $\phi_i$ into irreducible $L$-characters of $G$, 
	where $m_i$ is the Schur index of $\chi^i_j$ ($j=1,2,\ldots,\delta_i$) and 
	$\chi_1^i, \chi_2^i, ...,\chi_{\delta_i}^i$ are algebraically conjugate characters.
	
	Let $x$ be an arbitrary element of $G$. Then $\chi_j^i(x)=\zeta^{u_1}+\zeta^{u_2}+
	\cdots +\zeta^{u_t}$(say). Then for any positive integer $a$ we get 
	$\chi_j^i(x^a) =\zeta^{au_1} + \zeta^{au_2} + \cdots + \zeta^{au_t}$. 
	Consider $a$ with $(a,u)=1$ and
	let 
	$\alpha$ be the $F$-automorphism of  $\zeta\mapsto \zeta^a$ of $L$. 
	Then $\phi_i(x)=\phi_i(x^a)$, and this is true for each $i = 1,2,\ldots,r$. 
	So by definition $F$-conjugacy, we get that $x$ is $F$-conjugate to $x^a$.
	
	Since  $\phi_1,\phi_2,\ldots,\phi_r$ are $F$-linearly independent 
	$\phi_i \neq 0$ and $m_i$ is not divisible by characteristic of $F$, 
	$\chi_1^i + \chi_2^i + \cdots +\chi_{\delta_i}^i$ is an 
	$F$-multiple of $\phi_i$. 
	
	It is to be proved that any element $\theta$ of $M$ is an $F$-linear 
	combination of $\phi_1, \phi_2, \ldots , \phi_r$. 
	Since $\theta$ belongs to $M$, $\theta$ is a class function on $G$ into $L$. 
	So we can write $\theta = \sum_i c_i \chi_i$, where
	$c_i \in L$ and $\chi_i$'s are distinct absolutely irreducible characters of $G$.
	Let $\alpha$ be an $F$-automorphism of $L$ and is defined to be 
	$\zeta \longmapsto {\zeta}^a$. Since  $x$ is $F$-conjugate to $x^a$, 
	then $\theta(x) = \theta(x^a)$. 
	Now $\sum_i c_i\alpha(\chi_i)(x) = \sum_ic_i (\chi_i)(x^a)
	=\theta(x^a)=\theta(x)=\sum_i c_i\chi_i(x)$. 
	So $c_i = c_j$ if and only if $\alpha(\chi_i)=\chi_j$. 
	Therefore $\theta$ is an $L$-linear combination of $\phi_i$, 
	say $\theta = \sum_{i=1}^{r}d_i{\phi_i}$.
	Now our claim is that each $d_i \in F$. As $\phi_1, \phi_2,\ldots, \phi_r$ 
	are linearly independent over $F$, there exist $x_j\in G$ for $j=1,2,\ldots,r$  
	such that the matrix $(\phi_i(x_j))$ $i,j = 1,2,\ldots, r$ is non-singular. 
	Therefore $d_i$ is uniquely determined by the system of equations is 
	$\theta(x_j) = \sum_{i=1}^{r}d_i{\phi_i(x_j)}$ for $j = 1,2,\ldots,r$. 
	It follows that $d_i = \alpha(d_i)$, implies that $d_i \in F$. 
	So $\theta$ is an $F$-linear combination of 
	$\phi_1, \phi_2,\ldots,\phi_r$. 
	Therefore $M=$span$\{\phi_1, \phi_2, \ldots , \phi_r\}$, as required. 
	This completes the proof of the theorem.
\end{proof}
Since the number of $F$-conjugacy classes is equal to the number of $F$-irreducible representations, and character determines the representation 
uniquely, we can list the $F$-character values on $F$-conjugacy classes
in the form of a square matrix over $F$. This is called the {\bf $F$-character table}. 
Thus the $F$-character table  is a square matrix with entries in $F$. 
The columns of $F$-character table are parametrized by $F$-conjugacy classes, 
and the rows are parametrized by irreducible $F$-characters. Since the number 
of $F$-conjugacy classes in general is less than or equal to the number of 
conjugacy classes, the size of the matrix representing $F$-character table 
is smaller than the usual character table. The $F$-character 
table contains more information than the usual character table. 
\subsection{Examples}

\begin{example}
	${D}_8 = \langle {x,y | x^4 = y^2 = 1, y^{-1}xy = x^{-1}}\rangle$.\\
	The set of all $\mathbb{Q}$-conjugacy classes of ${D}_8$ are:
	
	\begin{center}
		$L_1 = \{e \}$, $L_2 = \{x^2\}$, $L_3 = \{x, x^3\}$, 
		$L_4 = \{y, x^2y\}$, $L_5 = \{xy, x^3y\}$.
	\end{center}
	The $\mathbb{Q}$-character Table of ${D}_8$ is:
	\begin{center}
		\begin{tabular}{|| l | c | r | l | l | l||}
			\hline \hline
			& $e$ & $x^2$ & ${x}$ & $y$ & ${xy}$ \\ \hline
			$\psi_1 $ & 1 & 1 & 1 & 1 & 1 \\ \hline
			$\psi_2 $ & 1 & 1 & 1 & -1 & -1 \\ \hline
			$\psi_3 $ & 1 & 1 & -1 & 1 & -1 \\ \hline
			$\psi_4 $ & 1 & 1 & -1 & -1 & 1 \\ \hline
			$\psi_5 $ & 2 & -2 & 0 & 0 & 0 \\ \hline
			\hline
		\end{tabular}
	\end{center}
\end{example}

\begin{example}
	${Q}_8 = \langle {x,y |x^4 =  1, x^2 = y^2, y^{-1}xy = x^{-1}}\rangle$.\\
	The set of all $\mathbb{Q}$-conjugacy classes of ${Q}_8$ are: 
	\begin{center}
		$L_1 = \{e \}$, $L_2 = \{x^2\}$, $L_3 = \{x, x^3\}$, $L_4 = \{y, x^2y\}$, 
		$L_5 = \{xy, x^3y\}$.
	\end{center}
	The $\mathbb{Q}$-character Table of ${Q}_8$ is:
	\begin{center}
		\begin{tabular}{|| l | c | r | l | l | l||}
			\hline \hline
			& $e$ & $x^2$ & ${x}$ & $y$ & ${xy}$ \\ \hline
			$\psi_1 $ & 1 & 1 & 1 & 1 & 1 \\ \hline
			$\psi_2 $ & 1 & 1 & 1 & -1 & -1 \\ \hline
			$\psi_3 $ & 1 & 1 & -1 & 1 & -1 \\ \hline
			$\psi_4 $ & 1 & 1 & -1 & -1 & 1 \\ \hline
			$\psi_5 $ & 4 & -4 & 0 & 0 & 0 \\ \hline
			\hline
		\end{tabular}
	\end{center}
\end{example}

\begin{remark}
	The $\mathbb{Q}$-character tables of $D_{8}$ and $Q_{8}$ are different, 
	on the other hand the usual character tables of $D_{8}$ and $Q_{8}$ are same.
\end{remark}

\begin{example}
	${\rm SL}_2(3)$ is generated by $x,y,z$ which satisfies following relations. 
	$\langle {x, y, z|x^4 = z^3 = 1, x^2 = y^{2}}, 
	{y^{-1}xy = x^{-1}, z^{-1}xz = y, z^{-1}yz = xy}\rangle$.\\
	The set of all conjugacy classes of ${\rm SL}_2(3)$ are:
	\begin{center}
		$ S_1 = \{e \}$, $S_2 = \{ x^2\}$, $S_3 = \{x, x^3, y, y^3, xy, xy^3\}$,
		$S_4 = \{z, zx, zy, zxy \}$, $S_5 = \{z^2, z^2x, z^2y, z^2xy\}$,
		$S_6 = \{zx^2, zx^3, zx^2y, zx^3y\}$, $S_7 = 
		\{z^2x^2, z^2x^3, z^2x^2y, z^2x^3y\}$.
	\end{center}
	The usual character table of ${\rm SL}_2(3)$ is:\\
	\begin{center}
		\begin{tabular}{|| l | c | r | r | r | r | r | r ||}
			\hline\hline
			& $e$ & $x^2$ & ${x}$ & $z$ & ${z^2}$ & $zx^2$ & ${z^2}x^2$\\ \hline
			$\chi_1$ & 1 & 1 & 1 & 1 & 1 & 1 & 1\\ \hline
			$\chi_2$ & 1 & 1 & 1 & $w$ & $w^2$ & $w$ & $w^2$ \\ \hline
			$\chi_3$ & 1 & 1 & 1 & $w^2$ & $w$ & $w^2$ & $w$\\ \hline
			$\chi_4$ & 3 & 3 & -1 & 0 & 0 & 0 & 0 \\ \hline
			$\chi_5$ & 2 & -2 & 0 & -1 & -1 & 1 & 1 \\ \hline
			$\chi_6$ & 2 & -2 & 0 & -$w$ & -$w^2$ & $w$ & $w^2$ \\ \hline
			$\chi_7$ & 2 & -2 & 0 & -$w^2$& -$w$ & $w^2$ & $w$ \\ \hline
			\hline
		\end{tabular}
	\end{center}
	The set of all $\mathbb{Q}$-conjugacy classes of ${\rm SL}_2(3)$ are:
	\begin{center}
		$ L_1 = \{e \}$, $L_2 = \{ x^2\}$,
		$L_3 = \{x, x^3, y, y^3, xy, xy^3\}$,
		$L_4 =  \{z, zx, zy, zxy, z^2, z^2x, z^2y, z^2xy\}$,
		$L_5 = \{zx^2, zx^3, zx^2y, zx^3y, z^2x^2, z^2x^3, z^2x^2y, z^2x^3y\}$.
	\end{center}
	The $\mathbb{Q}$-character table of ${\rm SL}_2(3)$ is:
	\begin{center}
		\begin{tabular}{|| l | c | r | r | r | r | r | r||}
			\hline
			\hline
			& $e$ & $x^2$ & ${x}$ & $z$ & $zx^2$\\ \hline
			$\psi_1 $ & 1 & 1 & 1 & 1 & 1 \\ \hline
			$\psi_2 $ & 2 & 2 & 2 & -1 & -1 \\ \hline
			$\psi_3 $ & 3 & 3 & -1 & 0 & 0 \\ \hline
			$\psi_4 $ & 4 & -4 & 0 & -2 & 2  \\ \hline
			$\psi_5 $ & 8 & -8 & 0 & 2 & -2  \\ \hline
			\hline
		\end{tabular}
	\end{center}
\end{example}
\section{{$F$-idempotents}}
In this section, we give a formula for computing the primitive central idempotents
of the group ring $F[G]$, where $F$ is a field of characteristic $0$, in terms of irreducible $F$-characters and $F$-conjugacy classes 
of $G$. Let $R = F[G]$. Then $R$ is a semisimple $F$-algebra, and irreducible 
$F$-representations of $G$ correspond to simple $R$-modules. 
Let $e_i$ ($i = 1,2,\ldots, r$) be the primitive central idempotents of $R$, 
then $R = \sum^{r}_{i = 1} Re_i$. Each $Re_i$ is a minimal two-sided ideal of $R$. 
Each $Re_{i}$ is a simple $F$-algebra with $e_{i}$ as the identity element. 
It is not an $F$-subalgebra of $R$ (unless r = 1). Each $Re_{i}$ is isotypical 
component of $R$. Let $V$ denote one of these simple $R$-modules and $e$ be 
the corresponding primitive central idempotent in $R$. Then by Schur's lemma, 
End$_{F[G]}(V)$ is a division ring $D$, whose center $E$ contains $F$, and 
dim$_F{E}$ is finite. Let $n=$ dim$_D{V}$. Then $Re$ is isomorphic 
to $M_n({D}^o)$, $V$ is isomorphic to $(D^o)^n$, and $Re$ is isomorphic to 
the direct sum of $n$ copies of $V$.
Let $\dim_{F}{E} = \delta$, and dim$_{E}{D} = m^2$. 
Then $\dim_F{V} = {\delta}{ m^2}{n}$, and $\dim_{F}{Re} = {\delta}{ m^2}{n^2}$.

Let $L_1, L_2,\dots , L_r$ be the $F$-conjugacy classes of $G$ and $C_{1}, C_{2}, \dots, C_{s}$ be the conjugacy classes of $G$. 
Let $\chi$ be any ${F}$-character of $G$. 
Let $\chi(L_i)$ denote the common value of $\chi$ over $L_i$. 
Let $L_{i}$ be the $F$-conjugacy class of $x$.
We denote the $F$-conjugacy class of $x^{-1}$ by $L_{i}{^{-1}}$. 
For any subset $S$ of $G$, $S^*$ denotes the formal sum of elements of $S$. 
\begin{theorem}
	Let $G$ be a finite group and $F$ a field of characteristic $0$. 
	Let $(\rho, V)$ be an irreducible $F$-representation of $G$ with 
	corresponding primitive central idempotent $e$. 
	Let $\chi$ be the $F$-character corresponding to 
	$\rho$ and $\mathrm{End}_{F[G]}(V) = D$, dim$_{D}(V) = n$. 
	Then 
	$$e = \frac{n}{|G|} \sum_{i=1}^r \chi(L_i^{-1} ) L_i^*.$$
\end{theorem}

\begin{proof}
	By Theorem \ref{sch} the 
	$\overline{F}$-representation $V\otimes_{F}{\overline{F}}$ splits into 
	$\delta$ distinct irreducible representations over $\overline{F}$ of degree
	$mn$, and multiplicity of each irreducible representation is $m$. Let
	\begin{align*}
	V\otimes_{F}{\overline{F}} = \oplus^{m}_{j = 1}\oplus^{\delta}_{i = 1}U_{i,j}
	\end{align*}
	be the direct sum decomposition into irreducible $\overline{F}$-representations 
	$G$ over $\overline{F}$. Let $\chi_{i,j}$ be the $\overline{F}$-character 
	corresponding to the representation $U_{i,j}$, 
	where $i\in \{1,2,\ldots,$ ${\delta}$ $\}$, and $j\in \{1,2,\ldots,m\}$. 
	Notice that for each $i\in \{1,2,\ldots,\delta \}$, 
	$\chi_{i,1} = \chi_{i,2} =\cdots =  \chi_{i,m}$. 
	Then the primitive central idempotent corresponding to the 
	$\overline{F}$-representation $U_{i,j}$ is
	$$e_{i,j} = \frac{mn}{|G|} \sum_{k = 1}^{s} {\chi_{i,j}(C_{{k}}^{-1})}C_{{k}}^*.$$
	Notice that for each $i\in \{1,2, \ldots,\delta \}$, 
	$e_{i,1} = e_{i,2} =\cdots= e_{i,m}$. One can show the primitive central idempotent 
	corresponding to $(\rho, V)$ (see \cite{Yam80}) is
	$$e = e_{1,1} + e_{2,1} + \cdots +  e_{\delta,1}.$$
	This implies that
	$$e = 
	\frac{mn}{|G|} \sum_{k = 1}^{s} \Big{(}\chi_{1,1} + \chi_{2,1} 
	+\cdots +\chi_{\delta,1} \Big{)} (C_{{k}}^{-1})C_{{k}}^* = 
	\frac{n}{|G|} \sum_{i = 1}^{r}{\chi(L_{{i}}^{-1})}L_{{i}}^*.$$
	This completes the proof.
\end{proof}

\begin{corollary}
	Let $e'' = \sum _{x \in G} {\chi(x^{-1})x}$. 
	Then $e'' = \{|G|/n\} e$, and 
	${e''}^2 = \{|G|/n\}^2 e^2 = \{|G|/n\}^2 e = \{|G|/n\} e''$. 
	So if we know $\chi$, then one can determine $n$. 
	Therefore one can read the complete set of primitive central idempotents 
	of $F[G]$ from the $F$-character table.
\end{corollary}

\begin{remark}
	From the above theorem we see that a primitive central idempotent in 
	$F[G]$ is an $F$-linear combinations of $F$-conjugacy class sums, and so is any central idempotent. 
\end{remark}
\chapter{Clifford theory}
In this chapter, we give a brief 
introduction to Clifford theory (see \cite{alg111}, \cite{kar112}, \cite{combi91}, \cite{Dor94}) on group representations and its 
applications on representations of finite solvable groups. 
In $1937$, Alfred H. Clifford describes the relation between representations 
of a finite group and those of a normal subgroup (see \cite{cl50}). 
The Clifford's decomposition theorem (see \cite{musi104}, Theorem $4.4.1$, \cite{Dor94}, Theorem $14.1$) 
states that the restriction of any irreducible representation 
over an arbitrary field of a finite group, to a normal subgroup, 
is a multiple of the direct sum of all conjugates of an irreducible 
representation of the normal subgroup in the whole group. As an application
of this theorem, we construct the inequivalent irreducible representations of
a finite group with a normal subgroup of prime index which helps us to inductively construct the irreducible representations of a finite solvable group. 

\section{Clifford's decomposition theorem}
Let $(\rho, V)$ be an irreducible representation of $G$. 
Let $H$ be a subgroup of $G$. 
In general, very little can be said about 
the restriction $\rho \downarrow^{G}_{H}$. The situation is quite 
different if $H$ is normal in $G$. If $H$ is a normal subgroup of $G$, 
then the Clifford's decomposition theorem 
gives a description of decomposition of 
$\rho \downarrow^{G}_{H}$ into irreducible representations of $H$.

\begin{definition}\rm
	Let $H$ be a subgroup of $G$. Let $(\eta, W)$ be a representation of $H$.
	For $g \in G$, let $H^{g} = gHg^{-1}$ be the $g$-conjugate of $H$. 
	The representation $(\eta^{g}, W^{g})$ of $H^{g}$ on the same vector space 
	$W= W^{g}$, defined by $\eta^{g}(ghg^{-1}) = \eta(h)$ for all $h \in H$ 
	is called {\bf{conjugate representation}} of $H$ associated to $g$.
\end{definition}
\begin{proposition}[\cite{musi104}, Theorem $4.3.5$]\label{clifford}
	Let $G$ be a group, and $H$ a subgroup of $G$. Let $(\eta, W)$ be a 
	representation of $H$.
	Let $N_{G}(H)$ denote the normaliser of $H$ in $G$. Then the set $I_{W} = I_{\eta} = \{g \in N_{G}(H)\,|\,
	( \eta^{g}, W^{g}) = (\eta, W )\}$ is a subgroup containing $H$, 
	called the {\bf inertia group} of the representation $(\eta, W)$.
\end{proposition}
\begin{proposition}[\cite{musi104}, Proposition $4.3.6$]
	Let $G$ be a group and $H$ be a subgroup of $G$. Let $(\eta, W)$ be 
	a representation of $H$. Then the representations 
	induced by all the conjugates of $(\eta, W)$ are equivalent.
\end{proposition}
Let $H$ be a normal subgroup of $G$. For a representation $(\eta, W )$ of 
$H$, let $I_{\eta}$ be the inertia group of $(\eta, W )$. 
Let $\{g_{1} = 1, g_{2}, \ldots , g_{m}\}$ be a complete set of left 
coset representatives of $I_{\eta}$ in $G$. Note that $\{(\eta^{g_{i}}, 
W^{g_i})\}_{i = 1}^m$ is a complete set of 
mutually inequivalent conjugate representations of $(\eta, W)$. Each element 
$g \in G$ permutes the conjugates under the isomorphism
$(\eta^{g_i})^g \cong \eta^{gg_i} \cong \eta^{g_{j(i)}y_i} 
\cong \eta^{g_{j(i)}}, \,\, 
\mathrm{where} \,\, gg_i = g_j(i)y_{i},\,\, y_{i} \in I_{\eta}.$

\begin{theorem} [\cite{musi104}, Theorem $4.4.1$, \cite{Dor94}, Theorem $14.1$]\label{clidecom}
	(Clifford's decomposition theorem)
	Let $F$ be a field of arbitrary characteristic. Let $(\rho, V)$ 
	be an irreducible representation of $G$. Suppose $(\eta, W)$ is an 
	irreducible representation of $H$ in $\rho \downarrow_{H}^{G}$. 
	Then we have the following.
	\begin{enumerate}
		\item $\rho \downarrow_{H}^{G}$ is the direct sum of conjugates 
		of $\eta$. Up to equivalence, each conjugate occurs with the same 
		multiplicity, say $e$.
		\item Let $(\rho_{i}, V_{i})$ be the isotypical component of 
		type $(\eta^{g_{i}}, W^{g_{i}})$. Then
		$$\rho_{i} = \underbrace{\eta^{g_{i}} \oplus \dots \oplus \eta^{g_{i}}}_{e} \, 
		\mathrm{and} \, \rho\downarrow_{N}^{G} = \bigoplus_{i = 1}^{m}
		\rho_{i} = (\bigoplus_{i = 1}^{m}{\eta^{g_{i}}})^{e}.$$
		\item $G$ transitively permutes the isotypical components of 
		$\rho \downarrow_{N}^{G}$.
		\item The isotypical component $(\rho_{1}, V_{1})$ of type $(\eta, W)$ 
		is a $F[I_{\eta}]$-module and 
		$\rho$ is equivalent to $\rho_{1}\uparrow_{I_{\eta}}^{G}$.
	\end{enumerate}
\end{theorem}
\begin{corollary}[see \cite{musi104}, Corollary $4.4.2$]
	The restriction of an irreducible representation $(\rho, V)$ of $G$ 
	to $H$ is isotypical if and only if the inertia group of any irreducible 
	component of $(\rho, V)$ is $G$.
\end{corollary}
\begin{corollary}[see \cite{musi104}, Corollary $4.4.3$]
	Let $(\rho, V)$ be an irreducible representation of $G$. 
	Then we have the following.
	\begin{enumerate}
		\item[(1)] Either $\rho \downarrow_{H}^{G}$ is isotypical,
		\item[(2)] or else there exists a subgroup $N$ with 
		$H \subseteq N \neq G$ and an irreducible representation 
		$(\eta, W)$ of $N$ such that $\rho = \eta \uparrow_{N}^{G}$.
	\end{enumerate}
\end{corollary}
\section{Representations of finite solvable groups}
Let $G$ be a finite group. Let $F$ be the field of complex numbers.
We denote the set of  conjugacy classes in $G$ by 
$\mathcal C_G$ and the set of irreducible representations of $G$ by 
$\Omega_{G}$. For an element $g$ in $G$,
we denote by $C_{G}(g)$ the $G$-conjugacy class of $g$. 
Let $H$ be normal subgroup of $G$. If $h$ in $H$, let $C_{H}(h)$ 
denotes the $H$-conjugacy class of $h$. Then $G$ starts acting on 
$\mathcal C_H$ by conjugation.
For $h \in H$, and $g$ in $G$, the $g$-action on $\mathcal C_H$ 
maps $C_{H}(h)$ onto $C_{H}(ghg^{-1})$.
In fact, $C_G(h)$ is a union of the $H$-conjugacy classes
$C_{H}(ghg^{-1})$ for all $g$ in $G$. In the $G$-action on $\mathcal C_H$, 
$H$ itself acts 
trivially so we have actually a $G/H$-action on $\mathcal C_H$; 
the $G$-conjugacy leaves the complement $G-H$ of $H$ also invariant, 
and so $G-H$ is  also a union of certain $G$-conjugacy classes.
As remarked earlier, for a finite group $G$, $\Omega_G$ and 
$\mathcal C_G$ contain the same number of elements,
and these two sets are dual to each other. When $H$ is a 
normal subgroup of $G$, then $G$ also starts acting on $\Omega_H$ 
by conjugation. Let $\eta$ be in $\Omega_H$, and $g$ in $G$. Then 
$g\cdot \eta(h)=\eta(ghg^{-1})$ is also a representation of $H$, 
which is conjugate representation conjugating by $g$ in $G$.
This action of $G$ restricted to $H$ is trivial, so we actually 
have the action of $G/H$. A remarkable fact is that the $G/H$-actions on 
$\Omega_H$ and 
$\mathcal C_H$ have the same number of orbits (see \cite{combi91}, Corollary$3.6.20$).
\subsection{Index-$p$ theorem}
\begin{definition}\rm
	Let $G$ be a group, and $H$ a subgroup of $G$. Let $(\eta, W)$ be a 
	representation of $H$. A representation $(\rho, W)$ of $G$ is said to be an 
	{\it{extension}} of $(\eta, W)$ if $\rho \downarrow_{H}^{G} = \eta$.
\end{definition}

The following theorem is well known, but for the sake of completeness we give a proof of the theorem.
\begin{theorem}[$\cite{alg111}$, Theorem $13.52$] (Index-$p$ theorem)\label{index}
	Let $G$ be a group, $H$ a normal subgroup of index $p$, 
	$p$ a prime and $\eta$ an irreducible representation of $H$.
	\begin{enumerate}
		\item If the $G/H$-orbit of $\eta$ is a singleton, then $\eta$ extends
		to $p$ mutually inequivalent representations 
		$\rho_1, \rho_2, \ldots, \rho_p$  of $G$.
		\item If the $G/H$-orbit of $\eta$ consists of $p$ points 
		$\eta = \eta_1, \eta_2,\ldots, \eta_p$
		then the induced representations $\eta_1\uparrow_H^G$, 
		$\eta_2\uparrow_H^G,\ldots, \eta_p\uparrow_H^G$, are equivalent,
		say $\rho$ and $\rho$ is irreducible.
	\end{enumerate}
\end{theorem}
\begin{proof}
	Since $H$ is a normal subgroup of index $p$, then 
	$G/H = \langle { xH } \rangle$, for some $ x \in G $. 
	Let $(\rho, V)$ be an irreducible representation of $G$, 
	and $(\eta, W)$ be an irreducible component of 
	$\rho \downarrow^{G}_{H}$.
	
	We first prove second statement of the theorem. We assume that the $G/H$-orbit of $\eta$ consists of $p$ points, 
	i.e., 
	$(\eta, W) \ncong (\eta^x, W^x)$.
	So all the irreducible representations 
	$\{(\eta^{x^i},W^{x^i})\}_{i=0}^{i=p-1}$ are pairwise inequivalent. 
	Therefore $(\eta, W)$ is irreducible representation to $H$ and is not 
	equivalent to $(\eta^s, W^s)$ for any
	$s\in {G - H}$, hence by the Mackey irreducibility criterion 
	(see \cite{ser118}, Chapter$7$, Proposition$23$)
	$\eta \uparrow_{H}^{G}$ is irreducible.
	By Frobenius reciprocity formula (see \cite{ser118}, Chapter$7$, Theorem$13$),
	$ \rho  = \eta \uparrow_{H}^{G}$ and also
	\begin{center}
		$\rho  \cong \eta \uparrow _{H}^{G} \cong \eta^x \uparrow _{H}^{G} \cong 
		\dots \cong \eta^{x^{p-1}}\uparrow _{H}^{G}.$
	\end{center}

Now we assume, the $G/H$-orbit of $(\eta, W )$ is singleton, that is,  
	$(\eta, W) \cong (\eta^x, W^x)$.
	Then $\{(\eta^{x^i},W^{x^i})\}_{i=0}^{i=p-1}$ are all equivalent 
	$H$-representations. As $(\rho, V)$  is an
	irreducible representation of $G$, then $V$ as $G$-representation 
	equal to  $W + xW + \cdots +x^{p-1}W$. Now the representation  on 
	$x^iW$ is equivalent to $\eta^{x^i}$ as $H$-representation. So by 
	our hypothesis the restriction of the $(\rho, V)$ to $H$ i.e., 
	$V_H$ is a direct sum of 1 or $p$ copies of $W$. Again by the 
	Frobenius Reciprocity formula, $V_H$ cannot be a direct sum of 
	$p$ copies of $W$. So $\rho$ is an extension of $\eta$. 
	In this case there are at most $p$ extensions of  an irreducible 
	representation of $H$, since all these extensions must occur as 
	sub-representations of the representation of $G$ induced from the  
	representation $(\eta, W)$ of $H$.
	
	To complete the proof, we need to show that in the first case 
	there are exactly $p$ distinct extensions of an irreducible 
	representation of $H$. The proof is indirect. We do not 
	construct the $p$ distinct representations, but simply count 
	the total number of irreducible representations of $G$ 
	obtained in this way. 
	
	Let $\eta_i$ be an  irreducible representation of 
	$H$ such that   
	$\eta_i$  is not equivalent to $\eta_i^x$:
	Let degree of $\eta_i = n_i$. The degree of the induced 
	representation  is $pn_i$, and we know that $p$ representations of 
	the form $\eta^{x^i}$ induce the same representation of $G$. So the 
	contribution to $|G|$ coming from these representations is 
	$(1/p){p^2n_i^2} = pn_i^2$. 
	
	Let $\eta_k$ be  an irreducible representation of 
	$H$ such that   
	$\eta_k$ is  equivalent to $\eta_k^x$.
	Let degree of $\eta_k = n_k$. The degree of the corresponding 
	representation of $G$ is also $n_k$. Let there be $m_k$  inequivalent 
	extensions of $\eta_k$.
	We have $|H| =  \sum n_i^2 +  \sum n_k^2$,
	and $|G| = p\sum n_i^2 +  m_k\sum n_k^2. = p|H|$. We see that each 
	$m_k$ must be $p$, and there are $p$ extensions of $\eta_k$.
\end{proof}

\begin{example}\rm
	Let $G = {\rm SL}_2(3)$, $H = Q_8$, $F = \mathbb{C}$. Here the index of
	$H$ in $G$ is equal to $3$. The group $Q_{8}$ has four irreducible 
	representations of degree 1 and one irreducible representation of degree 2. 
	The trivial representation of $Q_8$ extends to three degree 1 
	representations of ${\rm SL}_2(3)$. The three non-trivial degree 1 
	representations of $Q_8$ induce the same irreducible representation of degree 3. 
	The remaining degree 2 representation of $Q_8$ extends to three distinct degree 
	2 irreducible representations of ${\rm SL}_2(3)$.
\end{example}
\begin{theorem}[$\cite{alg111}$, Theorem $13.52$]\label{ext}
	Let $G$ be a finite group and  $H$ be a normal subgroup of of $G$ such that 
	$G/H$ is a cyclic group of order $n$. 
	Let $G/H =  \langle xH \rangle$, for some $x\in G$.
	Let $(\eta, W)$ be an irreducible representation of $H$ over $\mathbb{C}$. 
	If $(\eta, W)$ is equivalent to  
	$(\eta^{x},W^{x})$, then there are exactly $n$ irreducible representations 
	$\rho_0, \rho_1, \ldots , \rho_{n-1}$ extending $\eta$ and satisfy 
	\begin{center}
		$\eta \uparrow_{H}^{G}\cong\rho_0\oplus\rho_1\oplus\cdots\oplus\rho_{n-1}$.
	\end{center}
	Moreover, if $\chi^0, \chi^1, \ldots , \chi^{n-1}$ are linear character of the 
	cyclic group $G/H$ in a suitable order,	we have 
	$\rho_k = \chi^k \otimes \rho_0$, i.e., 
	$\rho_k(g) = \chi^k(gH)\rho_0(g)$ for all $g \in G$. 
\end{theorem}
%
\begin{corollary}
	Let $G$ be a solvable group. Then $G$ has a maximal 
	subnormal series such that the successive quotients are 
	isomorphic to cyclic groups of prime order. The last but one term in 
	the maximal subnormal series is a cyclic group of prime order. 
	So all its irreducible 
	complex representations are of degree one. So starting with degree one 
	representations of subgroups, we can build all the irreducible complex 
	representations of $G$ by the processes of  extension and induction.
\end{corollary}

\begin{definition}\rm
	Let $G$ be a finite group. Let $H$ be a subgroup of $G$ and of index $n$. 
	Let $\eta$ be an $F$-representation of $H$.
	An $F$-representation $\rho$ of $G$ is said to be a 
	{\it generalized extension} of $\eta$ if $\rho \downarrow_{H}^{G} = 
	\underbrace{ \eta \oplus \cdots \oplus \eta }_{m \, \mathrm{copies}}$, where $1 < m \leq n$.
\end{definition}

\begin{theorem}
	Let $G$ be a group and $H$ a normal subgroup of index $p$, a prime. 
	Let $F$ be a field of characteristic $0$ or prime to $|G|$. Let $\eta$ be an 
	irreducible $F$-representation of $H$.
	\begin{enumerate}
		\item[(1)] If $G/H$-orbit of $\eta$ is singleton, then there exists an irreducible $F$-representation of $G$, which is either an extension or a generalized extension of $\eta$.
		\item[(2)] If the $G/H$-orbit of $\eta$ consists of $p$ points $\eta = \eta_1, \eta_2, ..., \eta_p$, then the induced representations $\eta_1\uparrow_H^G$, $\eta_2\uparrow_H^G, ... , \eta_p\uparrow_H^G$, are equivalent, say $\rho$ and $\rho$ is irreducible.
	\end{enumerate}
\end{theorem}

\begin{proof}
	$(1)$ By Frobenius reciprocity theorem (see \cite{ser118}, Chapter$7$, Theorem$13$), there exists an irreducible $F$-representation $\rho$ of $G$ such that $\rho \downarrow^{G}_{H}$ contains $\eta$. By Theorem$\ref{clidecom}$, $\rho \downarrow^{G}_{H}$ contains all the conjugates of $\eta$ with the same multiplicity. Since $G/H$-orbit of $\eta$ is singleton, then each conjugate of $\eta$ is equivalent to $\eta$. So, $\rho \downarrow^{G}_{H}$ is an isotypical component of type $\eta$, that is, $\rho \downarrow_{H}^{G} = 
	{ \eta \oplus \cdots \oplus \eta }$, $m$ copies (say). Since $\rho$ is a sub-representation of $\eta\uparrow^{G}_{H}$, implies $m \leq p$. Therefore, $\rho$ is either an extension or a generalized extension of $\eta$. 
	
	$(2)$ Follows from the proof of Theorem$\ref{index}$.
\end{proof}
\begin{corollary}
	Let $G$ be a finite solvable group. Let $F$ be a field of characteristic $0$ or prime to $|G|$. Then every irreducible representation of $G$ can be obtained 
	starting from an irreducible representation of a cyclic subgroup by
	the process of successive inductions, extensions and generalized extensions.
\end{corollary}
\begin{example}\rm
	Let $G = {\rm SL}_2(3)$, $H = Q_8$, $F = \mathbb{Q}$. 
	The index of $H$ in $G$ is equal to $3$. $Q_{8}$ has 
	four $\mathbb{Q}$-representations of degree 1 and one irreducible 
	$\mathbb{Q}$-representation of degree 4. The trivial representation of 
	$Q_8$ extends to trivial representation of ${\rm SL}_2(3)$. 
	Two copies of the trivial representation of $Q_{8}$ extends to degree $2$ 
	irreducible representation of ${\rm SL}_{2}(3)$, that is, a genenalized extension 
	of degree 2. The three non-trivial degree 1 representations of $Q_8$ induce 
	the same irreducible representation of degree 3. 
	The degree 4 representation of $Q_8$ extends to a degree 4 irreducible 
	representation of ${\rm SL}_2(3)$. Two copies of the degree 4 irreducible 
	representations of $Q_{8}$ extends to degree 8 irreducible representation 
	of ${\rm SL}_2(3)$, that is, a generalized extension of degree 8.
\end{example}
\subsection{Consequences of index-$p$ theorem}
Let $G$ be a group and $H$ a normal subgroup of $G$. Let 
$\Omega_{G}$ and $\Omega_{H}$ denotes the set of irreducible complex 
representations of $G$ and $H$ respectively. 
Suppose that $G/H$ is isomorphic to a cyclic group 
$ C_p = \langle {u} \rangle $, $p$ a prime. In any $C_p$-action on 
any set $X$, the orbits are of two types. Either a point $x$ in $X$ is
fixed, i.e., its orbit is a singleton set $\{ x \}$, or
the orbit consists of $p$ points $\{x, u(x), \ldots , u^{p-1}(x)\}$.
Apply this observation to the $G/H$-action on the set $\mathcal {C}_{H}$ 
of $H$-conjugacy classes of $H$.
Let there be $k$ fixed points and $l$ orbits each consisting of $p$ 
points, and therefore the cardinality of $\mathcal {C}_{H}$ is equal 
to $k + lp.$ Again apply this observation to the $G/H$-action on the 
set $\Omega_H$ of irreducible representation of $H$.
In this case also, there are again $k$ fixed points and $l$ 
orbits each consisting of $p$ points.

By the Theorem \ref{index}, 
either every irreducible representation of $H$ 
extends $p$ mutually inequivalent irreducible representations of $G$ or 
$p$ mutually inequivalent irreducible representations induce the same 
irreducible representations of $G$.
Let $\eta_{1}, \eta_{2}, \dots , 
\eta_{k}$ be the irreducible representations of $H$, which extends to 
an irreducible representation of $G$. If $\rho_{i, 0}$ is 
an extension of 
$\eta_{i}$, $i = 1, 2, \ldots , k$, then by Theorem\ref{ext}, $\rho_{i,0} \otimes {\chi}$ 
is also an extension of $\eta_{i}$, where $\chi$ is an irreducible 
character of $G/H$, which is isomorphic to $C_{p}$. Thus there are 
$p$ extensions $\rho_{i, r},$ $r = 0, 1, \ldots , p - 1$ of each 
$\eta_{i}$.
Let $\eta_{j, 1}, \eta_{j,2}, \ldots , \eta_{j, p}$, 
$ l = 1, 2, \ldots , l$, be the irreducible $H$-representations 
so that $\eta_{j,1}, \eta_{j,2}, \ldots , \eta_{j,p}$ induce 
the same irreducible $G$-representation $\rho_{j}$. Thus there are 
$k + lp$ irreducible $H$-representations, and $kp + l$ irreducible 
$G$-representations.

Let $s_{1}, s_{2}, \ldots , s_{k}$ be the $H$-conjugacy class 
representatives in $H$, which are invariant under conjugation by $G$, 
i.e., $C_{H}(s_{i}) = C_{G}(s_{i}), i = 1,2, \ldots , k$. 
Let $t_{j,1}, t_{j,2}, \ldots , t_{j,p}$ ($j= 1,2, \ldots , l$) be 
the representatives of distinct $H$-conjugacy classes in $H$, 
so that $t_{j,1}, t_{j,2}, \ldots , t_{j,p}$ are conjugate in $G$. 
Thus there are $k + l$ many $G$-conjugacy classes in $H$, and 
since the number of conjugacy classes in $G$ is equals to the 
number of irreducible representations of $G$, therefore there 
are $(kp + l)-(k + l) = (p-1)k$ many $G$-conjugacy classes in 
$G - H$. This discussion can be summarized in 
the following remarks.
\begin{remark}
	Let $G$ be a group, and $H$ a normal subgroup of index $p$, a prime. 
	Let $k$ denotes the number of fixed points, and $l$ denotes the number 
	of orbits of size $p$ under the action of $G/H$ on $\mathcal{C}_H$ by 
	conjugation. Then
	\begin{enumerate}
		\item[(1)] $|\Omega_H| = k + lp$,
		\item[(2)] $|\Omega_G| = pk + l$,
		\item[(3)] The number of $G$-conjugacy classes in $G-H$ is $(p-1)k$.
	\end{enumerate}
\end{remark}
\chapter{Diagonal subalgebra}
In this chapter, we find a convenient set of generators of a "diagonal subalgebra" (unique up to conjugacy) of complex group algebra of a finite solvable group in terms of a long system of generators. A finite solvable group 
always has a subnormal series of which successive quotients are cyclic groups of prime order. 
By Theorem $\ref{index}$, such a maximal subnormal series is a 
multiplicity free chain of subgroups over $\mathbb{C}$, so one can  
consider the "Gelfand-Tsetlin algebra" (see \cite{Mu106}, Section $2$) associated with that. We find a 
system of generators of the "Gelfand-Tsetlin algebra", which is a "diagonal subalgebra", in terms of a long system of generators.
Besides that, for a finite solvable group $G$, a deep theorem due to S.D. Berman gives an inductive construction of the primitive central idempotents of $\mathbb{C}[G]$, we give a proof of this theorem. To my knowledge, this proof in not available in literatures. 
\section{Diagonal subalgebra of a group algebra}\label{diag}
Let $R$ be a ring with unity $1$.
A nonzero element $x$ in $R$ is called an {\it{idempotent}} if $x^{2} = x$. 
An element $x$ is called {\it{central}} if it belongs to center of $R$.
An idempotent $x$ in $R$ is said to be {\it primitive} if it cannot be written
as $x = x_1 + x_2$, where $x_1, x_2$ are idempotents such that 
$x_1x_2 = x_2x_1 = 0$.
In this case, $x_1$ and $x_2$ are called {\it orthogonal} idempotents.
For example, if $x \not= 1$ is an idempotent, then $1-x$ is also an idempotent,
and they are orthogonal to each other.
A central idempotent $x$ in $R$ is said to be {\it centrally primitive} if it
cannot be written
as $x = x_1 + x_2$, where $x_1, x_2$ are central idempotents such that 
$x_1x_2 = x_2x_1 = 0$.\\
\noindent
{\bf{Question:}} Let $F$ be an algebraically closed field  
with characteristic zero.
Let $V$ be an $n$-dimensional vector space over ${F}$ and  
$\mathrm{End}V$ be the associative ring of endomorphisms of $V$.
What are all the central idempotents of End$(V)$?
\begin{proposition}\label{identity}
	The idempotents in $\mathrm{End}\,V$ are projection operators, and 
	the only central idempotent is identity.
\end{proposition}
\begin{proof} 
	Let $A: V \rightarrow V$ be a non-zero and non-identity 
	idempotent operator, so $A^2 = A$. Then 
	the minimal  polynomial $m_A(x)$ of $A$ divides $x^2-x=x(x-1)$. 
	The minimal polynomial can not be $x$ or  $x-1$ since $A$ is neither 
	$0$ nor identity. Thus $m_A(x)=x(x-1)$, which has distinct roots. 
	Thus $V$ is direct sum of eigenspaces of $A$ with eigenvalues $0$ and $1$, 
	and these eigenspaces are precisely $\ker A$ and $Im(A)=A(V)$. It 
	follows that $A$ is projection onto $Im(A)$ along $\ker A$. 
	
	On the other hand, since the only central elements in ${\rm End}(V)$ are 
	the scalar transformations, it follows that the only central idempotent is 
	identity. This completes the proof.
\end{proof}
Note  that for any  decomposition of $V$ as a direct sum of two subspaces 
$W \oplus W',$
an element $\vec v$ in $V$ can be written uniquely as 
$\vec v = \vec w + \vec w'$, where $\vec w \in W$ and $\vec w' \in W'.$ 
Moreover
$\vec v \rightarrow \vec w$ is simply the {\it projection operator} $p_W$ 
of $V$ onto the subspace $W$. Choose a basis $\underline{e}$ of $W,$ and a
basis
$\underline{e'}$ of $W'.$ Then $\underline{f} = (\underline{e},  
\underline{e'})$ is a basis of $V$, and the matrix
of $p_W$ w.r.t. this basis is
\begin{equation*}
[p_W] =
\begin{bmatrix}
I_m & 0\\
0 & 0
\end{bmatrix},
\end{equation*}
where dim$W$ is $m$.
Thus a  projection operator is an idempotent in $\mathrm{End}\,V$.
Conversely, from the matrix representation it follows that
every idempotent in $\mathrm{End}\,V$ arises in this way.
More precisely, the set of idempotents in $\mathrm{End}\,V$ is in a natural 
1-1 correspondence with the set
of ordered direct sum decompositions into two subspaces of $V$.

Let $R$ be ring with unity $1$. If $x \neq 1$ is an idempotent, 
then $1-x$ is also an idempotent and $x$, $1-x$ are orthogonal to each other.
Notice that if $x$ is a central idempotent, then $Rx$  is a two-sided ideal. 
In the ring $Rx$ the element $x$ acts as the identity element. If $x \not= 1$ 
is a central idempotent in $R$, then $R = Rx \oplus R(1-x)$ is a direct sum of 
rings.
A slight generalization: Let
\begin{equation*}
1 = e_1 + e_2 + \cdots + e_k,\tag{*}
\end{equation*}
where $e_i$'s are central, pairwise orthogonal, idempotents. Then
\begin{equation*}
R = Re_1 \oplus  Re_2 \oplus \cdots \oplus Re_k
\end{equation*}
is a direct sum of rings. Note that $Re_i$'s are two sided ideals in $R$.
A decomposition of $1$ as in $(*)$ is called a {\it{\bf partition of unity}}, 
to borrow a phrase from Differential Topology.
It is easy to see that the set of partitions of unity is in a natural 1-1 
correspondence with
the set of decompositions of $R$ as a direct sum of rings. We call an expression of $1$ as in $(*)$, where $e_i$'s are pairwise orthogonal idempotents, not necessarily central, a {\it{\bf weak partition of unity}}.
Evidently, a weak partition of unity gives rise to a direct sum decomposition into left-ideals $Re_i$'s.

Now suppose $R$ is a finite-dimensional ${F}$-algebra. Then the left ideals are 
subspaces.
So for any proper direct sum decomposition of $R$ into left ideals, the
dimensions of the summands must go down. So a maximal direct sum decomposition 
of $R$
into left ideals exists. So the maximal direct sum decomposition cannot be 
properly refined further, and
the corresponding idempotents are actually primitive.

Now we apply these considerations to the case $R = \mathrm{End}\,V$, where $V$ is an $n$-dimensional space.
In this case, $I_{V}$ is the only central idempotent (see Theorem \ref{identity}). So a proper direct 
sum decomposition of $R$ into
two-sided ideals does not exist. However, a weak partition of unity, where 
the idempotents are primitive,
exists. Moreover, the primitive idempotents simply correspond to projection 
operators
onto one dimensional subspaces. So a maximal direct sum decomposition of 
$R$ into left ideals has exactly $n$ summands.

Write a maximal weak partition of unity of the identity operator,
$ I_{V} = f_1  + f_2 + \ldots + f_n$, where $f_i$'s are primitive idempotents, 
which correspond to projections operators
onto $1$-dimensional subspaces of $V$. The  ${F}$-algebra $\mathcal D = 
\langle f_1, f_2, \ldots, f_n\rangle$
is called a {\it \bf diagonal subalgebra} of $\mathrm{End}\,V$. The set 
$\{f_1, f_2, \ldots, f_n\}$ is a basis of
$\mathcal D.$ Although its cardinality is $n$, it should not be confused 
with a basis of $V$! In fact, if $\{e_1, e_2, \ldots, e_n\}$ is  basis of $V$, and $W_i = 
\langle e_i\rangle$, the $1$-dimensional subspace generated by
$e_i$, then the projection operators $p_{W_i}$ form a basis of a diagonal 
subalgebra.
In geometric terms, we may call $p_{W_i}$'s, the  {\it \bf basis-directions}. Thus the set of basis directions 
is in a 1-1 correspondence
with the set of diagonal subalgebras of $\mathrm{End}\,V$. 
Note that
$\{e_1, e_2, \ldots, e_n\}$ and $\{u_1e_1, u_2e_2, \ldots, u_ne_n\},$ 
where $u_i$'s are non-zero elements of ${F}$
correspond to the same basis-directions.

Note that the group $\mathrm{GL}(V)$ acts on $V$, on $\mathrm{End}\,V$,
on the set of bases of $V$, as well as on the set of
basis-directions of $V$. This action is simply transitive on the set of bases of $V$, whereas it is only  transitive on the set of basis-directions of $V$.
In  particular, two diagonal subalgebras are not only isomorphic, but they are isomorphic by a $\mathrm{GL}(V)$- conjugacy! So up to $\mathrm{GL}(V)$-conjugacy there
is a {\it unique} diagonal subalgebra of $\mathrm{End}\,V$.

From a view point of algebraic groups, a diagonal subalgebra of 
$\mathrm{End}\,V$ may be characterized as a
maximal abelian  diagonalizable subalgebra of $\mathrm{End}\,V$. 
In this subalgebra each operator has $n$ linearly independent
eigenvectors. This maximal abelian subalgebra of $\mathrm{End}\,V$ is unique 
up to $\mathrm{GL}(V)$-conjugacy. A curious reverse question:
Can we recover $V$ from $\mathrm{End}\,V$?
That is, if we consider $R =\mathrm{End}\,V$ as an abstract ring, 
can we recover $V$?
Before answering this question, we need to ask what does the question 
mean?
Starting with $R$, we have a category of $R$-modules, which include the 
left ideals of $R.$
So the question can be made precise
as an open-ended question:  \textit{characterize $V$ among $R$-modules}.
The Skolem-Noether theorem asserts that $\mathrm{GL}(V)$ can be recovered as
the group of ring automorphisms of the ring $R.$ The answer to the original
question is: Yes.
There is a unique minimal $R$-module up
to isomorphism; that module  is $V$! It can also be actually realized as 
a minimal left 
ideal in $R$.
\section{Representation theory from the view point of group rings}
Let $G$ be a group of order $N$ and $k$ be the number of conjugacy classes 
in $G$. Let $F$ be an algebraically closed field of characteristic zero. 
Then by Frobenius theorem $k$ is also the number of 
irreducible representations $(\rho_i, V_i)$ of degree $d_{i}$, where $i$ 
runs over the set $\{1,  2, \dots , k\}$. Let $\chi_{i}$ be the 
character corresponding to the representation $(\rho_i, V_i)$. 

A representation $\rho: G \rightarrow {\rm GL}(V)$ canonically extends 
to an ${F}$-algebra homomorphism  
$\rho_*: {F}[G] \rightarrow \mathrm{End}\,V$, making $V$ an ${F}[G]$-module. 
This makes the representation theory of $G$ categorically isomorphic to the 
theory of ${F}[G]$-modules. 
The characters extend by linearity as ${F}$-valued functions on ${F}[G]$. 
If $\chi: G \rightarrow {F}$ is the 
character of a representation, we denote its extension by
$\chi_*: {F}[G] \rightarrow {F}$.
For $g$ in $G$, let $C(g)$ denotes the conjugacy class of $g$ in $G$.
We shall also write $C_G(g)$ for $C(g)$
when we wish to note $G$ specifically.
Let $Z$ denotes the center of $F[G]$. There is a good
understanding of 
the subalgebra $Z$ in terms the internal structure of $G$,
particularly the conjugacy classes in $G$. Note that for any $h$  
in $G$, $hC(g)h^{-1} = C(g),$ so $hC(g) = C(g)h$. Let $\overline{C}(g)$ 
denote an element of 
$F[G]$ which is a formal sum of elements in $C(g)$. So 
$h\overline{C}(g) = \overline{C}(g)h$. That is 
$\overline{C}(g)$ belongs to the center $Z$ of $F[G]$. 
Converse is also true.
Let $g_i$, $i = 1, 2,\ldots, k$ denotes 
the conjugacy class representatives in $G$. Since for $i\neq j$, 
$C(g_i) \cap C(g_j)$ is empty, we see that 
$\overline{C}(g_i)$'s are linearly independent elements in $Z$. 
So $\overline{C}(g_i)$'s from a basis of $Z.$

A major reinterpretation of the Frobenius theory is that
${F}[G]$ is isomorphic to a direct sum of $\mathrm{End}(V_i)$, 
$i = 1,  2,\ldots, k$.
Let $e_i$ be an element in ${F}[G]$ which corresponds to
the identity elements of $\mathrm{End}(V_i)$. A beautiful formula 
relates $e_i$ to $\chi_i$: $e_i = \frac{d_i}{|G|} \sum_g \,\,\chi_i(g^{-1})g$.
These $e_i$'s are all the primitive central idempotents in ${F}[G]$.
Let $\mathcal D_i$ be a diagonal subalgebra of $\mathrm{End}(V_i)$. Then
$\mathcal D = \bigoplus^{k}_{i = 1}\mathcal D_i$ is called a diagonal 
subalgebra of ${F}[G]$. Clearly $\mathcal D_i = \mathcal De_i$, and $\mathcal D$ is unique up to $(GL(V_{1}), \ldots , GL(V_{k})))$-conjugacy.

Let $\{f_{i1}, f_{i2}, \cdots, f_{id_i} \}$ be the unique set of basis-directions of $\mathcal{D}_i$. Then $e_i=f_{i1}+f_{i2}+\cdots + f_{id_i}$ 
is the maximal weak partition of unity in $\mathrm{End}\, V_i$. Moreover, $V_i$ can be recovered from $\mathrm{End}\,V_i$ as the minimal 
left $R$-ideal $Rf_{ij}$, for any choice of $j$. The passage from Frobenius to Young is the same as 
\begin{enumerate}
	\item identifying in $R$ the objects corresponding to $f_{ij}$, 
	\item choosing a basis in $Rf_{ij}\cong V_i$, for any $j $,
	\item writing down the matrices of $\rho_i(x)'$s w.r.t. this basis. 
\end{enumerate}
Young did it for the Symmetric group $S_n$. He used the natural inclusions $1 = S_1 \subset S_2 \subset \cdots \subset S_n.$
We do it for any finite solvable group algorithmically, in terms of some of its subgroups.
\section{Gelfand-Tsetlin algebras} \label{gel}
We introduce {\it \bf Gelfand-Tsetlin bases} and
{\it \bf Gelfand-Tsetlin algebras} (see \cite{Mu106}, Section $2$) for an inductive 
chain of finite groups.
Let $$1= G_{0} < G_{1} < \cdots < G_{n} < \cdots$$
be an inductive chain of finite groups with simple branching, that is,
the multiplicity of the restriction of 
an irreducible representation of $G_{n}$ to $G_{n-1}$ is equal to one or
zero. Let $\Omega_{G_{n}}$ denotes the set of inequivalent finite 
dimensional complex irreducible representations of $G_{n}$. Let 
$V^{\lambda}$ denotes the irreducible $\mathbb{C}[G_{n}]$-module
corresponding to the irreducible representation $\lambda$ in 
$\Omega_{G_{n}}$. Therefore, $V^{\lambda}$ is the representation space corresponding to the irreducible representation $\lambda$.
The {\it \bf{Bratelli diagram}} of this inductive chain, whose vertices 
are the elements of the set
$$\bigcup_{n \geq 0}{\Omega_{G_{n}}}$$
and two vertices $\mu$, $\lambda$ are joined by $k$ directed edges
from $\mu$ to $\lambda$ whenever $\mu$ in $\Omega_{G_{n-1}}$ and
$\lambda$ in $\Omega_{G_{n}}$ for some $n$, and the multiplicity 
of $\mu$ in the restriction of $\lambda$ to $G_{n-1}$ is $k$. We 
write $\mu \nearrow \lambda$ if there is an edge from $\mu$ to $\lambda$.
\noindent
From now onwards, we assume that the branching multigraph defined
above is actually graph, i.e., the multiplicities of all the 
restrictions are $0$ or $1$. In this case we say that the {\it \bf{branching}} 
or {\it\bf{multiplicities}} are {\it\bf{simple}}. Since the branching is simple, 
the decomposition
$$V^{\lambda} = \bigoplus_{\mu}{V^{\mu}},$$
where the sum is over all $\mu$ in $\Omega_{G_{n-1}}$ with
$\mu \nearrow \lambda$, is a canonical decomposition. Iterating this 
decomposition, we obtain a canonical decomposition of $V^{\lambda}$
into irreducible $G_{1}$-modules
\begin{align*}
V_{\lambda} = \bigoplus_{T}{V_{T}}\tag{1},
\end{align*}
where the sum is over all possible chains
\begin{align*}
T = \lambda_{0} \nearrow \lambda_{1} \nearrow \cdots 
\nearrow \lambda_{n}, \tag{2}
\end{align*}
with $\lambda_{i} \in \Omega_{G_{i}}$ and $\lambda_{n} = \lambda$. 
By choosing a nonzero vector $v_{T}$ in each one dimensional space $V_{T}$, we obtain a basis $\{v_{T}\}$ of $V^{\lambda}$ and is called the {\it \bf{Gelfand-Tsetlin basis}} of $V^{\lambda}$. By definition of $v_{T}$, we have 
$\mathbb{C}[G_{i}]v_{T} = V^{\lambda_{i}}$, $i = 0, 1, \ldots ,n.$
Also note that chains in $(2)$ are in bijection with directed paths in the branching graph from the unique element $\lambda_{0}$ to $\lambda$.

We have identified a canonical basis up to scalars, the Gelfand-Tsetlin basis, in each irreducible irreducible representations of $G_{n}$. 
A natural question at this point is to identify those elements of 
$\mathbb{C}[G_{n}]$ that act diagonally on this basis for every
irreducible representation. In other words, consider the algebra 
isomorphism
\begin{align*}
\mathbb{C}[G_{n}]
\cong \bigoplus _{\lambda \in \Omega_{G_{n}}}{\mathrm{End}}(V^{\lambda}). 
\tag{3}\end{align*}
Let $\mathcal{D}(V^{\lambda})$ consists of all linear operators 
on $V^{\lambda}$ which are diagonalizable 
w.r.t. the Gelfand-Tsetlin basis of $V^{\lambda}$.
Let $Z_{r}$ denotes the center of the algebra $\mathbb{C}[G_{r}]$,
$r = 0, 1, \dots , n$ and set ${GZ_{n}} = \langle{\,Z_{0}, Z_{1}, 
	\ldots ,Z_{n} \,}\rangle$. Notice that ${GZ_{n}}$ is commutative 
subalgebra of $\mathbb{C}[G_{n}]$. The algebra ${GZ_{n}}$ is called the
{\it\bf{Gelfand-Tsetlin algebra}} of the inductive family of group algebras
$\mathbb{C}[G_{i}], i = 0,1, \dots , n$. We want to emphasize that the
notion of $GZ_{n}$-subalgebra of an algebra $\mathbb{C}[G_{n}]$ does depend 
on the structure of the inductive family $\mathbb{C}[G_{i}], i = 0, 1, \dots, n$,
and not only $\mathbb{C}[G_{n}]$ itself; so if we choose another inductive family inside $\mathbb{C}[G_{n}]$, then $GZ_{n}$ also can change.

The following theorem (see \cite{Mu106}, Theorem $2.1$) is known, but for our purpose we state the theorem with proof.
\begin{theorem} [see \cite{Mu106}, Theorem $2.1$]\label{murli}
	The Gelfand-Tsetlin algebra ${GZ_{n}}$ is the inverse image
	of $\displaystyle{\oplus_{\lambda \in \Omega_{G_{n}}} 
		\mathcal{D}(V^{\lambda})}$ under the isomorphism $(3)$, 
	that is, ${GZ_{n}}$ consists of all elements of $\mathbb{C}[G_{n}]$ 
	which act diagonally in the Gelfand-Tsetlin basis in every irreducible 
	representation of $G_{n}$. Thus ${GZ_{n}}$ is a maximal commutative 
	subalgebra of $\mathbb{C}[G_{n}]$ and its dimension is 
	$\displaystyle{\sum_{\lambda \in \Omega_{G_{n}}}{\mathrm{dim} V^{\lambda}}}$.
\end{theorem}
\begin{proof}
	Consider a chain $T$ from $(2)$ above. Let $p_{\lambda_{i}}$ denotes 
	the primitive central idempotent corresponding to the irreducible 
	representation $\lambda_{i} \in \Omega_{G_{i}}$, $i = 0,1, \dots , n$.
	Note that each $p_{\lambda_{i}} \in Z_{i}$, $i = 0,1, \dots , n$. 
	We define $p_{T}$ by
	$$p_{T} = p_{\lambda_{0}}p_{\lambda_{1}} \dots p_{\lambda_{n}},$$
	which is an element in $GZ_{n}$.
	A little reflection shows that the image of $p_{T}$ under the 
	isomorphism $(3)$ is $(f_{\mu} : \mu \in \Omega_{G_{n}})$, where 
	$f_{\mu} = 0$, if $ \mu \neq \lambda$ and $f_{\lambda}$ is the 
	projection of $V_{T}$ with respect to the decomposition of $(1)$ of 
	$V^{\lambda}$. Note that $p_{T}$ is a primitive idempotent in 
	$\mathbb{C}[G_{n}]$.
	
	It follows that the image of $GZ_{n}$ under $(3)$ contains 
	$\oplus_{\lambda \in \Omega_{G_{n}}}{\mathcal{D}(V^{\lambda})}$, 
	which is a maximal commutative subalgebra of 
	$\oplus_{\lambda \in \Omega_{G_{n}}}\mathrm{End}{(V^{\lambda})}$. 
	Since $GZ_{n}$ is 
	commutative, and ${GZ_{n}}$ is the inverse image of 
	$\displaystyle{\oplus_{\lambda \in \Omega_{G_{n}}} \mathcal{D}(V^{\lambda})}$ 
	under 
	the isomorphism $(3)$, this implies that ${GZ_{n}}$ is a maximal commutative
	subalgebra of $\mathbb{C}[G_{n}]$, and its dimension is equal to 
	$\displaystyle{\sum_{\lambda \in \Omega_{n}}{\mathrm{dim} V^{\lambda}}}$.
\end{proof}
\begin{remark}
	The Gelfand-Tsetlin algebra $GZ_{n}$ is a diagonal 
	subalgebra of $\mathbb{C}[G_{n}]$ w.r.t. the Gelfand-Tsetlin basis in every 
	complex irreducible representation of $G_{n}$.
\end{remark}
\section{The diagonal subalgebra of a finite solvable group}
We have seen that a finite solvable group always has a subnormal series whose
successive quotients are cyclic groups of prime order and a long presentation associated with it. We find a convenient set of generators of the Gelfand-Tsetlin algebra associated with such a maximal subnormal series in terms of a long system of generators.
\subsection{Berman's theorem}
Let $G$ be a finite solvable group. Let us fix a subnormal series: 
$$1 = G_0 < G_{1} < \cdots < G_{n} = G$$
such that successive quotient groups are cyclic groups of prime order. If all the primitive central idempotents of $\mathbb{C}[G_{t}]$, 
$2 \leq t < n$ are known to us, then the primitive central of idempotents of $\mathbb{C}[G_{t + 1}]$ can be constructed starting from them. 
This construction is detailed in the following theorem  due to S. D. Berman (see \cite{berman120}).
For the sake of completeness, we give an indirect proof of the theorem, which we do not find in any literature.
\begin{theorem} [\cite{berman120}] (Berman)\label{berman} 
	Let $G$ be a finite group and $H$ be a normal subgroup of index $p$,
	a prime. 
	Let $G/H = \langle {xH} \rangle$, for some $x$ in $G$. 
	Let $(\eta, W)$ be an irreducible representation of $H$ over $\mathbb{C}$. 
	We distinguish two cases:
	\begin{itemize}
		\item[(1)] If $e_{\eta}$ is a central idempotent in $R$, 
		then $\eta$ extends to $p$ distinct irreducible representations
		$\rho_{1}, \rho_{2}, \dots , \rho_{p}$ (say) of $G$ over $\mathbb{C}$. 
		It follows that $(\overline{C}_{G}(x))^{p}{e_{\eta}} = {\lambda}{e_{\eta}}$,
		where $\lambda$ is a non-zero complex number. For each $i$, 
		$1 \leq i \leq p$,
		$${e_{\rho_{i}} = \frac{1}{p} 
			\Big{(} 1 + \zeta^ic + \zeta^{2i} c^2 +\cdots +\zeta^{i(p-1)}c^{p-1}\Big{)} 
			e_{\eta},}$$
		where $c = \frac{\overline{C}_{G}(x){e_{\eta}} }{\sqrt[p]{\lambda}}$ and $\zeta$ 
		is a primitive $p^{th}$ root of unity in $\mathbb{C}$. 
		Moreover 
		$$e_{\eta} = e_{\rho_1} + e_{\rho_2} + \cdots + e_{\rho_p}.$$ 
		
		\item[(2)] If $e_{\eta}$ is not a central idempotent in $R$,
		then all the representations
		$\eta\uparrow_H^G, \eta^{x}\uparrow_H^G, \ldots ,\eta^{x^{p-1}}\uparrow_H^G$ 
		are equivalent to an irreducible representation $\rho$ (say) of $G$ 
		over $\mathbb{C}$. In this case,
		$$e_{\rho} = e_{\eta} + e_{\eta^{x}} + \cdots + e_{\eta^{x^{p-1}}}.$$
	\end{itemize}
\end{theorem}
\begin{proof}
	$(1)$ Since $e_{\eta}$ is a central 
	idempotent in $\mathbb{C}[G]$, which implies that $\eta \cong \eta^{x}$. By Theorem $\ref{index}$, $\eta$ extends to $p$ distinct irreducible representations $\rho_{0}, \rho_{1}, \dots ,$ $\rho_{p-1}$ (say) of $G$. By Frobenius reciprocity theorem, the induced representation 
	${\eta}\uparrow^{G}_{H} \cong \rho_{0} \oplus \rho_{1} \oplus \dots \oplus \rho_{p-1}$, which implies that 
	$\mathbb{C}[G]e_{\eta} = \bigoplus^{p-1}_{i = 0} \mathbb{C}[G]e_{\rho_{i}}$ and $e_{\eta} = \sum^{p-1}_{i = 0}e_{\rho_{i}}$. Note that $\mathbb{C}[G]e_{\eta}$ is a ring with identity $e_{\eta}$. Notice that $\mathbb{C}[G]e_{\eta}$ is the direct sum of $p$ minimal two-sided ideals $\mathbb{C}[G]e_{\rho_{i}}, i = 0,1, \dots , p-1$, and therefore $\mathbb{C}[G]e_{\eta}$ contains $p$ primitive central idempotents. 
	
	Since $G/H=\langle xH\rangle$, for some $x\in G$, this implies that $x^p\in H$.
	Then one can show that $(\overline{C}_G(x))^p=x^pa$ for some 
	$a\in \mathbb{C}[H]$. It follows that $(\overline{C}_G(x))^p$ is central element in
	$\mathbb{C}[H]$. Hence, by Schur's Lemma,
	$$(\overline{C}_G(x))^pe_{\eta}=\lambda e_{\eta}, \hskip5mm \lambda_x\in
	\mathbb{C}.$$
	
	{\it Claim:} The element $x\in G-H$ can be chosen in such a
	way that $\lambda \neq 0$.
	
	To prove the claim, let $\rho$ be an extension of $\eta$, with corresponding
	character $\chi_{\rho}$ and the primitive central idempotent $e_{\rho}$.
	If $\chi_{\rho}$ vanishes on $G-H$, then
	$$e_{\rho} = \frac{1}{|G|}\sum_{g\in G}
	\chi_{\rho}(g^{-1})g = \frac{1}{p}\frac{1}{|H|}\sum_{h\in H}
	\chi_{\eta}(g^{-1})g = \frac{1}{p}e_{\eta},$$
	and this implies that $e_{\rho}$ is {\it not} an idempotent, a contradiction. Thus there exists
	$x\in G-H$ such that $\chi_{\rho}(x)\neq 0$.
	
	If $V$ is the representation
	space corresponding to $\eta$, then it is also a representation space for the
	extension $\rho$. Extend $\rho:G\rightarrow {\rm GL}(V)$ linearly  to algebra
	homomorphism  $\mathbb{C}[G]\rightarrow {\rm End}(V)$, which we still denote by
	$\rho$. Similarly, extend $\eta$ to the algebra homomorphism
	$\eta:\mathbb{C}[H]\rightarrow {\rm End}(V)$.
	Now by Schur's Lemma, $\rho(\overline{C}_G(x))=\mu I$ for
	some $\mu \in \mathbb{C}$. Taking traces of both sides, we get
	$$|C_G(x)|\chi_{\rho}(x)=\mu \deg\rho.$$
	Since $\chi_{\rho}(x)\neq 0$, we get $\mu \neq 0$. Then
	$$\lambda I = \rho(\lambda e_{\eta})=\rho((\overline{C}_G(x))^pe_{\eta})=
	\rho( (\overline{C}_G(x) e_{\eta})^p)=\mu^pI\neq 0,$$
	hence $\lambda \neq 0$. This proves the Claim.
	
	For $i=0,1, \dots, p-1$, let 
	$$f_{i} = \frac{1}{p} 
	\Big{(} 1 + \zeta^ic + \zeta^{2i} c^2 +\cdots +\zeta^{i(p-1)}c^{p-1}\Big{)}e_{\eta},$$
	where $c = \frac{\overline{C}_{G}(x){e_{\eta}} }{\sqrt[p]{\lambda}}$ and $\zeta$ 
	is a primitive $p^{th}$ root of unity in $\mathbb{C}$. It is clear that all $f_{i}$'s are non zero. Since $e_{\eta}$ and $\bar{C}_{G}(x)$ are central in $\mathbb{C}[G]$, then all $f_{i}$'s are central in $\mathbb{C}[G]$. It is clear that
	\begin{equation*}
	c^pe_{\eta}=e_{\eta}.\tag{*}
	\end{equation*}
	Suppressing the index $i$, let us write
	$$f=\frac{1}{p}(1+\zeta c + \cdots + \zeta^{(p-1)}c^{p-1})e_{\eta}.$$
Then
	\begin{align*}
	cf &= \frac{1}{p} (ce_{\eta} +\zeta c^2e_{\eta} + \cdots +
	\zeta^{(p-2)}c^{p-1}e_{\eta} + \zeta^{p-1}c^pe_{\eta})\\
	&=\frac{1}{p} (ce_{\eta} +\zeta c^2e_{\eta} + \cdots +
	\zeta^{(p-2)}c^{p-1}e_{\eta} + \zeta^{p-1}c^pe_{\eta})\\
	&=\frac{1}{p} (ce_{\eta} +\zeta c^2e_{\eta} + \cdots +
	\zeta^{(p-2)}c^{p-1}e_{\eta} + \zeta^{-1}e_{\eta}) \mbox{ (by (*))}\\
	&= \frac{\zeta^{-1}}{p}(1+\zeta c + \cdots + \zeta^{(p-1)}c^{p-1})e_{\eta}\\
	&=\zeta^{-1}f.
	\end{align*}
	Thus for $0\leq i,k<p$, we can deduce from above calculation that
	\begin{equation*}
	c^kf_i=\zeta^{-ik}f_i.\tag{**}
	\end{equation*}
	Now we show that $f_if_j=\delta_{i,j}f_j$.
	\begin{align*}
	f_if_j &=\frac{1}{p}\Big{(} \sum_{k=0}^{p-1} (\zeta^ic)^k \Big{)}e_{\eta}
	f_j\\
	&=\frac{1}{p} \Big{(}\sum_{k=0}^{p-1} \zeta^{ik} (c^kf_j)\Big{)}e_{\eta}\\
	&=\frac{1}{p} \Big{(}\sum_{k=0}^{p-1}\zeta^{ik}\zeta^{-jk} f_j
	\Big{)}e_{\eta}\hskip5mm (\mbox{by (**)})\\
	&=\frac{1}{p} \Big{(}\sum_{k=0}^{p-1}\zeta^{(i-j)k}
	\Big{)}f_je_{\eta}\\
	&=\delta_{i,j}(f_je_{\eta})=\delta_{i,j}f_j.
	\end{align*}
	We now show that $\sum^{p-1}_{i = 0} {f_{i}} = e_{\eta}$. 
	\begin{align*}
	\sum^{p-1}_{i = 0} {f_{i}} & = \sum_{i=0}^{p-1}\Big{(} \frac{1}{p}  \sum_{k=0}^{p-1} (\zeta^ic)^k \Big{)}e_{\eta}\\
	& = \frac{1}{p} 
	\Big{\{} p + \Big{(}\sum^{p-1}_{i = 0}\zeta^{i}\Big{)}c + 
	\Big{(}\sum^{p-1}_{i = 0}\zeta^{2i}\Big{)}c^{2} +  \cdots + \Big{(}\sum^{p-1}_{i = 0}\zeta^{(p-1)i}\Big{)}c^{p-1}\Big{\}}e_{\eta}\\
	& = e_{\eta}
	\end{align*}
	Notice that each $f_{{i}}$ belongs to $\mathbb{C}[G]e_{\eta}$. 
	Therefore, $\{f_{i}, i = 0,1, \dots , p-1\}$ are mutually pairwise orthogonal central idempotents of the ring $\mathbb{C}[G]e_{\eta}$, and whose sum is $e_{\eta}$. Since $\mathbb{C}[G]e_{\eta}$ contains $p$ primitive central idempotents, $\{f_{i}, i = 0,1, \dots , p-1\}$ are all the primitive central idempotents of $\mathbb{C}[G]e_{\eta}$. Hence, for each $i = 0,1, \dots , p-1$,
	$${e_{\rho_{i}} = \frac{1}{p} 
		\Big{(} 1 + \zeta^ic + \zeta^{2i} c^2 +\cdots +\zeta^{i(p-1)}c^{p-1}\Big{)} 
		e_{\eta},}$$
	where $c = \frac{\overline{C}_{G}(x){e_{\eta}} }{\sqrt[p]{\lambda}}$ and $\zeta$ 
	is a primitive $p^{th}$ root of unity in $\mathbb{C}$.  
	Moreover, 
	$$e_{\eta} = e_{\rho_0} + e_{\rho_1} + \cdots + e_{\rho_{p-1}}.$$ This completes the proof of $(1)$. 
	
	$(2)$ If $e_{\eta}$ is not a central idempotent in $\mathbb{C}[H]$, then $\eta$ is not
	equivalent to $\eta^{x}$. By Theorem \ref{index} 
	induced representations $\eta\uparrow_H^G, \eta^{x}\uparrow_H^G,
	\ldots ,\eta^{x^{p-1}}\uparrow_H^G$ 
	are all equivalent to $\rho$ (say) and $\rho$ is irreducible.
	Since the character $\chi_{\rho}$ of $\rho$ vanishes outside
	normal subgroup $H$ and $\chi_{\rho}\downarrow^{G}_{H} = \chi_{\eta} + \chi_{\eta^{x}} + \ldots + \chi_{\eta^{x^{p-1}}}$,
	we have $e_{\rho} = e_{\eta} + e_{\eta^{x}} + \cdots + e_{\eta^{x^{p-1}}}.$
	This completes the proof of $(2)$.
\end{proof}

\begin{corollary}\label{cor1}
	Let $G$ be a finite group and $H$ be a normal subgroup of prime index 
	$p$ of $G$.
	Let $Z_{G}$ and $Z_{H}$ denote the centers of $\mathbb{C}[G]$
	and $\mathbb{C}[H]$ respectively.
	Then $Z_{G} = \langle Z_{H}, \overline{C}_{G}(x) \rangle$.
\end{corollary}
\begin{proof}\rm
	Let ${\eta_{1}}, {\eta_{2}}, \dots , {\eta_{k}}$ be the inequivalent irreducible $\mathbb{C}$-representations of $G$ and $e_{\eta_{1}}, e_{\eta_{2}}, \dots , e_{\eta_{k}}$ be their corresponding primitive central idempotents of $\mathbb{C}[G]$. Then $ Z_{H} = \langle{ e_{\eta_{1}}, e_{\eta_{2}}, \ldots , e_{\eta_{k}} }\rangle$.
	By Theorem \ref{berman}, $Z_{G} = \langle e_{\eta_{1}}, \ldots , {e_{\eta_{k}}, \overline{C}_{G}(x)} \rangle$. So, $Z_{G} = \langle Z_{H}, \overline{C}_{G}(x) \rangle$.
\end{proof}
The following theorem gives a convenient set of generators of the
Gelfand-Tsetlin algebra of group algebra of a finite solvable 
group associated with a subnormal series whose successive quotients are cyclic groups of prime order.
\begin{theorem}\label{diagonal}
	Let $G$ be a finite solvable group. Let
	\begin{center}
		$\langle e \rangle =  G_0 < G_1 < \cdots < G_n = G$
	\end{center}
	be a subnormal series with quotient groups $G_{i}/G_{i-1}$ isomorphic to cyclic groups
	of prime order $p_{i}$
	and associated long presentation 
	\begin{align*}
	G=\langle x_1,\ldots,x_n \, |\, & x_i^{p_i} = w_i(x_1,\ldots,x_{i-1}), 
	x_i^{-1}x_jx_i = w_{ij}(x_1,\ldots, x_{i-1}), \,\, j<i \rangle,
	\end{align*}
	where $w_i$ and $w_{ij}$ are certain words in $x_1, x_2,\ldots, x_{i-1}$. 
	For each $i = 1, 2, \ldots, n$, let $\overline{C}_{G_{i}}(x_{i})$ 
	denotes the $G_{i}$-conjugacy class sum of $x_{i}$ in $\mathbb{C}[G_{i}]$. 
	Then
	$${GZ}_{n} = \langle\,\, {\overline{C}_{G_1}(x_1)}, 
	\,\,{\overline{C}_{G_2}(x_2)},\ldots, {\overline{C}_{G_n}(x_n)}\rangle.$$
\end{theorem}
\begin{proof}
	By definition, $GZ_{n} = \langle{Z_{1}, Z_{2}, 
		\ldots , Z_{n}}\rangle$, where $Z_{i}$ denotes the center of $\mathbb{C}[G_{i]}]$. We use mathematical induction on $n$.
	Since $G_{1}$ is a cyclic group of prime order, then $GZ_{1} =
	\mathbb{C}[G_{1}] = \langle{\overline{C}_{G_1}(x_1)}\rangle$, the
	statement is true for $n = 1$. 
	Assume that the result is true for $n-1$, that is , $ GZ_{n-1} = 
	\langle\,\, {\overline{C}_{G_1}(x_1)},\, \ldots, 
	\, {\overline{C}_{G_{n-1}}(x_{n-1})}\rangle$. 
	We need to prove that $GZ_{n} = \langle{GZ_{n-1}, 
		{\overline{C}_{G_n}(x_n)}}\rangle$. 
	It is clear that $\langle{GZ_{n-1}, {\overline{C}_{G_n}(x_n)}}\rangle
	\subseteq GZ_{n}$, as ${\overline{C}_{G_n}(x_n)}$ belongs to $Z_{n}$.
	So it suffices to show that $Z_{n} \subseteq \langle{GZ_{n-1}, 
		{\overline{C}_{G_n}(x_n)}}\rangle$. By Corollary \ref{cor1}, 
	$Z_{n} \subseteq \langle{ Z_{n-1}, \overline{C}_{G_n}(x_n)}\rangle$, 
	and hence $Z_{n} \subseteq \langle{GZ_{n-1}, 
		{\overline{C}_{G_n}(x_n)}}\rangle$. 
\end{proof}
\begin{remark}
	Let $G$ be a fnite solvable group. Let 
	$$<1> = G_0 < G_1 < \cdots < G_n = G$$
	be a subnormal series such that successive quotients are cyclic groups of prime order. Then Gelfand-Tsetlin algebra $GZ_{n}$ associated with the subnormal series is the diagonal 
	subalgebra of $\mathbb{C}[G]$ w.r.t. the Gelfand-Tsetlin basis of every irreducible $\mathbb{C}$-representation of $G$. 
\end{remark}

\begin{remark}
	The elements $ {\overline{C}_{G_1}(x_1)}, {\overline{C}_{G_2}(x_2)},
	\ldots , {\overline{C}_{G_n}(x_n)} $ are like Young-Jucys-Murphy elements
	(see \cite{ver107}) in the representation theory of symmetric groups.
\end{remark}
\chapter{Algorithmic construction of matrix representations of solvable groups}
In this chapter, we describe an algorithm to construct
the irreducible matrix representations of a finite solvable group over $\mathbb{C}$. We can construct the complete set of primitive central idempotents and the diagonal subalgebra of the complex group algebra of a finite solvable group in terms of a long system of  generators, which we have sen in the previous chapter. With the help of them, we now describe the algorithm step by step in the following way. 
\section{An algorithm for constructing irreducible complex representations of solvable groups}

\noindent
{\it\textbf{Step 1:}} Let $G$ be a finite solvable group. Fix a subnormal series
\begin{equation}\tag{*}
\langle{e}\rangle = G_{o} < G_{1} < \dots < G_{n} = G
\end{equation}
such that $G_{i}/G_{i-1}$ is isomorphic to cyclic group of order $p_{i}$, $p_{i}$ a prime.
Then $G$ has a long presentation associated with the maximal subnormal series $(*)$
\begin{align*}
\langle x_1, x_2, \dots , x_n \, |\, & x_i^{p_i} = w_i(x_1,..., x_{i-1}), x_i^{-1}x_jx_i = w_{ij}(x_1,..., x_{i-1}) \mbox{ for } j<i \rangle,
\end{align*}
where $w_i$ and $w_{ij}$ are certain words in $x_1, x_2, ..., x_{i-1}$.\\
\noindent
{\it\textbf{Step 2:}} Let $(\rho, V_{\rho})$ be an irreducible representation of $G$ over $\mathbb{C}$ and $e_{\rho}$
be the corresponding primitive central idempotent in $\mathbb{C}[G]$. 
By the Theorem\ref{berman}, we inductively construct the PCI-diagram associated with the subnormal series $(*)$. Then from the obtained PCI-diagram we get the primitive central idempotent $e_{\rho}$.
\\
\noindent
{\it\textbf{Step 3:}}
Let ${\mathcal D}_G$ be the Gelfand-Tsetlin algebra of inductive family of group algebras $\mathbb{C}[G_{i}]$, $i = 0, 1, \dots , n$, 
which is a diagonal 
subalgebra (unique up to conjugacy) of the group algebra $\mathbb{C}[G]$. Then by  Theorem\ref{diagonal} we get, 
$$ {\mathcal D}_G = \langle\,\, {\overline{C}_{G_1}(x_1)}, \,\,{\overline{C}_{G_2}(x_2)},\,\, \cdots,\,\, {\overline{C}_{G_n}(x_n)}\rangle \,\,\,\,\, 
\mathrm{as}\,\, 
\mathbb{C-}\mathrm{algebra}.$$
\\
\noindent
{\it\textbf{Step 4:}} Let $R$ be the group algebra $\mathbb{C}[G].$ Then $Re_{\rho}$ is the minimal two-sided
ideal of $R$. The algebra
${\mathcal D}_{\rho}:= {\mathcal D}_Ge_{\rho}$ is the diagonal subalgebra of $Re_{\rho}$.
\\
\noindent
{\it\textbf{Step 5:}} Let $d$ be the degree of the representation $(\rho, V_{\rho})$. The diagonal subalgebra ${\mathcal D}_{\rho}$ contains a unique system $\{f_1, f_2,\dots , f_d\}$ of primitive 
idempotents such that $f_{1} + f_{2} + \dots + f_{d} = e_{\rho}$. The unique 
system $\{f_{1}, f_{2}, \dots , f_{d}\}$ of primitive idempotents can be obtained from 
the PCI-diagram associated to the above subnormal series $(*)$. Then either of the minimal 
left-ideals $Rf_1, Rf_2,\dots , Rf_d$ serves as a model for the representation space 
$V_{\rho}$. In fact, there are $d$ number of paths from the vertex $1$ to the vertex 
$e_{\rho}$ in the PCI-diagram. Moreover, if 
$1 = e_{o}, e_{1}, \ldots , e_{n} = e_{\rho}$ are the vertices appear in one such path, 
then by the proof of Theorem $\ref{murli}$, $e_{o}e_{1} \dots e_{n}$ is the corresponding primitive idempotent, and which is a member of the unique system of primitive idempotents $\{f_{1}, f_{2}, \dots , f_{d}\}$. 
So, the minimal left-ideal $Re_{o}e_{1} \dots e_{n}$ of $R$ serves as a model for the representation space $V_{\rho}$.
\\
\noindent
{\it\textbf{Step 6:}}
We can construct a basis of $V_{\rho}$ inside $R$, and in terms of this basis, we can write 
the matrices of $\rho(x_i)'$s.
\begin{remark}
	Let $G$ be a finite group and $F$ be an algebraically closed field of characteristic zero.
	Let $V_i$, $i = 1, 2, \dots , k$
	denotes mutually non-isomorphic irreducible representations of $G$.
	Then we know 
	\begin{equation}\tag{**}
	F[G] \cong  \oplus^{k}_{i = 1} \mathrm{End}(V_i).
	\end{equation}
	Let $\mathcal D_i$ be the inverse image of the diagonal subalgebra (unique up to conjugacy) of $End (V_i)$ under the above isomorphism $(**)$. Then $\mathcal D = \oplus^{k}_{i = 1}\mathcal D_i$
	is the diagonal subalgebra (unique up to conjugacy) of the group algebra $F[G].$

Then with the help of the primitive central idempotents of $F[G]$ and an well defined set of primitive (not necessarily central) idempotents contained in the diagonal subalgebra of $F[G]$,
we get the inequivalent irreducible matrix representations of $G$. 
 \end{remark}
\chapter{Representations of finite abelian groups}
In this chapter, we give an algorithm for constructing the irreducible representations of a finite abelian group over a field of characteristic $0$ or prime to the order of the group, and a systematic way to compute the primitive central idempotents of the abelian group algebra. Besides that, for an abelian group $G$, using a long system of generators of $G$, we obtain simple expressions of the primitive central idempotents in $\mathbb{C}[G]$, the primitive central idempotents in $\mathbb{Q}[G]$ and its Wedderburn decomposition.
\section{Long presentations of abelian $p$-groups, $p$ a prime}\label{long}
Let $G$ be an abelian $p$-group of order $p^N$, where $p$ is a prime.  Then $G$ is isomorphic to the direct sum of cyclic $p$-groups, say 
$$ G\simeq  \underbrace{ (C_{p^{r_1}} \oplus \dots \oplus C_{p^{r_1}})}_{l_1} 
\oplus 
\underbrace{(C_{p^{r_2}} \oplus\dots \oplus  C_{p^{r_2}} )}_{l_2}
\oplus \dots \oplus 
\underbrace{ (C_{p^{r_m}} \oplus \dots \oplus C_{p^{r_m}})}_{l_m}, $$
where $r_1>r_2>\dots > r_m\geq 1$. Note that 
$$N = \underbrace{ (r_1 + \dots + r_1)}_{l_{1}} + \underbrace{ (r_2 + \dots  +  r_2)}_{l_{2}} + \dots  + \underbrace{( r_m + \dots + r_m )}_{l_{m}}$$
is a partition of $N$ written in decreasing order, and we may call it an {\it ordered partition} of $N$. Thus, every isomorphism class of an abelian $p$-group of order $p^{N}$ uniquely determines
an ordered partition of $N$. Conversely, every ordered partition of $N$ determines an abelian $p$-group of order $p^{N}$ up to isomorphism.

There are three parameters associated with the above decomposition of $G$: 
\begin{itemize}
	\item[(1)] the exponent of a homogenous component;
	\item[(2)] the number of terms in each homogenous component;
	\item[(3)] the number of homogenous components. 
\end{itemize}
With respect to these three parameters, we write the decomposition of $G$ as 
\begin{equation}
G = {\prod_{i = 1}^{m}\prod_{j = 1}^{l_i}{C_{p^{r_{i}}, j}}},
\end{equation}
where $C_{p^{r_i}, j}$ denotes the $j$-th factor of the {\bf{homocyclic component}} of exponent ${p^{r_i}}$.  We define a {\bf{short presentation }} of $C_{p^{r_i}, j}$ as $\langle{ y_{(r_{i},j)}| y^{p^{r_i}}_{(r_{i},j)} = 1}\rangle.$
Then a short presentation of $G$ is defined by
\begin{align*}
\langle y_{(r_{i},j)}\,|\,y^{p^{r_i}}_{(r_{i},j)} = 1,
\hskip3mm j=1,\dots,l_i, \hskip3mm i=1,\dots,m \rangle.
\end{align*}
Notice that a short presentation contains the minimum number of generators. 
We use three indices to represent a generator in a long presentation of $G$. We take a long presentation of $C_{p^{r_i}, j}$ as follows:
\begin{align*}
\Big{\langle} x_{\{(r_{i},j),r_{i}\}}, x_{\{(r_{i},j),r_{i} - 1\}}, ... , x_{\{(r_{i},j), 1\}}\, | \,
& x^{p}_{\{(r_{i},j),r_{i}\}} =x_{\{(r_{i},j),r_{i} - 1\}},\\
& x^{p}_{\{(r_{i},j),r_{i} - 1\}} = x_{\{(r_{i},j),r_{i} - 2\}},\\
& \dots ,\\
& x^{p}_{\{(r_{i},j),1\}} = 1 \Big{\rangle}.
\end{align*}
We call the set $\{x_{\{(r_{i},j),r_{i}\}}, x_{\{(r_{i},j),r_{i} - 1\}}, ... , x_{\{(r_{i},j), 1\}}\}$ a {\bf{long system of generators}} of $C_{p^{r_i}, j}$.
Therefore a {\bf{long presentation}} of $G$ is
\begin{align*}
{\Big \langle} x_{\{(r_{i},j), \hskip2mm r_{i}\}},  \hskip2mm x_{\{(r_{i},j),r_{i} - 1\}}, \dots ,x_{\{(r_{i},j), 1\}} \hskip2mm {|} \hskip2mm
& x^p_{\{(r_{i},j),r_{i}\}} = x_{\{(r_{i},j),r_{i} - 1\}},\\
& x^p_{\{(r_{i},j),r_{i} - 1\}} = x_{\{(r_{i},j),r_{i} - 2\}},\\
&\dots ,\\
& x^{p}_{\{(r_{i},j),1\}} = 1,\\
& j=1,\dots, l_i, i=1,\dots, m
\Big{\rangle}.
\end{align*}
We call $(s,j)$ the place index and $a$ the power index of the long generator $x_{\{(s,j), a\}}$. An interesting thing is that, there is an enumeration in a system of long generators with respect to the lexicographic order. A remarkable thing is that every element of $G$ can be expressed uniquely as the product of $x^{\alpha_{\{{(s,j), a}\}}}_{\{(s,j), a\}}$'s, where
$0 \leq \alpha_{\{{(s,j), a}\}} \leq p - 1$.
 \section{Primitive central idempotents of abelian $p$-groups, $p$ a prime, over algebraically closed fields}
Let $G$ be a finite abelian $p$-group, where $p$ is a prime. Let $F$ be an algebraically closed field of characteristic $0$ or prime to $p$. Then all the irreducible representations of $G$ are of degree $1$. To compute the primitive central idempotents of $F[G]$, it is sufficient to compute the primitive central idempotents for its cyclic components. 
The following two theorems give simple expressions of the primitive central idempotents in ${F}[C_{p^{n}}]$. 
\begin{theorem}
	Let $G$ be $C_{p^{n}}$. Let
	\begin{center}
		$G = \langle{x_{1}, x_{2}, \dots ,x_{n} \mid x_{1}^{p} = 1, x_{2}^{p} = x_{1}, \dots , x_{n}^{p} = x_{n - 1}}\rangle$
	\end{center}
	be the long presentation of $G$.
	Let $F$ be an algebraically closed field with characteristic $0$ or prime to $p$.
	Let $\zeta_{n}$ be a $p^{n}$th root of unity in ${F}$.
	Then every primitive central idempotent of ${{F}}[G]$ can be expressed as
	$$e_{{\zeta^{-1}_1}{x_{1}}}e_{{\zeta^{-1}_2}{x_{2}}} \dots e_{{\zeta^{-1}_n}{x_{n}}} \,\,\mathrm{with}\,\zeta^{p}_{n} = \zeta_{n - 1}, \dots, \zeta^{p}_{2} = \zeta_{1}.$$
	where $e_{X} = (1+ X + \dots + X^{p-1})/{p}$ and $X$ is an indeterminate.
\end{theorem}
\begin{proof}
	Note that every element of $G$ can be expressed uniquely as $x_{1}^{i _1}x_{2}^{i_2} \dots x_{n}^{i_n}$, where
	$i_{j}$'s takes values in $\{0, 1, \dots , p -1\}$. Let $\rho$ be an irreducible representation of $G$.
	Then $\rho(x_j)$ is a $p^j$-th root of unity, say $\zeta_{j}$, $j = 1, 2, \dots, n$. Therefore $\rho(x^{p}_{j})= \{\rho(x_{j})\}^{p} = \zeta_{j}^p = \zeta_{j - 1}$, $j = 1,2, \dots , n$. Then the primitive central idempotent corresponding to $\rho$ is
	\begin{align*}
	e_{\rho} &= \frac{1}{p^n} 
	\begin{Bmatrix} 
	\sum_{i_1=0}^{p-1} \sum_{i_2=0}^{p-1} \cdots \sum_{i_n=0}^{p-1} 
	\rho\{ (x_1^{i_1} x_2^{i_2} \cdots x_n^{i_n})^{-1} \} (x_1^{i_1} x_2^{i_2} \cdots x_n^{i_n}) 
	\end{Bmatrix}\\
	&= \frac{1}{p^n} 
	\begin{Bmatrix} 
	\sum_{i_1=0}^{p-1}  \cdots \sum_{i_{n-1}=0}^{p-1} 
	\rho\{ (x_1^{i_1} x_2^{i_2} \cdots x_{n-1}^{i_{n-1}})^{-1} \}  (x_1^{i_1} x_2^{i_2} \cdots x_{n-1}^{i_{n-1}}) 
	\end{Bmatrix}\\
	& \hspace{1.2cm}\Big{( } 
	1+\frac{x_n}{\rho(x_n)} + \cdots + \frac{x_n^{p-1}}{\rho(x_n)^{p-1}}
	\Big{ )}
	\\
	&= \frac{1}{p^{n-1}} 
	\begin{Bmatrix} 
	\sum_{i_1=0}^{p-1}  \cdots \sum_{i_{n-1}=0}^{p-1} 
	\rho\{ (x_1^{i_1} x_2^{i_2} \cdots x_{n-1}^{i_{n-1}})^{-1} \}  (x_1^{i_1} x_2^{i_2} \cdots x_{n-1}^{i_{n-1}}) 
	\end{Bmatrix} 
	e_{\zeta_n^{-1}x_n}.
	\end{align*} 
	If we keep continue this process, we get
	$$e_{\rho} = e_{{\zeta^{-1}_1}{x_{1}}}e_{{\zeta^{-1}_2}{x_{2}}} \dots e_{{\zeta^{-1}_n}{x_{n}}}, \,\,\mathrm{where}\,\zeta^{p}_{n} = \zeta_{n - 1}, \dots, \zeta^{p}_{2} = \zeta_{1}.$$
This completes the proof.
\end{proof}
\begin{theorem}
	Let $G$ be $C_{p^{n}}$, with the long presentation
	\begin{center}
		$G = \langle{x_{1}, x_{2}, \dots ,x_{n} \mid x_{1}^{p} = 1, x_{2}^{p} = x_{1}, \dots , x_{n}^{p} = x_{n - 1}}\rangle$.
	\end{center}
	Let $H$ be the unique subgroup of index $p$ of $G$ with the long presentation
	\begin{center}
		$H = \langle{x_{1}, x_{2}, \dots , x_{n - 1} \mid x_{1}^{p} = 1, x_{2}^{p} = x_{1}, \dots , x_{n - 1}^{p} = x_{n - 2}}\rangle$.
	\end{center}
	Let $F$ be an algebraically closed field of characteristic $0$ or prime to $p$.
	Let $\eta$ be an irreducible representation of $H$, and $e_{\eta}$ be its corresponding primitive central idempotent in ${F}[H]$. Let
	$e_{\eta} = e_{\zeta^{-1}_{1}x_{1}}e_{\zeta^{-1}_{2}x_{2}}\dots e_{\zeta^{-1}_{n - 1}x_{n - 1}}$, where $\zeta_{n-1}$ is a $p^{n-1}$th root of unity in $F$, with
	$\zeta^{p}_{n - 1} = \zeta_{n - 2}, \dots, \zeta^{p}_{2} = \zeta_{1}$.
	Then
	\begin{enumerate}
		\item[$1.$] $\eta$ extends to $p$ mutually inequivalent irreducible representations $\rho_{0}, \rho_{1}, \dots , \rho_{p-1}$ of $G$.
		\item[$2.$] The primitive central idempotents associated with $\rho_{i}$'s are $e_{\eta}e_{\epsilon^{-i}\zeta^{-1}_{n}x_{n}}$,
		where $i$ runs over the set $\{0,1, \dots, p - 1\}$, $\zeta_{n}$ is a fixed $p$-th root of $\zeta_{n - 1}$ and $\epsilon$ is a primitive
		$p$-th root of unity in $F$.
	\end{enumerate}
\end{theorem}
\begin{proof}
	Since $G$ is abelian, $\eta$ extends to an
	irreducible representation of $G$. Suppose that $\eta$ extends to $\rho$, which implies that $\{\rho(x_{n})\}^{p} = \rho({x_{n}}^p) = \eta({x_{n}}^p) = \eta(x_{n - 1}) =
	\zeta_{n - 1}$. Therefore, $\rho(x_{n})$ is equal to $\epsilon^{i}{\zeta_{n}}$, where $\epsilon$
	is a primitive $p$-th root of unity, $\zeta_{n}$ is a fixed $p$-th root of $\zeta_{n - 1}$, and $i$ runs over $\{0, 1, \dots , p - 1\}$. For each $i \in \{0, 1, \dots , p - 1\}$, $\rho_{i}: x_{n} \mapsto {\epsilon^{i}{\zeta_{n}}}$
	defines an extension of $\eta$ and $\rho_{i}, \rho_{j}$ are mutually inequivalent for $i \neq j$. So, $\eta$ extends in $p$ distinct ways. This completes the proof of $(1)$.
	
	Let $e_{\rho_{i}}$ denote the primitive central idempotent corresponding to $\rho_{i}$. Then for each $i \in\{0, 1, \dots , p-1\}$,
	\begin{align*}
	e_{\rho_{i}} & =  
	\frac{1}{p^n}
	\begin{Bmatrix}
	\sum_{h \in H}{\rho_i(h^{-1})h} + \dots + 
	\sum_{h \in H}   \rho_i (h^{-1})h 
	\rho_i(x_n^{-(p-1)}) x_n^{p-1} \end{Bmatrix}\\
	& = \begin{Bmatrix} \frac{1}{p^{n}}\sum_{h \in H}\rho_{i}(h^{-1})h 
	\end{Bmatrix}\begin{Bmatrix}
	1 + \frac{x_n}{\rho_i(x_n)} + \dots +
	\frac{x_n^{p - 1}}{\rho_i(x_n^{p - 1}) } \end{Bmatrix}\\
	& = \frac{e_{\eta}}{p}
	\begin{Bmatrix}
	1 + \frac{x_n}{\rho_i(x_n)} + \dots + 
	\frac{x_n^{p - 1}}{\rho_i{(x_n^{p - 1})}}
	\end{Bmatrix}\\
	& = \frac{e_{\eta}}{p}
	\begin{Bmatrix} 
	1 + \frac{x_n}{\rho_i(x_n)} + \dots + 
	\frac{x_n^{p - 1}}{\rho_i(x_n^{p-1})}
	\end{Bmatrix}
	\\
	& = e_{\eta}e_{\epsilon^{-i}{\zeta^{-1}_{n}}{x_{n}}},
	\end{align*}
	where $\zeta_{n}$ is a fixed $p$-th root of $\zeta_{n - 1}$ and $\epsilon$ is a primitive $p$-th root of unity in ${F}$.
This completes the proof of $(2)$.
\end{proof}
\begin{remark}
	In terms of characters, the expression of every primitive central idempotents of $F[C_{p^{n}}]$ has $p^n$ terms. Such an expression is a product of $n$ factors, and each factor containing $p$
	terms. So each primitive central idempotent in $F[C_{p^{n}}]$ has an expression containing $pn$ terms.
\end{remark}
\section{Wedderburn decomposition of abelian group algebras} \label{dec}
Let $\Phi_{m}(X)$ denote the $m^{th}$ cyclotomic polynomial.
Let $\zeta_{m}$ be a primitive $m^{th}$ root of unity in $\mathbb{C}$. Then the mapping
$\frac{\mathbb{Q}[X]}{\langle \Phi_{m}(X) \rangle} \longrightarrow \mathbb{Q}(\zeta_{m})$
defined by $X + \langle \Phi_{m}(X) \rangle \longmapsto \zeta_{m}$ is an isomorphism. The field $\mathbb{Q}(\zeta_{m})$
is an extension of $\mathbb{Q}$ of degree $\phi(m)$, which is  known as a {\bf{cyclotomic field}}. Since  
$X^n - 1 = \prod_{m \mid n} \Phi_m(X)$ is the factorization into irreducible polynomials over $\mathbb{Q}$, we have
$$\mathbb{Q}[C_{n}] \simeq  \frac{\mathbb{Q}[X]}{\langle X^n - 1 \rangle} = \frac{\mathbb{Q}[X]}{\displaystyle \langle \prod_{m \mid n} \Phi_m(X)\rangle} \simeq \bigoplus_{m \mid n} \frac{\mathbb{Q}[X]}{\langle \Phi_m(X)\rangle} \simeq \bigoplus_{m \mid n} \mathbb{Q}(\zeta_m).$$
\begin{theorem}[\cite{milies126}, Theorem $3.5.4$]
	Let $G$ be a finite abelian group of order $n$, and let $F$ be a field of characteristic $0$ or prime to $n$. Then 
	$$\displaystyle{F[G] \simeq \bigoplus_{d |n}f_{d}F(\zeta_{d})},$$	
	where $\zeta_{d}$ denotes a primitive root of unity of order $d$ and $f_{d} = \frac{n_{d}}{[F(\zeta_{d}):F]}$. In this formula, $n_{d}$ denotes the number of elements of order $d$ in $G$.
\end{theorem}
\begin{theorem}[\cite{milies126}, Corollary $3.5.5$] $\label{cor}$
	Let $G$ be a finite abelian group of order $n$. Then 
	$$\displaystyle{\mathbb{Q}[G] \simeq \bigoplus_{d | n}{f_{d}}{\mathbb{Q}(\zeta_{d})}},$$
	where $\zeta_{d}$ denotes a primitive root of unity of order $d$ and $f_{d}$ is the number of cyclic subgroups (or cyclic factors) of $G$.
\end{theorem}
\section{Pull back of idempotents}
Let $G$ be a finite group. Let $F$ be a field of characteristic $0$ or prime to $|G|$. For a subgroup $H$ of $G$, we define $\widehat{H} = {\sum_{h \in H}h}$ and $e_{H} = \frac{1}{|H|}{\widehat{H}}$. 
\begin{theorem}[\cite{milies126}, Proposition 3.6.7]\label{pro}
	Let $G$ be a finite group and $N$ be a normal subgroup of $G$. Let $F$ be a field of characteristic $0$ or prime to $|G|$. 
	Then 
	$e_N$ is a central idempotent of $F[G]$, $F[G]e_{N} \simeq F[G/N]$ and $F[G] = F[G]e_{N} \oplus F[G](1 - e_{N})$.
\end{theorem}
\begin{proof}
	It is clear that $e_{N}$ is an idempotent of $F[G]$. As $N$ is a normal subgroup of $G$, then for any $g \in G$ we have $g^{-1}Ng = N$; therefore $g^{-1}\widehat{N}g = \widehat{N}$. Thus $\widehat{N}g = g\widehat{N}$, for all $g \in G$, which implies that ${e_N}g = g{e_N}$, for all $g \in G$. This shows that $e_{N}$ is central in $G$. Consequently, $e_{N}F[G] = F[G]e_{N}$, which implies that $e_{N}$ is central in $F[G]$. Since we have shown that $e_{N}$ is a central idempotent in $F[G]$, it is clear that $F[G] = F[G]e_{N} \oplus F[G](1 - e_{N})$. To see that $F[G/N] = F[G]e_{N}$, we shall first show that $G/N \simeq Ge_{N}$ as groups. In fact, it is easy to see that the map $\phi: G \rightarrow Ge_{N}$ given by $g \longmapsto ge_{N}$ is a group epimorphism. Since Ker$(\phi) = H$, the result follows. As $Ge_{N}$ is a basis of $F[G]e_{N}$ over $F$, we readily have that $F[G]e_{N} \simeq F[G/N]$.
\end{proof}
\begin{definition}\rm
	Let $G$ be a finite group and $N$ be a normal subgroup of $G$. Let $F$ be a field of characteristic $0$ or prime to $|G|$. Let $\phi: F[G] \longrightarrow F[G/N]$ be the projection map induced from the natural epimorphism: $G \longrightarrow G/N$. Let $\bar{e} \in F[G/N]$ be an element. Then an element $e \in F[G]$ is called a {\bf lift} of the element $\bar{e}$ if $\phi(e) = \bar{e}$, and the element $e_{N}e \in F[G]$ is called {\bf the pull back} of $\bar{e}$.
\end{definition}
\section{Primitive central idempotents of $\mathbb{Q}[G]$, $G$ an abelian $p$-group}\label{ra}
In this section, we give simple expressions for the primitive central idempotents of rational group algebra of an abelian $p$-group, $p$ a prime, in terms of a long system of generators. 

Let $G$ be an abelian $p$-group of order $p^m$, $m \geq 1$. 
Let us consider a subnormal series
$$\langle  e  \rangle = G_{o} \leq G_1 \leq \dots \leq G_{l}\leq G_{l+1} \leq \dots \leq G_{m} = G$$ such that $[G_{i}\colon G_{i-1}] = p$. Associated with the subnormal series, we have a long presentation of $G$. 
Let $\{x_1, x_2, \dots , x_m\}$
be the associated long system of generators of $G$. 
We set  
$e_{X}= (1+ X + \dots + X^{p-1}) / {p}$ and  $e_{X}'=1-e_{X}$, where $X$ is an indeterminate.
\begin{theorem}\label{th}
	Let $G$ be an abelian $p$-group of order $p^{m}$, $m \geq 1$. 
	Then the following hold true. 
	\begin{itemize}
		\item[1.] Every non-trivial primitive central idempotent of $\mathbb{Q}[G_l]$ can be written as a product of some $e_{x}$'s, where $x \in G_{l}$ and exactly one $e_{x_j}'$ with $j \in \{1, 2, \dots,l\}$. 
		\item[2.] If $e$ is the trivial primitive central idempotent of $\mathbb{Q}[G_l]$, then it is the sum of two primitive central idempotents $ee_{x_{l+1}}$ and $ee'_{x_{l+1}}$ of $\mathbb{Q}[G_{l+1}]$.
	\end{itemize}
\end{theorem}
\begin{proof}
	(1) First assume that $G$ is a cyclic $p$-group of order $p^{m}$. Then $G_l$ is also a cyclic group, so it is sufficient to prove the result for $G$ only. Let 
	$$G = \langle{x_{1}, x_{2}, \dots, x_{m} \mid x^{p}_{1} = 1, x^{p}_{2} = x_{1}, \dots , x^{p}_{m} = x_{m - 1}}\rangle$$ 
	be the long presentation of $G$.
	By Theorem $\ref{cor}$, $\mathbb{Q}[G] \simeq \displaystyle{\bigoplus_{i = 0}^{m}{\mathbb{Q}(\zeta_{p^{i}})}}$, So $\mathbb{Q}[G]$ has
	$(m + 1)$ primitive central idempotents.
	Note that for $1\leq i\leq j$, 
	$$(e_{x_{1}}e_{x_{2}} \dots e_{x_{i}}) (e_{x_{1}}e_{x_{2}} \dots e_{x_{j}}) = (e_{x_{1}}e_{x_{2}} \dots e_{x_{j}}).$$ 
	Let $e_{o} = e_{x_{1}}e_{x_{2}} \dots e_{x_m}$. Then for $i \geq 1$, 
	$$e_{0}e_{i} = (e_{x_{1}}e_{x_{2}} \dots e_{x_m})e_{i}
	= (e_{x_{1}}e_{x_{2}}\dots e_{x_m})\{(e_{x_{1}}e_{x_{2}}\dots e_{x_{i - 1}}) - (e_{x_{1}}e_{x_{2}} \dots e_{x_{i}})\}
	= 0.$$
	Hence $e_{0}$ is orthogonal to $e_{i}.$
	For $1 \leq i < j$, we get
	\begin{align*}
	{e_i}{e_{j}} &= (e_{x_{1}}e_{x_{2}} \dots e^{'}_{x_{i}})(e_{x_{1}}e_{x_{2}} \dots e^{'}_{x_{j}})\\
	&= \{e_{x_{1}}e_{x_{2}} \dots e_{x_{i - 1}}(1 - e_{x_{i}})\}\{e_{x_{1}}e_{x_{2}} \dots e_{x_{j - 1}}(1 - e_{x_{j}})\}\\
	& = 0.
	\end{align*}
	One can check that $\sum_{i = 0 }^{m}{e_{i}} = 1$. Thus $e_{0}, e_{1}, \dots , e_{m}$ are
	$(m + 1)$ pairwise orthogonal central idempotents whose sum is $1$. So they are the complete set of primitive central idempotents of
	$\mathbb{Q}[G]$. 
	
	Let $G$ be a non-cyclic abelian $p$-group. We give arguments for $G$ which can be applied to each $G_l$. 
	Let $e$ be a non-trivial primitive central idempotent in  $\mathbb{Q}[G]$. Since $\mathbb{Q}[G]$ is isomorphic to the direct sum of fields, $\mathbb{Q}[G]e$ is a field with multiplicative identity $e$. The multiplicative group of $\mathbb{Q}[G]e$ contains the finite subgroup $Ge = \{ge \,|\, g \in G \}$, which is necessarily cyclic.
	Let $K = \{ g \in G \mid ge = e \}$, then $G/K \simeq Ge$ which is a cyclic group of order $p^{n}$ (say).  We may consider a long presentation of $G/K$:
	${G}/{K}=\langle{\bar{x}_{1}, \bar{x}_{2}, \dots , \bar{x}_{n} \mid {\bar{x}_{1}}^{p}= 1, {\bar{x}_{2}}^{p}= \bar{x}_{1},
		\dots , {\bar{x}_{n}}^{p} = \bar{x}_{n - 1}}\rangle,$
	where $\bar{x_{i}}$ is the image of $x_{i}$ in $G/K$. By Lemma $\ref{pro}$, $\mathbb{Q}[G] = \mathbb{Q}[G]{e_{K}} \oplus \mathbb{Q}[G]{(1 - e_{K})}$ and since 
	$e_{K}e = e$, it follows that $\mathbb{Q}[G]e = \mathbb{Q}[G]{e_{K}}e$ is actually a simple component of
	$\mathbb{Q}[G]{e_{K}}$. Thus by Lemma $\ref{pro}$, $e$ is also a primitive central idempotent of $\mathbb{Q}[G]{e_{K}} \simeq \mathbb{Q}[G/K]$.
	
	Let $\varphi: \mathbb{Q}[G] \rightarrow \mathbb{Q}[{G}/{K}]$ be the natural homomorphism induced by the natural map $G \rightarrow G/K$.
	Recall that $e = e_{K}e \mapsto \phi(e)$ under the isomorphism $\mathbb{Q}[G]e_{K} \rightarrow \mathbb{Q}[G/K]$. Thus $\phi(e)$ is a primitive central idempotent of $\mathbb{Q}[G/K]$.
	Write $e = \sum_{g \in G}{e(g)g}$, $e(g) \in \mathbb{Q}$. Since $e$ is the non-trivial primitive central idempotent in $\mathbb{Q}[G]$, there exists some $e(g) \neq {1 / \mid{G}\mid}$. Since $ke = e$, for any $k \in K$, it follows that $e(g) = e(kg)$, for any $g \in G$ and
	$k \in K$, that is,  coefficients of  elements of cosets of $K$ in the sum of $e$ are equal. Thus
	$e = \sum_{i = 0}^{p^{n} - 1}{\alpha_{i}{e_{K}}{x^{i}_{n}}}$,
	$\alpha_{i} \in \mathbb{Q}$ and  ${\alpha_{i}}/{\mid{K}\mid}\neq {1/\mid{G}\mid}$ for some $i$, which implies $\alpha_{i} \neq \frac{1}{|G/K|}$. Since
	$\varphi(e) = \sum_{i = 0}^{p^{n} - 1}{\alpha_{i}}{\bar{x}_{n}}^{i}$,
	we have $\varphi(e) \neq e_{{G}/{K}}$. Hence, we get 
	$\varphi(e) = e_{i} = (e_{\bar{x_{1}}}e_{\bar{x_{2}}} \dots e_{\bar{x}_{i-1}} - e_{\bar{x_{1}}}e_{\bar{x_{2}}} \dots e_{\bar{x_{i}}}),$	
	for some $i$, $1 \leq i \leq n$.
	Also, $e = e_{K}(e_{x_{1}}e_{x_{2}} \dots e_{x_{i - 1}}- e_{x_{1}}e_{x_{2}} \dots e_{x_{i}})$, which implies that
	$e = e_{K}(e_{x_{1}}e_{x_{2}} \dots e_{x_{i - 1}})( 1 - e_{x_{i}})$, for some $i$. Again, $e_{K}$ can be
	expressed as the product of a certain number of $e_{x}$'s, $x \in G$. Hence each non trivial primitive central idempotent at the $l$-th stage of
	PCI-diagram can be written as a product of some $e_{x}$'s, where $x \in G$ and exactly one $e_{x_j}'$ for some $j \in \{1,\dots,l\}$. This completes the proof of $(1)$.
	
	(2) Let $e$ be the trivial primitive central idempotent of $\mathbb{Q}[G_l]$. Clearly, $e$ is a central idempotent in $\mathbb{Q}[G_{l + 1}]$. 
	Hence $\mathbb{Q}[G_{l + 1}]e$ is a two-sided ideal in $\mathbb{Q}[G_{l + 1}]$, which is the semisimple component of the induced representation induced from the irreducible representation
	corresponding to $e$.
	The two-sided ideal $\mathbb{Q}[G_{l + 1}]e$ is the direct sum of minimal left ideals in $\mathbb{Q}[G_{l + 1}]$. It is clear that  $\mathbb{Q}[G_{l + 1}]ee_{x_{l+1}}$ is one simple component of $\mathbb{Q}[G_{l + 1}]e$.
	Since $\mathbb{Q}[G_{l + 1}]e$ is direct sum of two simple components, $\mathbb{Q}[G_{l + 1}]ee^{'}_{x_{l+1}}$ is another simple component of $\mathbb{Q}[G_{l + 1}]e$, and so $e = ee_{x_{l+1}} + ee^{'}_{x_{l+1}}$. This completes the proof of $(2)$.
\end{proof}
\begin{theorem}\label{th1}
	Let $e$ be a non-trivial primitive central idempotent of $\mathbb{Q}[G_l]$. By Theorem \ref{th}, $e$ can be expressed as the product of some $e_{x}$'s, where $x \in G_{l}$ and exactly one $e^{'}_{x_j}$, for some $j \in \{1, 2,\dots, l\}$. 
	\begin{itemize}
		\item[1.] If $x_j = x_{l+1}^{p^s}$, for some $s\geq 1$, then $e$ is a primitive central idempotent of $\mathbb{Q}[G_{l+1}]$. 
		\item[2.] If $x_j\neq x_{l+1}^{p^s}$, for any $s\geq 1$, then $e$ is the sum of $p$ distinct primitive central idempotents $ee_{x_{l+1}}, ee_{x_jx_{l+1}}, \dots, ee_{x_j^{p-1}x_{l+1}}$ of $\mathbb{Q}[G_{l+1}]$.
	\end{itemize}
\end{theorem}
\begin{proof}
	Since $e$ is a non-trivial primitive central idempotent, $e$ is pull back of the primitive central idempotent of $\mathbb{Q}[G_{l}/K]$, where $G_{l}/K$ is a cyclic group. Moreover this corresponds to the unique faithful irreducible representation of $G_{l}/K$. Since $e$ contains $e^{'}_{x_j}$ as a factor, it follows $e = e_{K}e^{'}_{x_{j}}$.
	
	$(1)$ Let $G_{l}/K \simeq C_{p^{n}}$. Since $x_{j} = x_{l + 1}^{p^s}$, we have $G_{l +1}/K \simeq C_p^{n + 1}$. Note that $e^{'}_{\bar {x_{j}}}$, where $\bar{x_{j}} = x_{j}K$, is the primitive central idempotent of $\mathbb{Q}[G_{l + 1}/K]$ corresponding to the unique faithful irreducible representation of $G_{l + 1}/K$. The pull back of the primitive central idempotent $e^{'}_{\bar{x_{j}}}$ of $\mathbb{Q}[G_{l + 1}/K]$ is $e_{K}e^{'}_{x_{j}}$. So, $e$ is a primitive central idempotent of $\mathbb{Q}[G_{l + 1}]$. This completes the proof of $(1)$.
	
	$(2)$ Note that $x_j \neq x_{l+1}^{p^s}$, for any $s\geq 1$, implies $\langle{K, x_{l+1}}\rangle, \langle{K, x_{j}x_{l+1}}\rangle, \dots , \\ \langle{K, x^{p-1}_{j}x_{l+1}}\rangle$ are $p$ distinct subgroups of $G_{l +1}$. Also for each $i$, all the $p$ quotients $G_{l + 1}/ \langle{K, x^{i}_{j}x_{l+1}}\rangle$ are isomorphic to $C_{p^{n}}$. 
	Then $e^{'}_{\bar{x_j}}$, where $\bar{x_j}=x_j{\langle{K, x_j^{i}x_{l+1}}\rangle}$, is the primitive central idempotent of $\mathbb{Q}[G_{l + 1}/{\langle{K, x_j^{i}x_{l+1}}\rangle}]$. Moreover, this is the primitive central idempotent corresponding to the unique faithful irreducible representation of $G_{l + 1}/{\langle{K, x_j^{i}x_{l+1}}\rangle}$. Therefore, the pull back of $e^{'}_{\bar{x_j}}$, where $\bar{x_j}=x_j{\langle{K, x_j^{i}x_{l+1}}\rangle}$,  is 
	$e_{\langle{K, x_j^{i}x_{l+1}}\rangle}{e^{'}_{x_j}} = {e}{e_{x_j^{i}x_{l+1}}}.$ So,      $ee_{x_j^{i}x_{l+1}}, i = 0, 1, \dots , p - 1$ are $p$ distinct primitive central idempotents of $\mathbb{Q}[G_{l + 1}]$ whose sum is equal to $e$. This completes the proof of $(2)$.  
\end{proof}
\section{Rules to draw the PCI-diagram of an abelian $p$-group}
Let $G$ be an abelian $p$-group of order $p^m$, $p$ a prime and $m \geq 1$. 
Let us fix a subnormal series
$$\langle e \rangle = G_{o} \leq G_1 \leq \dots \leq G_{l}\leq G_{l+1} \leq \dots \leq G_{m} = G$$ such that $[G_{i}\colon G_{i-1}] = p$. 
Associated with the maximal subnormal series, we always have a long presentation of $G$.  
We set 
$e_{X}= (1+ X + \dots + X^{p-1})/{p}$ and  $e_{X}'=1-e_{X}$, where $X$ is an indeterminate.
By Theorem $\ref{th}$ and Theorem $\ref{th1}$, we state four rules to draw the associated PCI-diagram of $G$ over $\mathbb{Q}$.\\\\
\noindent
{\it \textbf{Rule 1:}}
{\it Each non-trivial primitive central idempotent at any stage can be expressed in a product form and
	contains exactly one $e^{'}_{z}$ as a factor,  where $z$ is a generator of the long presentation.\\
	\noindent
	{\it \textbf{Rule 2:}}
	{\it Let $e$ be the trivial primitive central idempotent in the $l$-th stage of the PCI- diagram. Let $u$ be the new generator
		introduced in the $(l + 1)$-st stage of the long presentation of $G$. Then $e$ is adjacent to precisely two vertices
		$ee_{u}$ and $ee^{'}_{u}$ in the $(l + 1)$-st stage of the PCI-diagram.} 
	
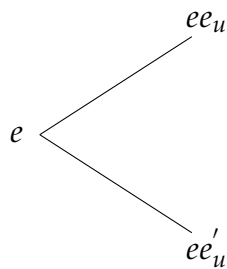
\begin{center}
		\begin{tikzpicture}
		\node at (0,0) {$e$};
		\node at (2.5,1.5) {$ee_{u}$};
		\node at (2.5,-1.5) {$ee^{'}_{u}$};
		\draw (0.3,0) -- (2.3,1.3);
		\draw (0.3,0) -- (2.3,-1.3);
		\end{tikzpicture}
	\end{center}
	\noindent
	{ \it \textbf{Rule 3:}}
	{\it Let $e$ be a non-trivial primitive central idempotent in the $l$-th stage of the PCI-diagram. Then $e$ contains precisely one
		$e^{'}_{z}$ as a factor, where $z$ is a generator of the long presentation. Let $u$ be the long generator introduced in the
		$(l + 1)$-st stage of the long presentation such that $z = u^{p^s}$, for some positive integer $s$. Then $e$ is adjacent to itself in the $(l + 1)$-st stage of
		the PCI-diagram.} 
	\begin{center}
		\begin{tikzpicture}
		\node at (0,0) {$e$};
		\node at (2.5,0) {$e$};
		\draw (0.3,0) -- (2.4,0);
		\end{tikzpicture}
	\end{center}
	\noindent
	{\it \textbf{Rule 4:}}
	{\it Let $e$ be a non-trivial primitive central idempotent in the $l$-th stage of the PCI-diagram. Then $e$ contains precisely one
		$e^{'}_{z}$ as a factor,
		where $z$ is a generator of the long presentation. Let $u$ be the long generator introduced in the $(l + 1)$-st stage of the long presentation
		such that $z \neq u^{p^s}$, for any positive integer $s$. Then $e$ is adjacent to precisely $p$ vertices $ee_{u}, ee_{zu}, \dots ,
		ee_{z^{p - 1}u}$ in the $(l + 1)$-st stage of the PCI-diagram.}
	
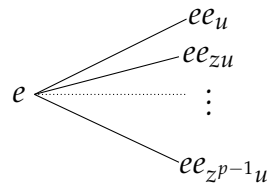
\begin{center}
		\begin{tikzpicture}
		\node at (0,0) {$e$};
		\node at (2.5,1) {$ee_{u}$};
		\node at (2.5,0.5) {$ee_{z{u}}$};
		\node at (2.5,0) {$\vdots$};
		\node at (2.7,-1) {$ee_{z^{p - 1}{u}}$};
		%
		\draw (0.2,0.0) -- (2.2, 1);
		\draw (0.2,0.0) -- (2.1,0.5);
		\draw [densely dotted] (0.2,0.0) to (2.2,0.0);
		\draw (0.2,0.00) -- (2.1,-.9);
		
		\end{tikzpicture}
	\end{center}
}
\section{Wedderburn decomposition of $\mathbb{Q}[G]$, $G$ an abelian $p$-group}

In this section, we compute the Wedderburn decomposition of rational group algebra of an abelian $p$-group, $p$ a prime. 

Let $G$ be an abelian $p$-group, $p$ a prime. Let $G = \prod_{r = 1}^{n}G_{r}$,
where $G_{r}$ is a product of $a_{r}$ copies of $C_{p^{r}}$. By Theorem \ref{cor}, $\mathbb{Q}[G]$ is isomorphic to direct sum of $f_{p^{i}}$ copies of cyclotomic fields $\mathbb{Q}(\zeta_{p^{i}}), i = 1,2, \dots , n$. We express the coefficients $f_{p^{i}}$ of $\mathbb{Q}(\zeta_{p^{i}})$ in terms of multiplicities of homocyclic components. Let $b_{r} = a_{r} + a_{r + 1} + \dots + a_{n}$ and
$c_{r} = a_{1} + 2a_{2} + \dots + (r-1)a_{r-1}$. Let $B_{r} = \prod_{s<r}{G_{s}}$ and then $G = B_{r} \times \prod_{s \geq r}{G_{s}}$. Let
$\Omega^{r}(G) = \{g \in G \,|\,g^{p^r} = 1\}$ and $H_{r}$ be the product of $b_{r}$ copies of $C_{p^{r}}$. Then $\Omega^{r}(G) = B_{r} \times H_{r}$. Let $E_{r}(G) = \{g \in G \,|\, o(g) = p^{r}\}$. Then $|E_{r}(G)| = |\Omega^{r}(G)| -|\Omega^{r-1}(G)|$. So, the number of cyclic subgroups of $G$ isomorphic to
$C_{p^{r}}$ is equal to ${|E_{r}(G)|}/{\phi(p^{r})}$, where $\phi$ is Euler's phi function. The following result is known, but for the sake of completeness we give a proof.
\begin{proposition}
	The number of subgroups of $G$ with factor groups isomorphic to
	$C_{p^r}$ is equal to the number of cyclic subgroups of $G$ isomorphic to $C_{p^r}$.
\end{proposition}
\begin{proof}
	Let $H$ be a subgroup of $G$ such that  $G/H \simeq C_{p^r}$. Let $\widehat{G}$ denote the dual group of G, the set of irreducible $\mathbb{C}$-representations of G.
	We consider ${\widetilde{H}} = \{\chi \in \widehat{G}\,|\, \chi = 1 ~on~ H\}$. Then clearly ${\widetilde{H}}$ is a subgroup of $\widehat{G}$.
	Let $\psi$ be an isomorphism from $\widehat{G}$ onto
	$G$, and assume that $\psi({\widetilde{H}})$ is equal to $K$ (say). It is easy to show that
	${\widetilde{H}}$ is isomorphic to $C_{p^r}$, and hence $K$ is isomorphic to $C_{p^r}$.
	
	Let $H_1$ and $H_2$ be two distinct subgroups of $G$ such that both $G/H_{1}$ and $G/H_{2}$ are isomorphic to $C_{p^r}$. Then there exists $h \in {H_1 - H_2}$,
	and therefore $\widetilde{H_1} \neq \widetilde{H_2}$. So $\psi{\widetilde{H_1}} \neq \psi{\widetilde{H_2}}$, which implies the number of subgroups of $G$ with factor groups isomorphic to $C_{p^r}$ is equal to the number of subgroups $G$ isomorphic to $C_{p^r}$.
\end{proof}
\begin{theorem}
	Let $G$ be an abelian $p$-group, and of exponent $p^n$. For $r \leq n$, the coefficient of $\mathbb{Q}(\zeta_{p^{r}})$ in the Wedderburn decomposition of $\mathbb{Q}[G]$ is equal to $p^{c_{r} + (r -1 )b_{r - 1}}(\frac{p^{{b_{r}}} - 1}{p - 1})$.
\end{theorem}
\begin{proof}
	By Theorem\ref{cor}, the coefficient of $\mathbb{Q}(\zeta_{p^{r}})$ in the Wedderburn decomposition of $\mathbb{Q}[G]$ is 
	$\frac{|E_{r}(G)|}{\phi(p^{r})}$, where $\phi$ is Euler's phi function. But $|E_{r}(G)| = |B_{r}||E_{r}(H_{r})| = p^{c_{r}}(p^{r{b_{r}}}  - p^{(r - 1)b_{r}}) = p^{c_{r} + {(r - 1)b_{r}}}(p^{b_{r}} - 1)$. So, the coefficient of $\mathbb{Q}(\zeta_{p^{r}})$ in the Wedderburn decomposition is $p^{c_{r} + {(r - 1)b_{r}}}(p^{b_{r}} - 1)/(p^{r} - p^{r-1}) = p^{c_{r} + (r -1 ){(b_{r}- 1})}(\frac{p^{{b_{r}}} - 1}{p - 1})$.
\end{proof}
\begin{corollary}
	The coefficient of $\mathbb{Q}(\zeta_{p^{r}})$ in the Wedderburn decomposition of $\mathbb{Q}[G]$ is equal to $p^{c_r + (r-1)(b_{r}-1)}\{1 + p + \dots + p^{(b_r -1)}\}$. Notice that first part of this polynomial is a power of $p$, while on the other hand the second part is coprime to $p$.
\end{corollary}
\section{Representations of cyclic groups}
In this section, we construct the irreducible representations of a cyclic group $G$ over a field $F$ of characteristic $0$ or prime to $|G|$. 
\begin{theorem}\label{theo}
	Let $G = C_{n} = \langle{x | x^{n} = 1}\rangle$. Let $F$ be a field of characteristic $0$ or prime to $n$. Let $X^{n} - 1 = \displaystyle \prod^{k}_{i = 1}f_{i}(X)$ be the decomposition into irreducible polynomials over $F$. Then $G$ has $k$ irreducible representations $\rho_{1}, \rho_{2}, \dots , \rho_{k}$, say and are defined by $\rho_{i}(x) = C_{f_{i}(X)}$, where $C_{f_{i}(X)}$ denotes the companion matrix of $f_{i}(X)$.
\end{theorem}
\begin{proof} 
	Consider the map $\phi: F[X] \longrightarrow F[G]$ given by: $f(X) \mapsto f(x)$. Since $\phi$ is a ring epimorphism, $F[G] \simeq \frac{F[X]}{\langle{Ker(\phi)}\rangle}$. It is easy to show that Ker$(\phi) = \langle {X^{n} - 1} \rangle$.
	So, we get  
	$$F[G] \simeq \frac{F[X]}{\langle X^{n} - 1 \rangle} = \displaystyle \frac{F[X]}{\langle \prod^{k}_{i = 1}f_{i}(X) \rangle}.$$
	Since $F$ is a field of characteristic $0$ or prime to $n$, $X^{n} - 1$ is a separable polynomial and thus $f_{i}(X) \neq f_{j}(X)$ if $i \neq j$. Using Chinese Remainder Theorem, we can write
	$$F[G] \simeq  \displaystyle \bigoplus^{k}_{i = 1}\frac{F[X]}{\langle f_{i}(X) \rangle}.$$
	Under this isomorphism, the generator $x$ is mapped to the element $(X+ \langle{f_{1}(X)}\rangle, X + \langle{f_{2}(X)}\rangle, \dots , X + \langle {f_{k}(X)}\rangle)$. Since for each $i$, $1 \leq i \leq k$, $f_{i}(X)$ is an irreducible polynomial, $\frac{F[X]}{\langle{f_{i}(X)}\rangle}$ is a field. So, $F[G]$ has $k$ simple components, which implies $G$ has $k$  irreducible representations.
	Let $\rho_{i}$ be the irreducible representation corresponding to the simple component $\frac{F[X]}{\langle{f_{i}(X)}\rangle}$. Then $\rho_{i}(x)$: $\frac{F[X]}{\langle{f_{i}(X)}\rangle} \longrightarrow \frac{F[X]}{\langle{f_{i}(X)}\rangle}$ is defined by multiplication of $[X] = X + \langle{f_{i}(X)}\rangle$. Let $d$ be the degree of the polynomial $f_{i}(X)$. Then $\{ [1], [X], \dots , [X^{d - 1}] \}$ is an ordered basis of $\frac{F[X]}{\langle{f_{i}(X)}\rangle}$, and w.r.t. this basis the matrix for $\rho_{i}(x)$ is $C_{f_{i}(X)}$, where $C_{f_{i}(X)}$ denotes the companion matrix of $f_{i}(X)$. 
	This completes the proof.
\end{proof}
\begin{theorem}\label{facy}
	Let $G = C_{n} = \langle x \,|\, x^{n} = 1 \rangle$. Let $F$ be a field of characteristic $0$ or prime to $n$. Let $\Phi_{n}(X) = \displaystyle \prod^{k} _{i = 1}{f_{i}(X)}$ be the factorization into irreducible polynomials over $F$. Then we have the following.
	\begin{itemize}
		\item [1.] There is a bijective correspondence between the set of faithful irreducible representations of $G$ over $F$ and the set of irreducible factors $\Phi_{n}(X)$ over $F$. 
		\item [2.] If $\rho_{i}$ is the faithful irreducible representation corresponding to the irreducible factor $f_{i}(X)$, then $\rho_{i}$ is defined by $x \mapsto C_{f_{i}(X)}$, where $C_{f_{i}(X)}$ denotes the companion matrix of $f_{i}(X)$.
	\end{itemize}
\end{theorem}
\begin{proof}
	$(1)$ Let $\bar{F}$ be the algebraic closure of $F$. Then 
	$$\bar{F}[G] \simeq \frac{\bar{F}[X]}{\langle { X^{n} - 1} \rangle} \simeq \displaystyle{\bigoplus_{i = 0}^{n-1} \frac{\bar{F}[X]}{\langle X - \zeta^{i}_{n} \rangle}},$$ where $\zeta_{n}$ is a primitive $n^{th}$ root of unity in $\bar{F}$.
	Notice that the simple components corresponding to the faithful irreducible $\bar{F}$-representations of $G$ are  $\frac{\bar{F}[X]}{\langle X - \zeta^{i}_{n} \rangle}$, $(i, n) = 1$. 
	Let $\eta_{i}$ be the representation corresponding to the simple component $\frac{\bar{F}[X]}{\langle X - \zeta^{i}_{n} \rangle}$. Therefore $\eta_{i}$ is defined by $x \mapsto \zeta^{i}_{n}$, $(i,n) = 1$. 
	Since each $\eta_{i}$ is of degree one, the character of $\eta_{i}$ is $\eta_{i}$ itself. 
	By Theorem $\ref{sch}$, for each $i$, $\displaystyle  \oplus_{\sigma \in  \mathrm{Gal}(F(\eta_{i})/F)}{^\sigma{\eta_{i}}}$ is an irreducible $F$-representation. Since for each $i$, $^\sigma{\eta_{i}}$ is a faithful irreducible $\bar{F}$-representation, $\displaystyle  \oplus_{\sigma \in  \mathrm{Gal}(F(\eta_{i})/F)}{^\sigma{\eta_{i}}}$ is a faithful irreducible $F$-representation. Also for each $i$, $f_{i}(X)$ is a separable polynomial, which implies $|\mathrm{Gal}(F(\eta_{i})/F)| = [F(\eta_{i}):F] = \mathrm{deg}{f_{i}(X)}$. So there are at least $k$ faithful irreducible $F$-representations. Let $\rho_{1}, \rho_{2}, \dots , \rho_{k}$ be those faithful irreducible $F$-representations of $G$. 
	It is quite clear that for each $i$, $\rho_{i}$ is determined by the irreducible factor $f_{i}(X)$ of $\Phi_{n}(X)$. It remains to show that they are all the faithful $F$-irreducible representations of $G$. If not, assume $\rho$ is a faithful irreducible $F$-representation, which is different from $\rho_{1}, \rho_{2}, \dots , \rho_{k}$. Then $\rho\otimes_{{F}} {\bar{F}}$ is also a faithful $\bar{F}$-representation of $G$, and is the direct sum of some faithful $\bar{F}$-irreducible representations. So, some of the $\eta_{i}$'s must occur in the decomposition, and which contradicts the assumption. Thus $\rho_{1}, \rho_{2}, \dots , \rho_{k}$ are the only faithful irreducible $F$-representations of $G$. This completes the proof of $(1)$.\\
	$(2)$ Follows from Theorem $\ref{theo}$.
\end{proof}
\begin{corollary}
	Let $F$ be a field of characteristic $0$ or prime to $n$. Then construction of the faithful irreduicble $F$-representations of $C_{n}$ reduces to explicit decomposition of $\Phi_{n}(X)$ into irreducible factors. This depends on arithmetic of $F$.
\end{corollary}
\section{Representations of finite abelian groups}
In this section, we construct the irreducible representations of a finite abelian group over a field of characteristic $0$ or prime to the order of the group.
\begin{lemma}\label{lem}
	Let $G$ be a finite group. Let $F$ be a field of characteristic $0$ or prime to $|G|$. If $G$ has a faithful irreducible representation, then $Z(G)$ is cyclic.
\end{lemma}
For the proof, see $\cite{gor109}$, Chapter $3$, Theorem $2.2$.
\begin{theorem}\label{abe1}
	Let $G$ be a finite abelian group. Let $F$ be a field of characteristic $0$ or prime to $|G|$. For a divisor $d$ of $|G|$, let $\mathcal{H}_{d} = \{H \leq G \,|\, G/H \simeq C_{d}\}$. Let $\Omega_{G}$ denote the set of irreducible $F$-representations of $G$ and $\Omega^{o}_{{d}}$ denote the set of faithful irreducible $F$-representations of $C_{d}$. Then $\Omega_{G}$ is in a bijective correspondence with the set 
	$\displaystyle {\cup_{d | |G|}(\mathcal{H}_{d} \times \Omega^{o}_{{d}})}$.
\end{theorem}
\begin{proof}\rm
	Let $(\rho, V)$ be an irreducible $F$-representation of $G$, and $K$ be its kernel. So, $\rho$ factors through a faithful irreducible representation of $G/K$. By Lemma $\ref{lem}$, $G/K$ is a cyclic group, and is isomorphic to $C_{d}$ (say). So, $\rho$ corresponds to an element in  
	$\displaystyle {\cup_{d | |G|}(\mathcal{H}_{d}\times \Omega^{o}_{{d}})}$ and this correspondence is one-one. Let $(\rho_{1}, V_{1})$ and $(\rho_{2}, V_{2})$ be two irreducible representations of $G$ with kernel $K_{1}$ and $K_{2}$ respectively. So, $\rho_{1}$ and $\rho_{2}$ factor through faithful irreducible representations $\eta_{1}$ (say) of $G/K_{1}$ and $\eta_{2}$ (say) of $G/K_{2}$ respectively. We can consider two cases: $G/K_{1}$ is isomorphic to $G/K_{2}$ and  $G/K_{1}$ is not isomorphic to $G/K_{2}$. In both the cases, it is clear that $\eta_{1} \ncong \eta_{2}$.
	
	For a divisor $d$ of $|G|$, let $(\eta, W)$ be a faithful irreducible representation of $G/H \simeq C_{d}$ (say), where $H$ is a normal subgroup of $G$. Then the pull back of $\eta$ by the natural group homomorphism: $G \rightarrow G/H$, say $\rho$ is an irreducible representation of $G$. So, corresponding to an element in $\displaystyle {\cup_{d | |G|}(\mathcal{H}_{d} \times \Omega^{o}_{{d}})}$ there is an element in $\Omega_{G}$. This defines a correspondence from $\displaystyle {\cup_{d | |G|}(\mathcal{H}_{d} \times \Omega^{o}_{{d}})}$ to $\Omega_{G}$, which is one-one also. Therefore $\Omega_{G}$ is in a bijective correspondence with the set 
	$\displaystyle {\cup_{d | |G|}(\mathcal{H}_{d} \times \Omega^{o}_{{d}})}$. This completes the proof. 
\end{proof}
\begin{remark}
	Note that $\mathcal{H}_{d}$ is an absolute invariant, depending only on $G$, whereas $\Omega^{o}_{{d}}$ depends on $F$.
\end{remark}
\begin{theorem} [$\cite{gor109}$, Chapter $3$, Theorem $2.3$]\label{abe2}
	Let $G$ be a finite abelian group. Let $F$ be a field of characteristic $0$ or prime to $|G|$. Then every irreducible representation of $G$ over $F$ factors through a faithful irreducible representation of a cyclic quotient.
\end{theorem}
\begin{proof}
	Let $(\rho, V)$ be an irreducible representation of $G$ and $K$ be the kernel of $\rho$. Then $\rho$ factors through a faithful irreducible representation of $G/K$. So, by Lemma \ref{lem}, $Z(G/K) = G/K$ is cyclic. This completes the proof.
\end{proof}
\begin{corollary}
	Construction of the irreducible representations of a finite abelian group reduces to construction of the faithful irreducible representations of its cyclic quotients.
\end{corollary}
\section{Algorithmic construction of representations of finite abelian groups}
Let $G$ be a finite abelian group. Let $F$ be a field of characteristic $0$ or prime to $|G|$. Summarizing the previous two sections, we write an algorithm for constructing the irreducible $F$-representations of $G$ in the following way.\\
\noindent
{\bf Step 1:} By Theorem \ref{abe1}, every irreducible representation of $G$ factors through a faithful irreducible representation of a cyclic quotient. So to construct the irreducible representations of $G$ it is sufficient to find faithful irreducible representations of its cyclic quotients. We first find all the distinct cyclic quotients of $G$.\\
\noindent
{\bf Step 2:} By Theorem \ref{facy}, the faithful irreducible representations of a cyclic quotient, say isomorphic to $C_{n}$ (say), can be constructed from decomposition of the $n^{th}$ cyclotomic polynomial $\Phi_{n}(X)$ into irreducible polynomials over $F$. So, for each such $n$, we need to find the decomposition of $\Phi_{n}(X)$ into irreducible polynomials. This is the arithmetic step, depending on the arithmetic of $F$.\\
\noindent
{\bf Step 3:} By Theorem \ref{facy}, if $\Phi_{n}(X) = \prod^{k} _{i = 1}{f_{i}(X)}$ is the factorization into irreducible polynomials over $F$, then there is a bijective correspondence between the set of faithful irreducible representations of $C_{n}$ and the set of irreducible factors $\Phi_{n}(X)$, that is, $\{f_{1}(X), f_{2}(x), \dots, f_{k}(X)\}$. If $\rho_{i}$ is the faithful irreducible representation corresponding to the irreducible factor $f_{i}(X)$, then $\rho_{i}$ is defined by: $x \mapsto C_{f_{i}(X)}$, where $C_{f_{i}(X)}$ denotes the companion matrix of $f_{i}(X)$.

By following these three steps, we obtain all the irreducible $F$-representations of $G$.
\section{Primitve central idempotents of abelian group algebras}
In this section, we describe a systematic way of computing the primitive central idempotents of a finite abelian group $G$ over a field $F$ of characteristic $0$ or prime to $|G|$. By Theorem \ref{abe2}, every primitive central idempotent of $F[G]$ is pull back of the primitive central idempotent in $F[G/K]$ corresponding to a {\it faithful} irreducible representation of $G/K$, where $G/K$ is isomorphic to a cyclic group. Let $G/K = \langle{ \bar{x} | {\bar{x}}^{n} = 1 }\rangle \simeq C_{n}$, where $\bar{x} = xK$. Let $e = \sum_{i = 0}^{n-1}{\alpha_{i}{\bar{x}}^{i}}$ be a primitive central idempotent in $F[G/K]$ corresponding to a {\it faithful} irreducible representation of $G/K$. Then the pull back of $e = \sum_{i = 0}^{n-1}{\alpha_{i}{\bar{x}}^{i}}$, that is, $e_{K}\sum_{i = 0}^{n-1}{\alpha_{i}{{x}}^{i}}$, where $e_{K} = {\sum_{k \in K}k}/{|K|}$ is a primitive central idempotent of $F[G]$. Thus the problem reduces to computing the primitive central idempotents of a cyclic group over a field of characteristic $0$ or prime to the order of the group.

Let $G = \langle {x | x^{n} = 1}\rangle$ be a cyclic group of order $n$. Let $F$ be a field of characteristic $0$ or prime to $n$. Then $F[G]$ is isomorphic to $F[X]/\langle X^{n} - 1 \rangle$. Let $X^{n} - 1 = \prod_{d | n}{\Phi_{d}(X)}$ be the "universal" decomposition into cyclotomic polynomials. For $\mathbb{Q}$, it is the decomposition into irreducible factors. In general, $\Phi_{d}(X)$ will decompose further. Let $X^{n} - 1 = (X-1)f_{2}(X)\dots f_{k}(X)$, with $f_{1}(X) = X - 1$, be the factorization into monic irreducible polynomials. Note that each $f_{i}(X)$ is a monic irreducible factor of $\Phi_{d}(X)$ for some $d | n$. Then $F[G]$ is abstractly isomorphic to the direct sum of $F$-algebras $F[X]/\langle f_{i}(X) \rangle, i = 1, 2, \dots , k$. For a fixed $i$, let $\zeta_{1}, \zeta_{2}, \dots , \zeta_{d}$ be the roots of $f_{i}(X)$ in the algebraic closure of $F$. Then the primitive central idempotent corresponding to $f_{i}(X)$ is the sum of $e_{\zeta_{j}x}, j = 1,2,\dots,k$, where $e_{X} = (1 + X + \dots + X^{n-1})/n$ and $X$ is an indeterminate. By Newton's formulae, power sums can be expressed in terms of elementary symmetric functions,  which implies the coefficients of $x, x^2, \dots , x^{n-1}$ can be expressed in terms of the coefficients of the $f_{i}(X)$. So, computation of the primitive central idempotents of $F[G]$ reduces to factorization of $X^{n} - 1$ into irreducible polynomials over $F$. By Theorem \ref{facy}, the computation of  primitive central idempotents of $F[G]$ corresponding to the faithful irreducible representations of $G$ reduces to factorization of $\Phi_{n}(X)$ into irreducible polynomials over $F$. This depends on the arithmetic of $F$.
\chapter{Examples}
In this chapter, we give few examples which illustrate construction of the irreducible matrix representations of a finite solvable group over $\mathbb{C}$ by using Clifford theory (see \cite{alg111}, \cite{cl50}, \cite{combi91}).
\section{Symmetric group $S_{3}$}
\begin{example}\rm
	$S_{3}$ is the group of symmetries of an equilateral triangle. Let $x$ be a rotational symmetry of the
	triangle with angle of rotation $120^0$. Let $y$ be a reflection in a line passing through a vertex and
	the midpoint of the opposite edge. 
	$S_3$ has a maximal subnormal series
	\begin{equation}
	\langle e \rangle = G_{o}  <  C_{3} = G_{1} < S_3 = G_{2}
	\end{equation}
	The long presentation associated with the series $(8.1.1)$ is
	\begin{center}
		$ S_3 = \langle x,y \, | \, x^3 = y^2 = 1, y^{-1}xy=x^{-1}\rangle$.
	\end{center}
	\noindent
	We obtain the primitive central idempotents in $\mathbb{C}[S_{3}]$ and 
	construct the irreducible matrix representations of $S_3$  
	using the long presentation.
	
	The only primitive central idempotent in $\mathbb{C}[G_0]$ is $e=1$, and the corresponding representation is trivial representation.
	The trivial representation of $G_0$ extends to three mutually inequivalent irreducible 
	representations of $G_1=\langle x\rangle$.
	To obtain the corresponding primitive central idempotents, consider the group algebra 
	$\mathbb{C}[G_1]$. Since $G_1$ is a cyclic group of order $3$,
	\begin{center}
		$\mathbb{C}[G_1] \cong \frac{\mathbb{C}[X]}{(x^3-1)} \cong \frac{\mathbb{C}[X]}{(x-1)} \oplus 
		\frac{\mathbb{C}[X]}{(x-\omega)} \oplus \frac{\mathbb{C}[X]}{(x-\omega^2)} \cong \mathbb{C} \oplus \mathbb{C} \oplus \mathbb{C},$
	\end{center}
	where $\omega$ is a primitive cube root of unity.  The identity elements in the last three components correspond to three primitive central idempotents in $\mathbb{C}[G_1]$ via the above isomorphism. To obtain primitive central 
	idempotents, write $e_X={1+X+X^2}/{3}$, where $X$ is an indeterminate.
	The primitive central idempotents of $\mathbb{C}[G_1]$ are
	$$e_x = \frac{1+x+x^2}{3},\,\, e_{\omega x}= \frac{1+\omega x+\omega^2 x^2}{3},\,\, e_{\omega^2 x} = \frac{1+\omega^2 x+\omega x^2}{3}.$$\\
	\noindent
	{\it{PCI-diagram of $G_{1}$}}:
	The PCI-diagram of $G_{1}$ associated with the series: $\langle{e}\rangle = G_{o} < C_{3} = \langle{x}\rangle = G_{1}$ is
	
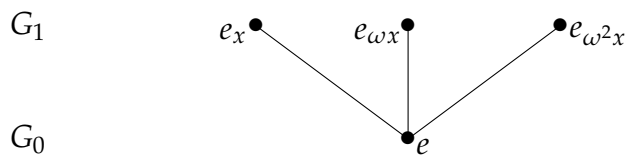
\begin{center}
		\begin{tikzpicture}
		\node at (0,0) {$\bullet$};
		\node at (-2,1.5) {$\bullet$};\node at (0,1.5) {$\bullet$};\node at (2,1.5) {$\bullet$};
		\draw (0,1.5)--(0,0)--(-2,1.5);\draw (0,0)--(2,1.5);
		\node at (0.2,-0.1) {$e$};
		\node at (-2.3,1.4) {$e_{x}$};\node at (-0.4,1.4) {$e_{\omega x}$};\node at (2.5,1.4) {$e_{\omega^2 x}$};
		\node at (-5,0) {$G_0$};\node at (-5,1.5) {$G_1$};
		\end{tikzpicture}
	\end{center}
	Here, $e_{X} =\frac{1+X+X^2}{3}, X \in \{x, \omega x, \omega^2 x\}, \omega^3=1,\, 
	\omega\neq 1.$
	
	We obtain the primitive central idempotents in $\mathbb{C}[G_2]$ starting with 
	primitive central idempotents of $\mathbb{C}[G_1]$. The primitive central idempotents 
	$e_{\omega x}$ and $e_{\omega^2 x}$ are primitive idempotents in $\mathbb{C}[G_2]$ but 
	not central. However, their sum $e_1=e_{\omega x}+e_{\omega^2 x}$ is a primitive central idempotent 
	in $\mathbb{C}[G_2]$.
	The primitive central idempotent $e_x$ of $\mathbb{C}[G_1]$ is a central idempotent 
	in $\mathbb{C}[G_2]$. Two primitive central idempotents $e_2=e_xe_{y}$ and 
	$e_3=e_xe_{-y}$, where $e_{y} = {(1 + y)}/2$, $e_{-y} = {(1 - y)}/2$ correspond to two extensions of $e_{x}$.\\
	\noindent
	{\it{PCI-diagram of $G_{2}$}}:
	The PCI-diagram associated with the series $(8.1.1)$ is
	
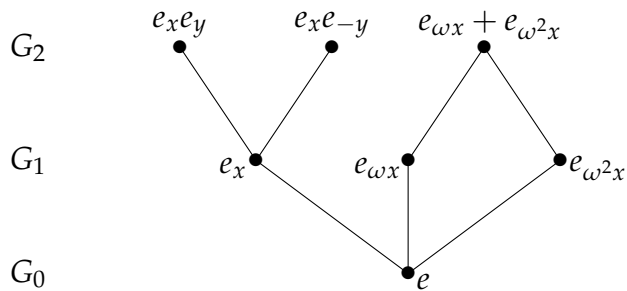
\begin{center}
		\begin{tikzpicture}
		\node at (0,0) {$\bullet$};
		\node at (-2,1.5) {$\bullet$};\node at (0,1.5) {$\bullet$};\node at (2,1.5) {$\bullet$};
		\node at (-3,3) {$\bullet$};\node at (-1,3) {$\bullet$};\node at (1,3) {$\bullet$};
		\draw (0,1.5)--(0,0)--(-2,1.5);\draw (0,0)--(2,1.5);
		\draw (-3,3)--(-2,1.5)--(-1,3);\draw (0,1.5)--(1,3)--(2,1.5);
		\node at (0.2,-0.1) {$e$};
		\node at (-2.3,1.4) {$e_{x}$};\node at (-0.4,1.4) {$e_{\omega x}$};\node at (2.5,1.4) {$e_{\omega^2 x}$};
		\node at (-3,3.3) {$e_{x}e_{y}$};\node at (-1,3.3) {$e_{x}e_{-y}$};\node at (1.08,3.3) {$e_{\omega x}+e_{\omega^2 x}$};
		\node at (-5,0) {$G_0$};\node at (-5,1.5) {$G_1$};\node at (-5,3) {$G_2$};
		\end{tikzpicture}
	\end{center}
	Here, $e_{X} =\frac{1+X+X^2}{3}, X \in \{x, \omega x, \omega^2 x\}, e_{Y} = \frac{1+Y}{2}, Y \in  \{y, -y\}, \omega^3=1,\, \omega\neq 1.$\\\\
	\noindent
	{\it{Irreducible representation of $G_{2}$ corresponding to the primitive central idempotent $e_1 = e_{\omega x} + e_{\omega^2 x}$}}:
	Let $\rho_1$ be the irreducible representation corresponding to $e_1$. The representation 
	space $V_{\rho_{1}}$ of $\rho_{1}$ is a
	two-dimensional vector space, and the subspace 
	$\langle { e_{\omega x}, ye_{\omega x} }\rangle$ of $\mathbb{C}[G_2]$ can serve as a 
	model for the representation space $V_{\rho_{1}}$ of $\rho_1$. Then, the values of 
	$\rho_1$ on the generators $x,y$ of $G_2$ are given by

\begin{align*}
	\rho_1(x) &\colon e_{\omega x} \mapsto  \omega^2 e_{\omega x},\\
	\rho_1(x) &\colon ye_{\omega x} \mapsto \omega ye_{\omega x},
	\end{align*}
	and
	\begin{align*}
	\rho_1(y) &\colon e_{\omega x} \mapsto ye_{\omega x},\\
	\rho_1(y) & \colon ye_{\omega x} \mapsto e_{\omega x}.
	\end{align*}
	Therefore w.r.t. the basis $\{ e_{\omega x}, ye_{\omega x} \}$ of $V_{\rho_{1}}$, the matrix representation of $\rho_{1}$ is
	\begin{equation*}
	[\rho_1(x)]=
	\left [
	\begin{array}{cc}
	\omega^2 & 0\\
	0 & \omega
	\end{array}
	\right], \mbox{ and }
	[\rho_1(y)]=
	\left [
	\begin{array}{cc}
	0 & 1\\
	1 & 0
	\end{array}
	\right].
	\end{equation*}
	\\
	\noindent
	{\it{Irreducible representation of $G_{2}$ corresponding to the primitive central idempotent $e_2 = e_{x}e_{y}$ }}:
	Let $\rho_2$ be the irreducible representation corresponding to $e_2$.
	The subspace $\langle e_xe_y\rangle$ of $\mathbb{C}[G_2]$ can serve as a model for the
	representation space $V_{\rho_{2}}$ for $\rho_2$. Then, the values of $\rho_2$ on the generators $x,y$ of $G_2$ are given by

\begin{align*}
	\rho_2(x) & \colon e_xe_y \mapsto e_xe_y,\\
	\rho_2(y) & \colon e_xe_y \mapsto e_xe_y.
	\end{align*}
	Thus, the representation $\rho_2$ is the trivial representation.\\
	\noindent
	{\it{Irreducible representation of $G_{2}$ corresponding to the primitive central idempotent $e_3 = e_{x}e_{-y}$}}:
	Let $\rho_{3}$ be the irreducible representation corresponding to the primitive central idempotent $e_{3}$. The subspace $\langle e_xe_{-y}\rangle$ of $\mathbb{C}[G_2]$ serves as a model for the
	representation space $V_{\rho_{3}}$ of $\rho_3$. Then the values of $\rho_3$ on the generators $x,y$ of $G_2$ are given by
\begin{align*}
	\rho_3(x) &\colon e_xe_{-y}\mapsto e_xe_{-y},\\
	\rho_3(y) &\colon e_xe_{-y}\mapsto -e_xe_{-y}.
	\end{align*}
	Therefore the matrix representation of $G_2$  w.r.t. the basis $\{e_xe_{-y}\}$ of $V_{3}$ is
	\begin{center}
		$[\rho_3(x)]=1$ and $[\rho_3(y)]=-1$.
	\end{center}
\end{example}
\section{Dihedral group $D_8$}
\begin{example}\rm
	$(1)$ We first construct the irreducible representations of $D_{8}$ over $\mathbb{C}$ by using Clifford theory.\\
	Let us consider the maximal subnormal series
	\begin{equation}
	\langle e \rangle = G_{o} < C_2 = G_1 < C_2 \times C_2 = G_2 < D_{8} = G_3.
	\end{equation}
	The long presentation associated with the series $(8.2.1)$ is
	$$D_{8} = \langle{x,y,z \,|\,  x^{2} = 1,\,  y^{2} = 1,\,  xy = yx ,\,  z^{2} = 1,\, z^{-1}xz = x, \, z^{-1}yz = xy}\rangle.$$
	Therefore 
	$G_0 = \{1\}$, $G_1 = \langle x \rangle$, $G_2=\langle x,y\rangle$ and  $G_3=\langle x,y,z\rangle$.
	Since $G_2$ is a subgroup of index $2$, one can construct the irreducible representations of $D_{8}$ starting with
	irreducible representations of $G_2$. There are four irreducible $\mathbb{C}$-representations of $G_2$ which are given by 
	\begin{align*}
	\eta_1(x)=1, \eta_1(y)=1  & & \eta_2(x)=1, \eta_2(y)=-1
	\end{align*}
	\begin{align*}
	\eta_3(x)=-1, \eta_3(y)=1 &  & \eta_4(x)=-1, \eta_4(y)=-1.
	\end{align*}
	Since $\eta_1 \cong \eta^{z}_1$ and $\eta_2 \cong \eta^{z}_2$, then $\eta_{1}$ and $\eta_{2}$ extends to two representations of $D_8$. The two extensions of 
	$\eta_1$ are given by 
	\begin{equation*}
	\rho_1(z) =1, \hskip3mm \rho_2(z) = -1.
	\end{equation*}
	Similarly, the two extensions of $\eta_2$ are given by
	\begin{equation*}
	\rho_3(z) =1, \hskip3mm \rho_4(z) = -1.
	\end{equation*}
	Since $\eta_3 \ncong \eta^{z}_3$ and $\eta^{z}_3 \cong \eta_4$, then $\eta_3$, $\eta_4$ induce the same 
	irreducible representation $\rho_{5}$ and is given by
	\begin{equation*}
	[\rho_{5}(x)] =
	\begin{bmatrix} 
	-1 & 0 \\
	0 & -1 
	\end{bmatrix}, \hskip3mm 
	[\rho_{5}(y)] = 
	\begin{bmatrix} 
	1 & 0 \\
	0 & -1 \end{bmatrix}, \hskip3mm  
	[\rho_{5}(z)] = 
	\begin{bmatrix}
	0 & 1 \\
	1 & 0 \end{bmatrix}.
	\end{equation*}
	
	\vspace{4mm}
	\noindent
	$(2)$ We now construct the irreducible representations of $D_{8}$ using our algorithm (see Chapter$6$).\\
	\noindent
	{\it{PCI-diagram of $G_2$}}: Write $e_X=\frac{1+X}{2}$, where $X$ is an indeterminate. The complete set of primitive central idempotents in $\mathbb{C}[G_2]$ are
	$$e_xe_y,  \hskip3mm  e_xe_{-y}, \hskip3mm e_{-x}e_y, \hskip3mm e_{-x}e_{-y}.$$
	The PCI-diagram of $G_2$ associated with the series $G_0<G_1<G_2$   
	is
	
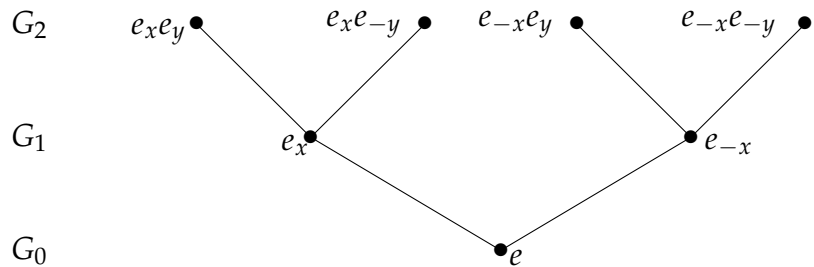
\begin{center}
		\begin{tikzpicture}
		\node at (0,0) {$\bullet$};
		\node at (-2.5,1.5) {$\bullet$};\node at (2.5,1.5) {$\bullet$};
		\node at (-4,3) {$\bullet$};\node at (-1,3) {$\bullet$};\node at (1,3) {$\bullet$};\node at (4,3) {$\bullet$};
		\draw (0,0)--(-2.5,1.5); \draw (0,0)--(2.5,1.5);
		\draw (-2.5,1.5)--(-4,3);\draw (-2.5,1.5)--(-1,3);\draw (2.5,1.5)--(4,3);\draw (2.5,1.5)--(1,3);
		\node at (0.2,-0.1) {$e$};
		\node at (-2.7,1.4) {$e_x$};\node at (3,1.4) {$e_{-x}$};
		\node at (-4.5,2.9) {$e_xe_y$};\node at (-1.8,3) {$e_xe_{-y}$};
		\node at (0.2,3) {$e_{-x}e_y$};\node at (3,3) {$e_{-x}e_{-y}$};
		\node at (-6.2, 0) {$G_{0}$};
		\node at (-6.2, 1.5) {$G_1$};
		\node at (-6.2, 3) {$G_2$};
		\end{tikzpicture}
	\end{center}
	{\it{PCI-diagram of $G_3$}}: Let $X$ and $e_X$ be as before. The complete set of primitive central idempotents in 
	$\mathbb{C}[G_3]$ are
	$$e_xe_ye_z,  \hskip2mm e_xe_ye_{-z}, \hskip2mm e_xe_{-y}e_z,  \hskip2mm e_xe_{-y}e_{-z}, \hskip2mm e_{-x}e_y + e_{-x}e_{-y} = 1 - e_x = e_{-x}.$$
	The PCI-diagram of $G_3$ associated with the series $(8.2.1)$ is
	
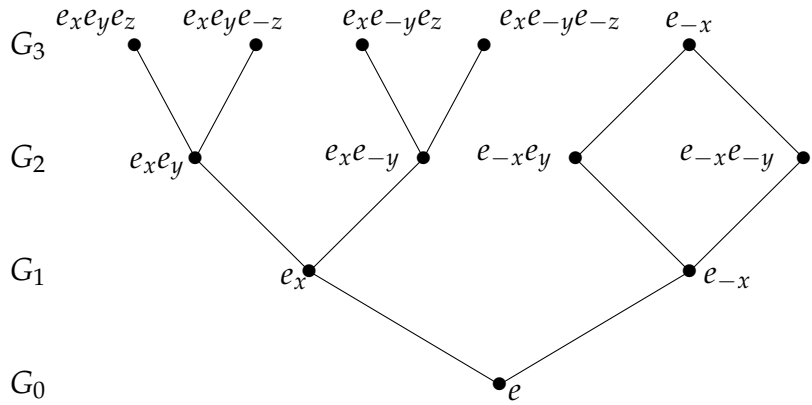
\begin{center}
		\begin{tikzpicture}
		\node at (0,0) {$\bullet$};
		\node at (-2.5,1.5) {$\bullet$};\node at (2.5,1.5) {$\bullet$};
		\node at (-4,3) {$\bullet$};\node at (-1,3) {$\bullet$};\node at (1,3) {$\bullet$};\node at (4,3) {$\bullet$};
		\node at (-4.8,4.5) {$\bullet$};\node at (-3.2,4.5) {$\bullet$};\node at (-1.8,4.5) {$\bullet$};\node at (-0.2,4.5) {$\bullet$};\node at (2.5,4.5) {$\bullet$};
		\draw (0,0)--(-2.5,1.5); \draw (0,0)--(2.5,1.5);
		\draw (-2.5,1.5)--(-4,3);\draw (-2.5,1.5)--(-1,3);\draw (2.5,1.5)--(4,3);\draw (2.5,1.5)--(1,3);
		\draw (-4,3)--(-4.8,4.5);\draw (-4,3)--(-3.2,4.5);
		\draw (-1,3)--(-1.8,4.5);\draw (-1,3)--(-0.2,4.5);
		\draw (1,3)--(2.5,4.5);\draw (4,3)--(2.5,4.5);
		\node at (0.2,-0.1) {$e$};
		\node at (-2.7,1.4) {$e_x$};\node at (3,1.4) {$e_{-x}$};
		\node at (-4.5,2.9) {$e_xe_y$};\node at (-1.8,3) {$e_xe_{-y}$};
		\node at (0.2,3) {$e_{-x}e_y$};\node at (3,3) {$e_{-x}e_{-y}$};
		\node at (-5.3,4.8) {$e_xe_ye_z$};\node at (-3.5,4.8) {$e_xe_ye_{-z}$};\node at (-1.4,4.8) {$e_xe_{-y}e_z$};
		\node at (0.8,4.8) {$e_xe_{-y}e_{-z}$};\node at (2.5,4.8) {$e_{-x}$};
		\node at (-6.2, 0) {$G_{0}$};
		\node at (-6.2, 1.5) {$G_1$};
		\node at (-6.2, 3) {$G_2$};
		\node at (-6.2, 4.5) {$G_3$};
		\end{tikzpicture}
	\end{center}
	\noindent
	{\it{Irreducible representation of $G_3$ corresponding to the primitive central idempotent $e_1 = e_{-x}$}}: 
	
	Let $(\rho_1,V_{\rho_1})$ be the irreducible representation of $G_3$ corresponding to $e_1$. Then $\dim_{\mathbb{C} } V_{\rho_1}=2$ and 
	the subspace $\langle{e_{-x}e_y, ze_{-x}e_y}\rangle$ of $\mathbb{C}[G_3]$ can serve as a model for the representation space for $\rho_1$. 
	Then the action of $x$, $y$, and $z$ on the basis $\{ e_{-x}e_y, ze_{-x}e_y\}$ is given by 
	\begin{align*}
	\rho_1(x) & \colon e_{-x}e_y \mapsto -e_{-x}e_y,\\
	& ze_{-x}e_y \mapsto -ze_{-x}e_y,\
	\end{align*}
	\begin{align*} 
	\rho_1(y) & \colon e_{-x}e_y \mapsto  e_{-x}e_y, \\
	& ze_{-x}e_y \mapsto -ze_{-x}e_y,\ 
	\end{align*}
	\begin{align*}
	\rho_1(z) & \colon  e_{-x}e_y \mapsto ze_{-x}e_y, \\
	& ze_{-x}e_y \mapsto e_{-x}e_y.\
	\end{align*}
	Therefore w.r.t. the basis $\{ e_{-x}e_y, ze_{-x}e_y\}$, the matrix representation of $\rho_1$ is
	\begin{equation*}
	[\rho_1(x)]=
	\left [
	\begin{array}{cc}
	-1 & 0\\
	0 & -1
	\end{array}
	\right],
	[\rho_1(y)]=
	\left [
	\begin{array}{cc}
	1 & 0\\
	0 & -1
	\end{array}
	\right],
	[\rho_1(z)]=
	\left [
	\begin{array}{cc}
	0 & 1\\
	1 & 0
	\end{array}
	\right].
	\end{equation*}
	
	The irreducible representations corresponding to the primitive central idempotents $e_xe_ye_z$, $e_xe_ye_{-z}$, $e_xe_{-y}e_z$ and $e_xe_{-y}e_{-z}$ are all $1$-dimensional, and 
	can be obtained in similar way. 
	In the following table, $e$ denotes a 
	primitive central idempotent in $\mathbb{C}[G_3]$, $(\rho,V_{\rho})$ denotes the corresponding irreducible representation. The model for the representation space $V_{\rho}$ is given as a 
	subspace of $\mathbb{C}[G_3]$. 
	\begin{center}
		\small
		\begin{tabular}{||c|c|c|c||}
			\hline\hline
			$e$ & $(\rho,V_{\rho})$ & Model for $V_{\rho}$  & Action of $\rho$ on $V$ \\ \hline \hline 
			$e_{2} = e_xe_ye_z$ & $(\rho_2, V_{\rho_2})$  & $\langle v\rangle=\langle e_xe_ye_z\rangle$ & $\rho_2(x): v\mapsto v, \hskip2mm \rho_2(y): v\mapsto v,\hskip2mm \rho_2(z) :v\mapsto v$ \\
			$e_{3} = e_xe_ye_{-z}$ & $(\rho_3, V_{\rho_3})$  & $\langle v\rangle=\langle e_xe_ye_{-z}\rangle$ & $\rho_3(x): v\mapsto v, \hskip2mm \rho_3(y): v\mapsto v,\hskip2mm \rho_3(z) :v\mapsto -v$ \\
			$e_{4} = e_xe_{-y}e_z$ & $(\rho_4, V_{\rho_4})$  & $\langle v\rangle=\langle e_xe_{-y}e_z\rangle$ & $\rho_4(x): v\mapsto v, \hskip2mm \rho_4(y): v\mapsto - v,\hskip2mm \rho_4(z) :v\mapsto v$ \\
			$e_{5} = e_xe_{-y}e_{-z}$ & $(\rho_5, V_{\rho_5})$  & $\langle v\rangle=\langle e_xe_{-y}e_{-z}\rangle$ & $\rho_{5}(x): v\mapsto v, \hskip1mm \rho_5(y): v\mapsto -v,\hskip1mm \rho_5(z) :v\mapsto -v$ \\
			\hline\hline
		\end{tabular}
	\end{center}
	\normalsize
\end{example}
%
%
\section{Quaternion group $Q_8$}
\begin{example}\rm
	$(1)$ We first construct the irreducible representations of $Q_{8}$ over $\mathbb{C}$ by using Clifford theory.\\
	Let us consider the maximal subnormal series
	\begin{equation}
	\langle e \rangle = G_{0} < C_2 = G_1 < C_4 = G_2 < Q_{8} = G_3.
	\end{equation}
	The long presentation of $Q_8$ associated with the series $(8.3.1)$ is 
	$$ Q_{8} = \langle{x,y,z \,|\, x^{2} = 1,\, y^{2} = x,\, z^{4} = 1,\,y^{2} = z^{2},\, z^{-1}yz = xy}\rangle. $$
	By construction, 
	$G_1 = \{ 1 \}$,  $G_1 = \langle x \rangle$, $G_2 = \langle  x,y \rangle$ $G_3 = \langle x,y,z \rangle$. 
	Since $G_2$ is a subgroup of index $2$ in $Q_{8}$, we can construct all the inequivalent irreducible 
	matrix representations of $Q_{8}$ starting with irreducible representations of $G_2$. There are four irreducible representations 
	$\eta_i$ ($i=1,2,3,4$) of $G_2$ which are given by
	\begin{align*}
	\eta_1(y) = 1 , \hskip1cm \eta_2(y) = -1, \hskip1cm  \eta_3(y) = i, \hskip1cm \eta_4(y) = -i. 
	\end{align*}
	The representation $\eta_1$ extends to $Q_8$ in two ways, and are given by 
	$$\rho_1(z) =1, \rho_2(z) = -1.$$
	Similarly there are two extensions of $\eta_2$, and are given by 
	$$\rho_3(z) =1, \rho_4(z) = -1.$$ 
	Since $\eta_{3} \ncong \eta^{z}_3$ and $\eta^{z}_{3} \cong \eta_4$, then $\eta_3$ and $\eta_4$ induce the same irreducible representation $\rho_{5}$, and is given by
	\begin{equation*}
	[\rho_{5}(x)] =
	\begin{bmatrix}
	-1 & 0 \\
	0 & -1 \end{bmatrix}, 
	\hskip3mm 
	[\rho_{5}(y)] = 
	\begin{bmatrix}
	i & 0 \\
	0 & -i \end{bmatrix},
	\hskip3mm 
	[\rho_{5}(z)] = 
	\begin{bmatrix}
	0 & -1 \\
	1 & 0 
	\end{bmatrix}.
	\end{equation*}
	\vspace{3mm}
	\noindent
	$\bullet$ We construct the irreducible representations of $Q_{8}$ by our algorithm (see Chapter$6$).\\
	{\it {PCI-diagram of $Q_{8}$}}: Write $e_X=\frac{1+X}{2}$, where $X$ is an indeterminate. 
	The complete set of primitive central idempotents of $\mathbb{C}[G_2]$ are 
	$$e_xe_y, e_xe_{-y}, e_{-x}e_{iy}, e_{-x}e_{-iy} \hskip5mm (i=\sqrt{-1}\in\mathbb{C}),$$
	The complete set of primitive central idempotents of $\mathbb{C}[G_3]$ are
	$$e_xe_ye_z, \hskip2mm e_xe_ye_{-z}, \hskip2mm e_xe_{-y}e_z, \hskip2mm e_xe_{-y}e_{-z} \mbox{ and } e_{-x}e_{iy} + e_{-x}e_{-iy} = 1 - e_x = e_{-x}.$$
	The PCI-diagram associated with the series $(8.3.1)$ is
	
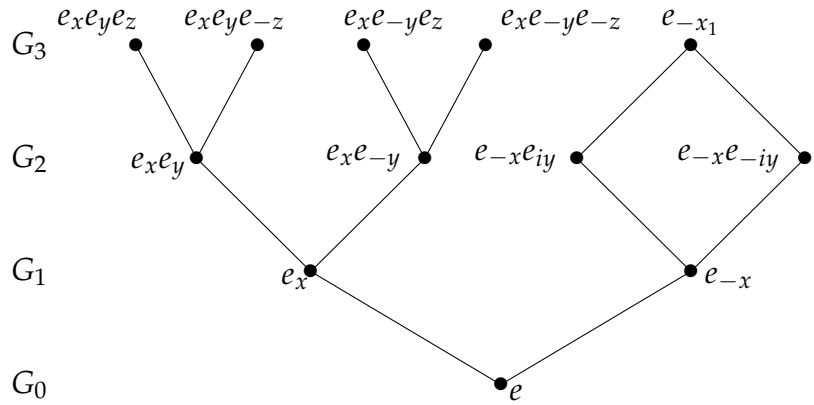
\begin{center}
		\begin{tikzpicture}
		\node at (0,0) {$\bullet$};
		\node at (-2.5,1.5) {$\bullet$};\node at (2.5,1.5) {$\bullet$};
		\node at (-4,3) {$\bullet$};\node at (-1,3) {$\bullet$};\node at (1,3) {$\bullet$};\node at (4,3) {$\bullet$};
		\node at (-4.8,4.5) {$\bullet$};\node at (-3.2,4.5) {$\bullet$};\node at (-1.8,4.5) {$\bullet$};\node at (-0.2,4.5) {$\bullet$};\node at (2.5,4.5) {$\bullet$};
		\draw (0,0)--(-2.5,1.5); \draw (0,0)--(2.5,1.5);
		\draw (-2.5,1.5)--(-4,3);\draw (-2.5,1.5)--(-1,3);\draw (2.5,1.5)--(4,3);\draw (2.5,1.5)--(1,3);
		\draw (-4,3)--(-4.8,4.5);\draw (-4,3)--(-3.2,4.5);
		\draw (-1,3)--(-1.8,4.5);\draw (-1,3)--(-0.2,4.5);
		\draw (1,3)--(2.5,4.5);\draw (4,3)--(2.5,4.5);
		\node at (0.2,-0.1) {$e$};
		\node at (-2.7,1.4) {$e_x$};\node at (3,1.4) {$e_{-x}$};
		\node at (-4.5,2.9) {$e_xe_y$};\node at (-1.8,3) {$e_xe_{-y}$};
		\node at (0.2,3) {$e_{-x}e_{iy}$};\node at (3,3) {$e_{-x}e_{-iy}$};
		\node at (-5.3,4.8) {$e_xe_ye_z$};\node at (-3.5,4.8) {$e_xe_ye_{-z}$};\node at (-1.4,4.8) {$e_xe_{-y}e_z$};
		\node at (0.8,4.8) {$e_xe_{-y}e_{-z}$};\node at (2.5,4.8) {$e_{-x_1}$};
		\node at (-6.2, 0) {$G_{0}$};
		\node at (-6.2, 1.5) {$G_1$};
		\node at (-6.2, 3) {$G_2$};
		\node at (-6.2, 4.5) {$G_3$};
		\end{tikzpicture}
	\end{center}
	\noindent
	{\it{Irreducible representation corresponding to the primitive central idempotent $e_1 = e_{-x}$}}: 
	Let $(\rho_1, V_{\rho_1})$ be the irreducible representation corresponding to $e_1$. Then the subspace $\langle{e_{-x}e_{iy}, ze_{-x}e_{-iy}}\rangle$ of $\mathbb{C}[G_3]$ can serve as a model for $V_{\rho_1}$. 
	The action of $\rho_{1}$ on the basis $\{e_{-x}e_{iy}, ze_{-x}e_{-iy}\}$ is given by
	\begin{align*}
	\rho_1(x) & \colon e_{-x}e_{iy} \mapsto -e_{-x}e_{iy},\\
	& ze_{-x}e_{iy} \mapsto -ze_{-x}e_{iy},\\
	\rho_1(y) & \colon e_{-x}e_{iy} \mapsto ie_{-x}e_{iy}, \\
	& ze_{-x}e_{iy} \mapsto -ize_{-x}e_{iy},\\
	\rho_1(z) & \colon  e_{-x}e_{iy} \mapsto ze_{-x}e_{iy},\\
	& ze_{-x}e_{iy} \mapsto -e_{-x}e_{iy}.
	\end{align*}
	Therefore w.r.t. the basis $\{ e_{-x}e_{iy}, ze_{-x}e_{-iy}\}$, the matrix representation of $\rho_1$ is:
	\begin{equation*}
	[\rho_1(x)] = \begin{bmatrix}
	-1 & 0\\
	0 & -1
	\end{bmatrix}, \hskip3mm 
	[\rho_1(y)] =
	\begin{bmatrix}
	i & 0\\
	0 & -i
	\end{bmatrix}, \hskip3mm 
	[\rho_1(z)] =
	\begin{bmatrix}
	0 & -1\\
	1 & 0
	\end{bmatrix}.
	\end{equation*}
	The irreducible representations corresponding to the primitive central idempotents $e_xe_ye_z,$ $e_xe_ye_{-z}, e_xe_{-y}e_z, e_xe_{-y}e_{-z}$ 
	are all 1-dimensional, and can be obtained in similar way. 
	In the following table, $e$ denotes a 
	primitive central idempotent in $\mathbb{C}[G_3]$, $(\rho,V_{\rho})$ denotes the corresponding irreducible representation, and the model space for the representation space $V_{\rho}$ is given as a subspace of $\mathbb{C}[G_3]$. 
	\begin{center}
		\small
		\begin{tabular}{||c|c|c|c||}
			\hline \hline
			$e$ & $(\rho,V_{\rho})$ & Model for $V_{\rho}$ & Action of $\rho$ on $V$\\\hline\hline
			$e_{2} = e_xe_ye_z$ &  $(\rho_2,V_{\rho_2})$ & $\langle v\rangle = \langle{e_xe_ye_z}\rangle$ & $\rho_2(x): v\mapsto v, \hskip1mm \rho_2(y) :v\mapsto v,\hskip1mm  \rho_2(z):v\mapsto v$ \\
			$e_{3} = e_xe_ye_{-z}$ &  $(\rho_3,V_{\rho_3})$ & $\langle v\rangle = \langle{e_xe_ye_{-z}}\rangle$ & $\rho_3(x): v\mapsto v, \hskip1mm \rho_3(y) :v\mapsto v,\hskip1mm  \rho_3(z):v\mapsto -v$ \\
			$e_{4} = e_xe_{-y}e_z$ &  $(\rho_4,V_{\rho_4})$ & $\langle v\rangle = \langle{e_xe_{-y}e_z}\rangle$ & $\rho_4(x): v\mapsto v, \hskip1mm \rho_4(y) :v\mapsto -v,\hskip1mm  \rho_4(z):v\mapsto v$ \\
			$e_{5} = e_xe_{-y}e_{-z}$ &  $(\rho_5,V_{\rho_5})$ & $\langle v\rangle = \langle{e_xe_{-y}e_{-z}}\rangle$ & $\rho_5(x): v\mapsto v, \hskip1mm \rho_5(y) :v\mapsto -v,\hskip1mm  \rho_5(z):v\mapsto -v$
			\\\hline\hline
		\end{tabular}
	\end{center}
\end{example}
\section{Special linear group $\mathrm{SL}_2(3)$}
\begin{example}\rm
	We construct the irreducible representations of $\mathrm{SL}_2(3)$ over $\mathbb{C}$ by using Clifford theory.
	Let us consider the maximal subnormal series
	\begin{equation}
	\langle e \rangle = G_{o} < C_2 = G_1 < C_4 = G_2 < Q_{8} = G_3 < \mathrm{SL}_2(3) = G_4.
	\end{equation}
	The long presentation associated with the series $(8.4.1)$ is
	\begin{align*}
	\mathrm{SL}_2(3) = \langle x,y,z,t \,|\, x^{2} = 1, y^{2} = x, y^{2} = z^{2}, z^{-1}yz = xy, t^{3} = 1, t^{-1}yt = z, t^{-1}zt = yz \rangle.
	\end{align*}
	The group $Q_8$ is a normal subgroup of index $3$ in $\mathrm{SL}_2(3)$. Now we construct the irreducible representations of $\mathrm{SL}_2(3)$ starting with the irreducible representations of $Q_{8}$. 
	
	Let $\rho_1$, $\rho_2$, $\rho_3$, $\rho_4$, $\rho_5$ be the five inequivalent irreducible representation $Q_{8}$ obtained in Example $(8.3)$. The trivial representation $\rho_1$ extends to three irreducible representations of $\mathrm{SL}_2(3)$, and are given by 
	$$\theta_1(t) = 1, \hskip3mm \theta_2(t) = \omega, \hskip3mm \theta_{3(t)} = \omega^{2}, \hskip5mm \mathrm{where} \, \omega = e^{2\pi i/3}.$$ 
	
	Notice that $\rho_2,\rho_3, \rho_4$ are conjugate to each other, so they induce the same irreducible representation $\theta_4$ of degree 3, and is given by
	\begin{center}
		$[\theta_4(x)] = I_{3},  \hskip3mm
		[\theta_4(y)] =
		\begin{bmatrix}
		1 & 0 & 0\\
		0  & -1  & 0\\
		0 & 0 & -1
		\end{bmatrix}, \hskip3mm
		[\theta_4(z)] = 
		\begin{bmatrix}
		-1 & 0 & 1\\
		0  & 1  & 0\\
		0 & 0 & -1 
		\end{bmatrix},$
	\end{center}
	\begin{center}
		$[\theta_4(t)] = 
		\begin{bmatrix}
		0& 0 & 1\\
		1  & 0  & 0\\
		0 & 1 & 0
		\end{bmatrix}$.
	\end{center}
	Since $\rho^{t}_{5} \cong \rho_{5}$, then $\rho_{5}$ extends to three irreducible representations of $\mathrm{SL}_2(3)$, let these extensions be $\theta_{5}, \theta_{6}, \theta_{7}$. 
	The matrix representations of $\theta_{5}, \theta_{6}, \theta_{7}$ are given by
	\begin{equation*}
	\theta_k(x)= 
	\begin{bmatrix}
	-1 & 0 \\
	0 & -1
	\end{bmatrix}, \hskip1mm 
	\theta_k(y)= 
	\begin{bmatrix}
	i & 0 \\
	0 & -i 
	\end{bmatrix},  \hskip1mm
	\theta_k(z)= 
	\begin{bmatrix}
	0 & -1 \\
	1 & 0 
	\end{bmatrix} \hskip3mm (k = 5,6,7),
	\end{equation*}
	and 
\small
	\begin{equation*}
	\theta_{5}(t) =\begin{bmatrix}
	\frac{-1+i}{2} & \frac{-1 - i}{2}\\
	\frac{1-i}{2} & \frac{-1-i}{2}\end{bmatrix}, 
	\theta_{6}(t) = \omega\begin{bmatrix}
	\frac{-1+i}{2} & \frac{-1 - i}{2}\\
	\frac{1-i}{2} & \frac{-1-i}{2}\end{bmatrix}, 
	\theta_{7}(t) = \omega^{2}\begin{bmatrix}
	\frac{-1+i}{2} & \frac{-1 - i}{2}\\
	\frac{1-i}{2} & \frac{-1-i}{2}\end{bmatrix}.
	\end{equation*}
	\normalsize
	\vspace{3mm}
	\noindent
	$(2)$ We construct the PCI-diagram associated with the series $(8.4.1)$.\\
	{{ Primitive central idempotents in $\mathbb{C}[Q_{8}]$ are}} 
$$e_xe_ye_z, \hskip3mm e_xe_ye_{-z}, \hskip3mm e_xe_{-y}e_z, \hskip3mm e_xe_{-y}e_{-z}, \hskip3mm e_{-x}e_{iy} + e_{-x}e_{-iy} = 1 - e_x.$$ 
	{{ Primitive central idempotents in $\mathbb{C}[\mathrm{SL}_2(3)]$ are}}\\ Let $\overline{C}(t)$ denotes the conjugacy class sum of $t$ as an element of $\mathbb{C}[\mathrm{SL}_2(3)]$. 
	\begin{center}
	\small
	\begin{tabular}{ll}

$u_1 = \frac{e_{-x}}{3} + \frac{1}{3}\begin{Bmatrix}\frac{\overline{C}(t)}{-2} + 
		\begin{pmatrix}\frac{\overline{C}(t)}{-2}\end{pmatrix}^{2}\end{Bmatrix}e_{-x},$
		& 
\hskip1mm $u_2 = \frac{e_{-x}}{3} + \frac{1}{3}\begin{Bmatrix}\frac{\overline{C}(t)}{-2} + 
		\omega\begin{pmatrix}\frac{\overline{C}(t)}{-2}\end{pmatrix}^{2}\end{Bmatrix}e_{-x}, $\\\\

$u_3 = \frac{e_{-x}}{3} + \frac{1}{3}\begin{Bmatrix}\frac{\overline{C}(t)}{-2} + 
		\omega^2\begin{pmatrix}\frac{\overline{C}(t)}{-2}\end{pmatrix}^{2}\end{Bmatrix}e_{-x},$ & \normalsize
		\hskip1mm
		$u_4 = e_xe_ye_{-z}+ e_xe_{-y}e_z + e_xe_{-y}e_{-z}, $\\\\
		$u_{5} = e_xe_ye_ze_{wt},$  & 
		\hskip1mm
		$u_{6} = e_xe_ye_ze_{w^{2}t},\hskip3mm u_{7} = e_xe_ye_ze_t,$
	\end{tabular}
\end{center}
where $e_X=\frac{1+X}{2}$; $X\in \{\pm\, x, \pm\, y,  \pm\,  z\}$; $e_Y=\frac{1+Y+Y^2}{3}$; $Y \in \{t,\, \omega t,\, \omega^2 t\}$.
\\
\noindent
{\it{PCI-diagram of $\mathrm{SL}_2(3)$}}:
The PCI-diagram associated with the subnormal series $(8.4.1)$ is
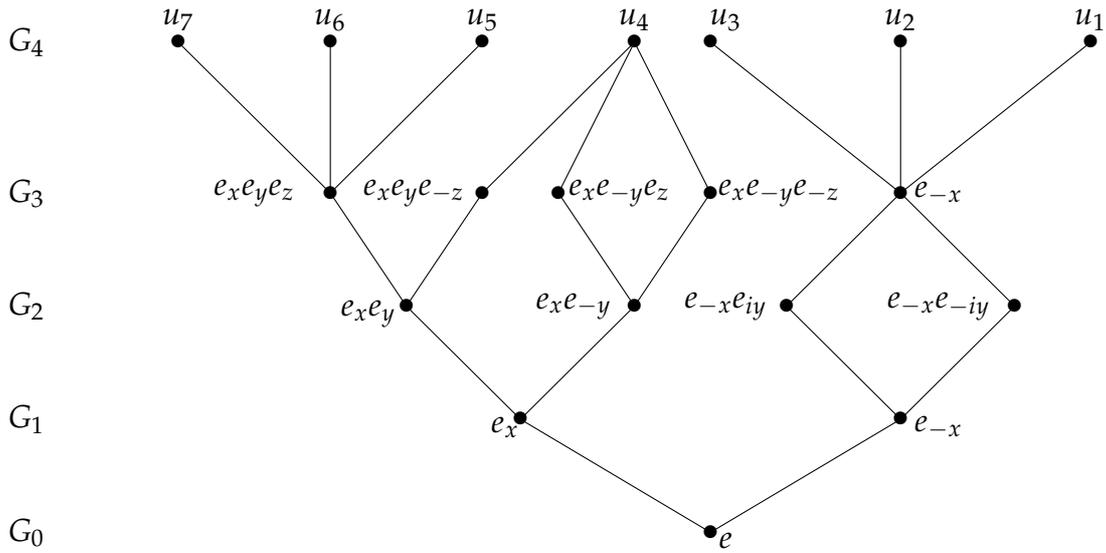
\begin{center}
	\begin{tikzpicture}
	\node at (0,0) {$\bullet$};
	\node at (-2.5,1.5) {$\bullet$};\node at (2.5,1.5) {$\bullet$};
	\node at (-4,3) {$\bullet$};\node at (-1,3) {$\bullet$};\node at (1,3) {$\bullet$};\node at (4,3) {$\bullet$};
	\node at (-5,4.5) {$\bullet$};\node at (-3,4.5) {$\bullet$};\node at (-2,4.5) {$\bullet$};\node at (-0.0,4.5) {$\bullet$};\node at (2.5,4.5) {$\bullet$};
	\draw (0,0)--(-2.5,1.5); \draw (0,0)--(2.5,1.5);
	\draw (-2.5,1.5)--(-4,3);\draw (-2.5,1.5)--(-1,3);\draw (2.5,1.5)--(4,3);\draw (2.5,1.5)--(1,3);
	\draw (-4,3)--(-5,4.5);\draw (-4,3)--(-3,4.5);
	\draw (-5, 4.5) -- (-7, 6.5); \draw (-5, 4.5) -- (-5, 6.5); \draw (-5, 4.5) -- (-3, 6.5);
	\node at (-7, 6.5) {$\bullet$}; \node at (-5, 6.5) {$\bullet$}; \node at (-3, 6.5) {$\bullet$};
	\draw (-1,3)--(-2,4.5);\draw (-1,3)--(-0.0,4.5);
	\draw (-3, 4.5)--(-1, 6.5) ;
	\draw (-2, 4.5)--(-1, 6.5) ;
	\draw (0, 4.5)--(-1, 6.5) ;
	\node at (-1, 6.5) {$\bullet$};
	\node at (-1, 6.8) { $u_4$};
	\draw (2.5, 4.5)--(0, 6.5) ;
	\draw (2.5, 4.5)--(2.5, 6.5) ;
	\draw (2.5, 4.5)--(5, 6.5) ;
	\node at (0, 6.5) {$\bullet$};
	\node at (2.5, 6.5) {$\bullet$};
	\node at (5, 6.5) {$\bullet$};
	\draw (1,3)--(2.5,4.5);\draw (4,3)--(2.5,4.5);
	\node at (0.2,-0.1) {$e$};
	\node at (-2.7,1.4) {$e_x$};\node at (3,1.4) {$e_{-x}$};
	\node at (-4.5,2.9) {$e_xe_y$};\node at (-1.8,3) {$e_xe_{-y}$};
	\node at (0.2,3) {$e_{-x}e_{iy}$};\node at (3,3) {$e_{-x}e_{-iy}$};
	\node at (-6,4.5) {$e_xe_ye_z$};\node at (-3.9,4.5) {$e_xe_ye_{-z}$};\node at (-1.2,4.5) {$e_xe_{-y}e_z$};
	\node at (.9,4.5) {$e_xe_{-y}e_{-z}$};\node at (3,4.5) {$e_{-x}$};
	\node at (-7, 6.8) {$u_{7}$};
	\node at (-5, 6.8) {$u_{6}$};
	\node at (-3, 6.8) {$u_{5}$};
	\node at (0.2, 6.8) { $u_3$};
	\node at (2.5, 6.8) { $u_2$};
	\node at (5 , 6.8) { $u_1$};
	\node at (-9, 0) {$G_{0}$};
	\node at (-9, 1.5) {$G_1$};
	\node at (-9, 3) {$G_2$};
	\node at (-9, 4.5) {$G_3$};
	\node at (-9, 6.5) {$G_4$};
	\end{tikzpicture}
\end{center}
\end{example}
\section{Alternating group $A_4$}
\begin{example}\rm
We construct the irreducible representations of $A_{4}$ over $\mathbb{C}$ by using our algorithm (see Chapter$6$).\\
Let us consider the maximal subnormal series
\begin{equation}
\langle e \rangle = G_0 < C_ 2 = G_1 < C_2 \times C_2 = G_2 < A_4 = G_3.
\end{equation}
The long presentation associated with the series $(8.5.1)$ is
$$ A_4 = \langle{ x, y, z | x^{2} = y^{2} = 1, xy = yx, z^{3} = 1, z^{-1}xz = y, z^{-1}yz = xy}\rangle.$$
Hence $G_0 = \{1\}$, $G_1 = \langle {x} \rangle$, $G_2 = \langle {x, y} \rangle$ and $G_3 = \langle  {x, y, z} \rangle$.\\ 
The primitive central idempotents in $\mathbb{C}[G_2]$, i.e., in $\mathbb{C}[C_{2} \times C_{2}]$ are 
$$e_xe_y, \hskip3mm e_xe_{-y}, \hskip3mm e_{-x}e_y, \hskip3mm e_{-x}e_{-y}.$$\\
The primitive central idempotents of $\mathbb{C}[G_3]$, i.e., in $\mathbb{C}[A_{4}]$ are 
$$e_xe_ye_z, \hskip2mm e_xe_ye_{\omega z},\hskip2mm e_xe_ye_{\omega^{2}z}, \hskip2mm e_{-x}e_y+e_xe_{-y}+e_{-x}e_{-y} = 1 - e_xe_y,$$
where $e_X =\frac{1+X}{2} \,\,(X\in \{\pm\, x, \pm\, y\}),\,\,\,  e_Y =\frac{1+Y+Y^2}{3} \,\,(Y\in \{z, \omega z, \omega^2 z\})$.\\\\
\noindent
{\it{PCI-diagram of $A_4$}:} The PCI-diagram associated with the series $(8.5.1)$ is
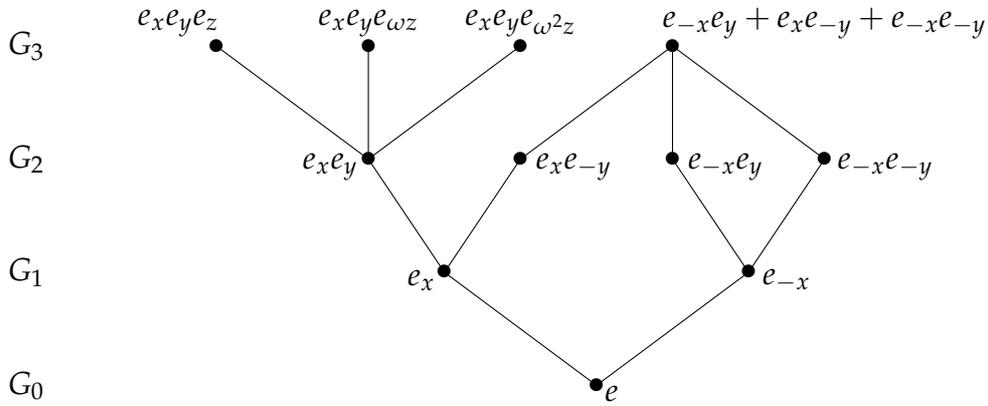
\begin{center}
	\begin{tikzpicture}
	\draw (-2,1.5)--(0,0)--(2,1.5);
	\draw (-1,3)--(-2,1.5)--(-3,3);
	\draw (1,3)--(2,1.5)--(3,3);
	\draw (-5,4.5)--(-3,3)--(-3,4.5);
	\draw (-3,3)--(-1,4.5);
	\draw (-1,3)--(1,4.5);
	\draw (1,3)--(1,4.5)--(3,3);
	\node at (0,0) {$\bullet$};
	\node at (-2,1.5) {$\bullet$};\node at (2,1.5) {$\bullet$};
	\node at (-3,3) {$\bullet$};\node at (-1,3) {$\bullet$};\node at (1,3) {$\bullet$};\node at (3,3) {$\bullet$};
	\node at (-5,4.5) {$\bullet$};\node at (-3,4.5) {$\bullet$};\node at (-1,4.5) {$\bullet$};\node at (1,4.5) {$\bullet$};
	\node at (0.2,-0.1) {$e$};
	\node at (-2.3,1.4) {$e_x$};\node at (2.5,1.4) {$e_{-x}$};
	\node at (-3.5,2.9) {$e_xe_y$};\node at (-0.3,2.9) {$e_xe_{-y}$};
	\node at (1.7,2.9) {$e_{-x}e_y$};\node at (3.8,2.9) {$e_{-x}e_{-y}$};
	\node at (-5.5,4.8) {$e_xe_ye_z$};\node at (-3,4.8) {$e_xe_ye_{\omega z}$};\node at (-1,4.8) {$e_xe_ye_{\omega^2 z}$};
	\node at (3,4.8) {$e_{-x}e_y+e_xe_{-y}+e_{-x}e_{-y}$};
	\node at (-7.5,0) {$G_0$};\node at (-7.5,1.5) {$G_1$};\node at (-7.5,3) {$G_2$};\node at (-7.5,4.5) {$G_3$};
	\end{tikzpicture}
\end{center}
Here, $e_X =\frac{1+X}{2} \,\,(X\in \{\pm\, x, \pm\, y\}),\,\,\,  e_Y =\frac{1+Y+Y^2}{3} \,\,(Y\in \{z, \omega z, \omega^2 z\})$\\\\
\noindent
{\it{Irreducible representation corresponding to the primitive central idempotent $e_1 = 1 - e_xe_y$:}} 
Let $(\rho_1,V_1)$ be the irreducible representation corresponding to $e_1$. The subspace $\langle{e_xe_{-y}, ze_xe_{-y}, z^{2}e_xe_{-y}}\rangle$ of $\mathbb{C}[G_3]$ can serve as a model for $V_{\rho_1}$. 
The action of $\rho$ on the basis $\{e_xe_{-y}, ze_xe_{-y}, z^{2}e_xe_{-y}\}$ is  given by
\begin{align*}
\rho_1(x): & e_xe_{-y} \mapsto xe_xe_{-y} = e_xe_{-y},\\  
& ze_xe_{-y} \mapsto xze_xe_{-y} = -ze_xe_{-y},\\             
&z^{2}e_xe_{-y} \mapsto xz^{2}e_xe_{-y} = -z^{2}e_xe_{-y},             
\end{align*}
\begin{align*}
\rho_1(y): & e_xe_{-y} \mapsto ye_xe_{-y} = -e_xe_{-y},\\ 
&ze_xe_{-y} \mapsto yze_xe_{-y} = -ze_xe_{-y},\\
& z^{2}e_xe_{-y} \mapsto yz^{2}e_xe_{-y} = z^{2}e_xe_{-y},
\end{align*}

and 
\begin{align*}
\rho_1(z) : & e_xe_{-y} \mapsto ze_xe_{-y} = e_xe_{-y},\\ 
& ze_xe_{-y} \mapsto z^{2}e_xe_{-y},\\ 
& z^{2}e_xe_{-y} \mapsto z^{3}e_xe_{-y} = e_xe_{-y}.
\end{align*}
Thus, the matrix representation of $\rho_1$ w.r.t. the basis $\{e_xe_{-y}, ze_xe_{-y}, z^{2}e_xe_{-y}\}$ is 
\begin{equation*}
[\rho_1(x)]=
\begin{bmatrix}
1 & 0  & 0\\
0 & -1 & 0\\
0 & 0  & -1
\end{bmatrix}, 
[\rho_1(y)]=
\begin{bmatrix}
-1 & 0 & 0\\
0 & -1 & 0\\
0 & 0 & 1
\end{bmatrix}, 
[\rho_1(z)] =
\begin{bmatrix}
0 & 0 & 1\\
1 & 0 & 0\\
0 & 1 & 0
\end{bmatrix}.
\end{equation*}
The representations corresponding to primitive central idempotents  $e_xe_ye_z$, $e_xe_ye_{\omega z}$ and 
$e_xe_ye_{\omega^2z}$ are all $1$-dimensional, and can be obtained in similar way. In the following table, $e$ denotes a 
primitive central idempotent in $\mathbb{C}[G_3]$, $(\rho,V_{\rho})$ denotes the corresponding irreducible representation, and the model space for the representation space $V_{\rho}$ is given as a subspace of $\mathbb{C}[G_3]$. 
\begin{center}
	\small
	\begin{tabular}{||c|c|c|c||}
		\hline \hline
		$e$ & $(\rho,V_{\rho})$ & Model for $V_{\rho}$ & Action of $\rho$ on $V_{\rho}$\\\hline \hline
		$e_{2} = e_xe_ye_z$ &  $(\rho_2,V_{\rho_2})$ & $\langle v\rangle = \langle{e_xe_ye_z}\rangle$ & $\rho_2(x): v\mapsto v, \hskip1mm \rho_2(y) :v\mapsto v,\hskip1mm  \rho_2(z):v\mapsto v$ \\
		$e_{3} = e_xe_ye_{\omega z}$ &  $(\rho_3,V_{\rho_3})$ & $\langle v\rangle = \langle{e_xe_ye_{\omega z}}\rangle$ & $\rho_3(x): v\mapsto v, \hskip1mm \rho_3(y) :v\mapsto v,\hskip1mm  \rho_3(z):v\mapsto \omega^2 v$ \\
		$e_{4} = e_xe_ye_{\omega^2 z}$ &  $(\rho_4,V_{\rho_4})$ & $\langle v\rangle = \langle{e_xe_ye_{\omega^2 z}}\rangle$ & $\rho_4(x): v\mapsto v, \hskip1mm \rho_4(y) :v\mapsto v,\hskip1mm  \rho_4(z):v\mapsto \omega v$\\
		\hline \hline
	\end{tabular}
\end{center}
\end{example}
\section{Symmetric group $S_4$}
\begin{example}\rm
We construct the irreducible representations of $S_{4}$ over $\mathbb{C}$ by using Clifford theory.
Let us consider the maximal subnormal series: 
\begin{equation}
\langle e \rangle = G_0  < C_2 = G_1 < C_{2} \times C_{2} = G_2 < A_4 = G_3 < S_4 = G_4 = G.
\end{equation}
The long presentation associated with the above series $(8.6.1)$ is 
\begin{center}
	$S_4 = \langle x, y, z, t \,|\, x^{2} = y^{2} = 1, \, y^{-1}xy = x, \, z^{3} = 1, \, z^{-1}xz = y, \, z^{-1}yz = xy,\, t^{2} = 1, \, t^{-1}xt = x, 
	\, t^{-1}yt = xy, \, t^{-1}zt = z^{-1} \rangle$
\end{center}
We may consider $x = (12)(34), y = (13)(24), z = (123), t = (12)$. By construction,
$$\{e\} = G_0, \hskip2mm G_1 = \langle{x}\rangle , \hskip2mm  G_2 = \langle{x,y}\rangle, \hskip2mm  G_3 = \langle{x,y,z}\rangle, \hskip2mm  G_4 = \langle x,y,z,t \rangle.$$

Let $\rho_i$ ($i=1,2,3,4$) denote the matrix representations of $G_3 = A_4$ obtained in Example $8.5$, and starting from them, we construct the irreducible matrix representations of $S_4$. 
The trivial representation $\rho_2$ of $A_4$  extends in two ways to a representation of $S_4$, one is the trivial  and the other is the sign representation of $S_4$, and we denote them by $\theta_1$ and $\theta_2$. 
Since $A_4$ is the commutator subgroup of $S_4$, they are the only $1$-dimensional representations of $S_4$. 
Since $\rho_3$ is not equivalent to $\rho_3^t$, and $\rho_{4} \cong \rho_3^t$, then  $\rho_3$, $\rho_4$ induce the same irreducible representation of $S_{4}$, denoted by $\theta_3$, and  is given by
$$[\theta_3(x)] = \begin{bmatrix}
1 &  0\\
0 & 1\end{bmatrix}, \hskip3mm 
[\theta_3(y)] = \begin{bmatrix}
1 &  0\\
0 & 1\end{bmatrix},\hskip3mm 
[\theta_3(z)] =\begin{bmatrix}
w &  0\\
0 & w^{2}\end{bmatrix}, 
[\theta_3(t)] =\begin{bmatrix}
0 &  1\\
1 & 0\end{bmatrix}.$$
Consider the representation $\rho_1$ of $A_4$. For the purpose, we rewrite it here
\begin{equation*}
[\rho_1(x)]=
\begin{bmatrix}
1 & 0  & 0\\
0 & -1 & 0\\
0 & 0  & -1
\end{bmatrix}, 
[\rho_1(y)] =
\begin{bmatrix}
-1 & 0 & 0\\
0 & -1 & 0\\
0 & 0 & 1
\end{bmatrix}, 
[\rho_1(z)] =
\begin{bmatrix}
0 & 0 & 1\\
1 & 0 & 0\\
0 & 1 & 0
\end{bmatrix}.
\end{equation*}
Since ${\rho_1} \cong \rho_{1}^{t}$, then $\rho_1$ extends in two ways to an irreducible representation of $S_4$, denote them by $\theta_4$, and $\theta_5$. 
To determine $\theta_4$ and $\theta_5$, it is enough to determine $\theta_4(t)$ and $\theta_5(t)$. For simplicity, let $Y_i=\theta_i(t)$, $i=4,5$. Note that $Y_i \in $ Int$(\rho_1^t, \rho_1)$, the space of intertwining operators between 
$\rho_1^t$ and $\rho_1$ with $Y_i^2 = I_3$, for $i = 4, 5$.  Now using the relations 
$$t^{-1}xt=x, \hskip5mm t^{-1}yt=xy, \hskip5mm t^{-1}zt=z^2$$
the images of $\rho_1^t$ at $x,y,z$ can be calculated from that of $\rho_1$, which are given by 
\begin{equation*}
[\rho_1^t(x)] =
\begin{bmatrix}
1 & 0 & 0\\
0 & -1 & 0\\
0 & 0 & -1\end{bmatrix}, 
\hskip3mm 
[\rho _1^{t}(y)] = \begin{bmatrix}
-1 & 0 & 0\\
0 & 1 & 0\\
0 & 0 & -1\end{bmatrix},  \hskip3mm 
[\rho _1^{t}(z)] = 
\begin{bmatrix}
0 & 1 & 0\\
0 & 0 & 1\\
1 & 0 & 0\end{bmatrix}.
\end{equation*}
Now $Y_i o \rho_1^t(g)=\rho_1(g) o Y_i$ for $g\in\langle x,y,z\rangle$ ($i=4,5$), considering these equations for $g=x,y,z$ and 
$Y_i^2=I_3$, we get  
\begin{equation*}
Y_4
=\begin{bmatrix}
1 & 0 & 0\\
0 & 0 & 1\\
0 & 1 & 0\end{bmatrix}
\mbox{ and } Y_5 = \begin{bmatrix}
-1 & 0 & 0\\
0  & 0  & -1\\
0 & -1 & 0\end{bmatrix}
\end{equation*} Thus, $\theta_i(t) = Y_{i}$ ($i=4,5$) are the two extensions of the representation $\rho_1$.

We have obtained the irreducible representations of $S_4$ by using Clifford theory, and now we obtain them by using our algorithm (see Chapter$6$). 

The primitive central idempotents in $\mathbb{C}[A_4]$ are
$$e_xe_ye_z,  \hskip3mm e_xe_ye_{wz}, \hskip3mm e_xe_ye_{w^2z}, \hskip3mm e_xe_{-y} + e_{-x}e_y + e_{-x}e_{-y} = 1 - e_xe_y.$$

Let $\overline{C}(t)\in\mathbb{C}[S_4]$ denote the conjugacy class sum of $t$ in $S_4$. The complete set of primitive central idempotents in $\mathbb{C}[S_4]$ are
\begin{align*}
& b_1 = \frac{(1 - e_xe_y)}{2} + \frac{1}{2}\begin{Bmatrix}  \frac{\overline{C}(t)(1 - e_{x}e_y)}{2}\end{Bmatrix},  \hskip3mm 
b_2 = \frac{(1 - e_xe_y)}{2} + \frac{1}{2}\begin{Bmatrix} \frac{\overline{C}(t)(1 - e_{x}e_y)}{-2}\end{Bmatrix}, \\
& b_3 = e_xe_ye_ze_t,  \hskip5mm b_4 = e_xe_ye_ze_{-t}, \hskip5mm b_5 = e_xe_ye_{wz} +  e_xe_ye_{w^2z}.
\end{align*}
{\it{PCI-diagram of $S_4$}}:
The PCI-diagram associated with the subnormal series $(8.6.1)$ is
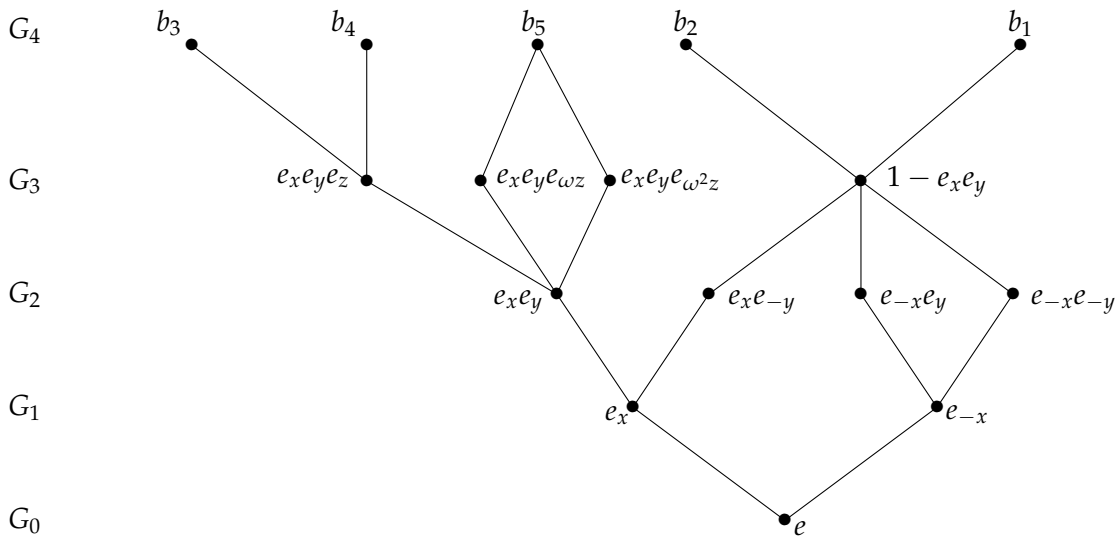
\begin{center}
	\small
	\begin{tikzpicture}
	\draw (-2,1.5)--(0,0)--(2,1.5);
	\draw (-1,3)--(-2,1.5)--(-3,3);
	\draw (1,3)--(2 ,1.5)--(3,3);
	\draw (-5.5,4.5)--(-3,3)--(-2.3,4.5);
	\draw (-3,3)--(-4, 4.5);
	\draw (-1,3)--(1,4.5);
	\draw (1,3)--(1,4.5)--(3,3);
	\draw (-7.8, 6.3)--(-5.5,4.5)--(-5.5,6.3);
	\draw (-4, 4.5)--(-3.25,6.3)--(-2.3,4.5);
	\draw (-1.3, 6.3)--(1,4.5)--(3.1,6.3);
	\node at (-5.5, 6.3) {$\bullet$};
	\node at (-3.25, 6.3) {$\bullet$};
	\node at (-7.8, 6.3) {$\bullet$};
	\node at (-1.3, 6.3) {$\bullet$};
	\node at (3.1, 6.3) {$\bullet$};
	\node at (0,0) {$\bullet$};
	\node at (-2,1.5) {$\bullet$};\node at (2,1.5) {$\bullet$};
	\node at (-3,3) {$\bullet$};\node at (-1,3) {$\bullet$};\node at (1,3) {$\bullet$};\node at (3,3) {$\bullet$};
	\node at (-5.5,4.5) {$\bullet$};\node at (-4,4.5) {$\bullet$};\node at (-2.3,4.5) {$\bullet$};\node at (1,4.5) {$\bullet$};
	\node at (0.2,-0.1) {$e$};
	\node at (-2.2,1.4) {$e_x$};\node at (2.4,1.4) {$e_{-x}$};
	\node at (-3.5,2.9) {$e_xe_y$};\node at (-0.3,2.9) {$e_xe_{-y}$};
	\node at (1.7,2.9) {$e_{-x}e_y$};\node at (3.8,2.9) {$e_{-x}e_{-y}$};
	\node at (-6.2,4.5) {$e_xe_ye_z$};\node at (-3.2 ,4.5) {$e_xe_ye_{\omega z}$};\node at (-1.5,4.5) {$e_xe_ye_{\omega^2 z}$};
	\node at (2,4.5) {$1 - e_xe_y$};
	\node at (-8.1,6.6) {$b_{3}$};\node at (-5.8 ,6.6) {$b_{4}$};\node at (-3.3, 6.6) {$b_{5}$};
	\node at (-1.3, 6.6) {$b_2$ }; 
	\node at (3.1, 6.6) {$b_1$}; 
	\node at (-10, 0) {$G_{0}$};
	\node at (-10, 1.5) {$G_1$};
	\node at (-10, 3) {$G_2$};
	\node at (-10, 4.5) {$G_3$};
	\node at (-10, 6.5) {$G_4$};
	\end{tikzpicture}
\end{center}
\noindent
{\it{Irreducible representation corresponding to the primitive central idempotent $b_1$}}: Let $(\rho_1, V_{\rho_1})$ be the irreducible representation corresponding to the primitive central idempotent $b_1$. 
The subspace $\langle{e_xe_{-y}b_1, ze_xe_{-y}b_1, z^{2}e_xe_{-y}b_1}\rangle$ of $\mathbb{C}[G_4]$ can serve as a model for $V_{\rho_1}$. 
The action of  $\rho_1$ on the basis $\{e_xe_{-y}b_1, ze_xe_{-y}b_1, z^{2}e_xe_{-y}b_1 \}$ is given by 
\small
\begin{equation*}
[\rho_1(x)] =
\begin{bmatrix}
1 & 0  & 0\\
0 & -1 & 0\\
0 & 0  & -1
\end{bmatrix},
[\rho_1(y)] =
\begin{bmatrix}
-1 & 0 & 0\\
0 & -1 & 0\\
0 & 0 & 1
\end{bmatrix},
[\rho_1(z)] =
\begin{bmatrix}
0 & 0 & 1\\
1 & 0 & 0\\
0 & 1 & 0
\end{bmatrix},
[\rho_1(t)] =  \begin{bmatrix}
1 & 0 & 0\\
0 & 0 & 1\\
0 & 1 & 0\end{bmatrix}.
\end{equation*}
\normalsize
\vskip2mm\noindent
{\it Irreducible representation corresponding to the primitive central idempotent $b_2$}: Let $(\rho_2, V_{\rho_2})$ be the irreducible representation corresponding to the primitive central idempotent $b_2$. The subspace
$\langle{e_xe_{-y}b_2, ze_xe_{-y}b_2, z^{2}e_xe_{-y}b_2}\rangle$ of $\mathbb{C}[G_4]$ can serve as a model for $V_{\rho_2}$. 
The action of $\rho_2$ on the basis  $\{ e_xe_{-y}b_2, ze_xe_{-y}b_2, z^{2}e_xe_{-y}b_2 \}$ is given by
\small
\begin{equation*}
[\rho_2(x)] =
\begin{bmatrix}
1 & 0  & 0\\
0 & -1 & 0\\
0 & 0  & -1
\end{bmatrix}\hskip-1mm,
[\rho_2(y)] =
\begin{bmatrix}
-1 & 0 & 0\\
0 & -1 & 0\\
0 & 0 & 1
\end{bmatrix}\hskip-1mm,
[\rho_2(z)] =
\begin{bmatrix}
0 & 0 & 1\\
1 & 0 & 0\\
0 & 1 & 0
\end{bmatrix}\hskip-1mm,
[\rho_2(t)] = \begin{bmatrix}
-1 & 0 & 0\\
0 & 0 & -1\\
0 & -1 & 0\end{bmatrix}\hskip-1mm.
\end{equation*}
\normalsize

\noindent{\it{Irreducible representation corresponding to the primitive central idempotent $b_3 = e_xe_ye_ze_t$}}: Let $(\rho_3,V_{\rho_3})$ be the irreducible representation corresponding to $b_3$. Then $\dim_{\mathbb{C}} V_{\rho_3}=1$. The subspace
$\langle e_xe_ye_ze_t\rangle$ of $\mathbb{C}[G_4]$ can serve as a model for $V_{\rho_3}$. The action of $\rho_3$ on $\{ e_xe_ye_ze_t \}$ is given by
\begin{align*}
\rho_3(x): e_xe_ye_ze_t \mapsto xe_xe_ye_ze_t = e_xe_ye_ze_t,  & \rho_3(y): e_xe_ye_ze_t \mapsto ye_xe_ye_ze_t = e_xe_ye_ze_t,\\
\rho_3(z): e_xe_ye_ze_t \mapsto ze_xe_ye_ze_t = e_xe_ye_ze_t , & \rho_3(t): e_xe_ye_ze_t \mapsto te_xe_ye_ze_t = e_xe_ye_ze_t .
\end{align*}
Therefore, the representation $\rho_3$ is the trivial representation of $S_4$.\\
\noindent
{\it{Irreducible representation corresponding to the primitive central idempotent $b_4 = e_xe_ye_ze_{-t}$}}: Let $(\rho_4,V_{\rho_4})$ be the irreducible representation corresponding to $b_4$. Then $\dim_{\mathbb{C}}V_{\rho_4}=1$. The subspace
$\langle e_xe_ye_ze_{-t}\rangle$ of $\mathbb{C}[G_4]$ can serve as a model for $V_{\rho_4}$. The action of $\rho_4$ on $\{ e_xe_ye_ze_{-t} \}$ is given by
\small
\begin{align*}
\rho_4(x):e_xe_ye_ze_{-t}\mapsto xe_xe_ye_ze_{-t}=e_xe_ye_ze_{-t},\hskip2mm &\rho_4(y): e_xe_ye_ze_{-t}\mapsto ye_xe_ye_ze_{-t} = e_xe_ye_ze_{-t},\\
\rho_4(z) \colon e_xe_ye_ze_{-t} \mapsto ze_xe_ye_ze_{-t} = e_xe_ye_ze_t, & \hskip2mm
\rho_4(t) \colon e_xe_ye_ze_{-t} \mapsto te_xe_ye_ze_{-t} = -e_xe_ye_ze_{-t} .
\end{align*}
\normalsize
Therefore, the representation of $\rho_3$ is given by
\begin{equation*}
\rho_4(x) = 1,  \hskip3mm \rho_4(y) = 1, \hskip3mm \rho_4(z) = 1, \hskip3mm \rho_4(t) = -1, 
\end{equation*}
and which is the sign representation of $S_{4}$.
\vskip3mm
\noindent
{\it Irreducible representation corresponding to the primitive central idempotent {$b_5 = e_xe_ye_{wz}$ $+$ $e_xe_ye_{w^{2}z}$}}: Let $(\rho_{5},V_{\rho_5})$ be the irreducible representation corresponding to $b_{5}$. Then $\dim_{\mathbb{C}}V_{\rho_5}=2$. The subspace $\langle e_xe_ye_{wz}, te_xe_ye_{wz}\rangle$ of $\mathbb{C}[G_4]$ can serve as a model for 
$V_{\rho_5}$. 
The action of $\rho_5$ on $\langle e_xe_ye_{wz}, te_xe_ye_{wz}\rangle$ is given by
\begin{align*}
\rho_{5}(x) : & e_xe_ye_{wz} \mapsto xe_xe_ye_{wz} = e_xe_ye_{wz} ,\\
& te_xe_ye_{wz} \mapsto xte_xe_ye_{wz} = txe_xe_ye_{wz} = te_xe_ye_{wz},
\end{align*}
and similarly, one can see that 
\small
\begin{align*}
\rho_{5}(y):  e_xe_ye_{wz} \mapsto  e_xe_ye_{wz} , & \hskip3mm\rho_{5}(z) : e_xe_ye_{wz} \mapsto w^{2}e_xe_ye_{wz} , &\rho_{5}(t) : e_xe_ye_{wz} \mapsto te_xe_ye_{wz},\\
te_xe_ye_{wz} \mapsto te_xe_ye_{wz}, &\hskip1cm  te_xe_ye_{wz} \mapsto wte_xe_ye_{wz}, & te_xe_ye_{wz} \mapsto  e_xe_ye_{wz}.
\end{align*}
\normalsize
Therefore w.r.t. the basis $\{ e_xe_ye_{wz}, te_xe_ye_{wz} \}$ the matrix representation of $\rho_{5}$ is 
$$[\rho_5(x)] =\begin{bmatrix}
1 &  0\\
0 & 1\end{bmatrix}, 
[\rho_5(y)] =\begin{bmatrix}
1 &  0\\
0 & 1\end{bmatrix}, 
[\rho_5(z)] =\begin{bmatrix}
w^{2} &  0\\
0 & w\end{bmatrix}, 
[\rho_5(t)] =\begin{bmatrix}
0 &  1\\
1 & 0\end{bmatrix}.$$
 
\end{example}
\chapter{Summary and Future work}
In this chapter, we present summary of the thesis and directions for future research work.
\section{Summary of the thesis}
This thesis started with an introduction which consists of definitions, motivation, main results of the thesis and organization of the thesis. All the results of the Chapter $2$, {'\it {Semisimple algebras}'} are well known.

In Chapter $3$, {'\it {$F$-conjugacy, $F$-character table, $F$-idempotents'}}, we presented a brief introduction to Schur's theory (see \cite{re98}, \cite{sc70}) on group representations over a field with characteristic $0$ or prime to $|G|$. We showed that $F$-conjugacy class of an element of order $n$ in a group $G$ is the union of certain conjugacy classes, and which is determined by the decomposition of $n$-th cyclotomic polynomial $\Phi_{n}(X)$ into irreducible polynomials over $F$. We gave a formula for computing the primitive central idempotents of a semisimple group algebra $F[G]$ in terms of $F$-characters and $F$-conjuagcy classes in $G$, which can be obtained from the $F$-character table.

In Chapter $4$, {'\it{Clifford theory'}}, we described the classical approach to inductive construction of the irreducible representations of a finite solvable group $\mathbb{C}$, and which is based on Clifford theory on group representations (see \cite{alg111}, \cite{cl50}, \cite{combi91}) with various examples. We noted some important consequences of this approach. 

In Chapter $5$, {'\it{Diagonal subalgebra'}}, we introduced the notion of a diagonal subalgebra (unique up to conjugacy) of a group algebra over an algebraically closed field of characteristic $0$. We gave an introduction to Gelfand-Tsetlin alegbra (see \cite{Mu106}) for an inductive chain of subgroups of a finite group with simple branching. A finite solvable group always has a subnormal series such that successive quotients are cyclic groups of prime order, and which is a multiplicity free chain of subgroups. So, one can consider the Gelfand-Tsetlin algebra associated with such a series, and in fact this is the diagonal subalgebra. We found a convenient set of generators of the Gelfand-Tsetlin algebra using a long system of generators. 
Apart from that we gave an indirect proof of the Berman's theorem (see \cite{berman120}), which gives an inductive construction of the primitive central idempotents of $\mathbb{C}[G]$, where $G$ is a finite solvable group.

In Chapter $6$, {'\it {Algorithmic construction of matrix representations of a finite solvable group over $\mathbb{C}$'}}, we gave an algorithm for constructing the irreducible matrix representations of a finite solvable group over $\mathbb{C}$ using a long presentation. 

For a finite solvable group $G$ of order $N = p_{1}p_{2}\dots p_{n}$, where $p_{i}$'s are primes, there always exists a subnormal series: $\langle {e} \rangle = G_{o} < G_{1} < \dots < G_{n} = G$ such that $G_{i}/G_{i-1}$ is isomorphic to a cyclic group of order $p_{i}$, $i = 1,2,\dots,n$. Associated with this series, there exists a long system of generators consisting $n$ elements $x_{1}, x_{2}, \dots , x_{n}$ (say), such that $G_{i} = \langle x_{1}, x_{2}, \dots , x_{i} \rangle$, $i = 1,2,\dots,n$. In terms of this system of generators and conjugacy class sum of $x_{i}$ in $G_{i}$, $i = 1,2, \dots, n$, we presented an algorithm for constructing the irreducible matrix representations of $G$ over $\mathbb{C}$ within the group algebra $\mathbb{C}[G]$. This algorithmic construction needs the knowledge of the primitive central idempotents, a well defined set of primitive (not necessarily central) idempotents and the diagonal subalgebra of $\mathbb{C}[G]$. 

In Chapter $7$, {'\it {Representations of finite abelian groups'}}, we gave an algorithm for constructing the irreducible matrix representations of a finite abelian group $G$ over a field $F$ of characteristic $0$ or prime to order of the group. Since every irreducible $F$-representation of $G$ factors through a faithful irreducible representation of a cyclic quotient, so it is sufficient to construct the faithful irreducible matrix representations of its cyclic quotients, which depends on arithmetic of $F$. By using the same philosophy we gave a systematic way of computing the primitive central idempotents of $F[G]$.
Besides that using a long presentation, we gave character-free expressions of the primitive central idempotents of a rational group algebra of finite abelian group and its Wedderburn decomposition. 

In Chapter $8$, {'\it{Examples'}}, we illustrate our algorithm and the classical approach based on Clifford theory for constructing the irreducible matrix representations of finite solvable groups over $\mathbb{C}$ with various examples.
\section{Future work}
\subsection{Representations of finite solvable groups over non- algebraically closed fields}
To my knowledge, the classification of representations over non-algebraically fields of finite solvable groups is not available in the literature. In this case, the arithmetic of the field will play significant role. First, I would like to consider the irreducible $\mathbb{Q}$-representations of finite solvable groups, then we consider the general case. 

\subsection{Primitive central idempotents of finite solvable groups over non-algebraically closed fields}
Computation of the primitive central idempotents of the semisimple group algebra of a finite group is an interesting problem. We have seen that there is a nice formula (see \cite{Yam80}) for computing the primitive central idempotents of a semisimple group algebra in terms of irreducible characters. In view of computational difficulties of this method, the following problem has been raised in recent years.

{\bf{Problem:}} Is there a character-free method for computing the primitive central idempotents of a semisimple group algebra $F[G]$.

This problem has been solved for certain classes of groups and for some specific fields. The primitive central idempotents in the group algebra of a finite nilpotent group over a field $F$ of characteristic $0$ or  prime to $|G|$ has been determined in \cite{al86}. This generates a natural interest in the following problem.

{\bf{Problem:}} Let $G$ be a finite solvable group and $F$ be a field of characteristic $0$ or prime to $|G|$. Compute the primitive central idempotents of ${F}[G]$.


\end{document}